\documentclass[a4paper, 14pt, reqno]{amsart}

\usepackage{amsmath,amssymb}
\usepackage{amsthm}
\usepackage[abbrev]{amsrefs}
\usepackage[top=20truemm,bottom=20truemm,left=30truemm,right=30truemm]{geometry}
\usepackage{comment}
\usepackage{empheq}
\usepackage{color}

\numberwithin{equation}{section}
\theoremstyle{plain}
\newtheorem{theorem}{Theorem}[section]
\newtheorem{lemma}[theorem]{Lemma}
\newtheorem{corollary}[theorem]{Corollary}
\newtheorem{proposition}[theorem]{Proposition}
\theoremstyle{definition}

\newtheorem{remark}[theorem]{Remark}

\providecommand{\abs}[1]{|#1|}

\newcommand{\Convection}[2]{#1 \cdot \nabla #2}
\newcommand{\Complex}{ \mathbb{C} }
\newcommand{\bracket}[1]{ ( #1 ) }
\newcommand{\dt}{ \partial_t }
\newcommand{\Div}{ \mathrm{div} \,  }
\newcommand{\Dive}{ \mathrm{div}_{\epsilon} \,  }
\newcommand{\ep}[1]{ {#1}_{\epsilon} }
\newcommand{\euclid}[1]{\mathbb{R}^{#1}}

\newcommand{\Fourier}{ \mathcal{F}}
\newcommand{\FourierInverse}{ \mathcal{F}^{ - 1 } }

\newcommand{\Lebesgue}[1]{ L^{ #1 } }

\newcommand{\ImaginaryPart}{ \mathrm{Im} }

\newcommand{\Inverse}[1]{ { \left( #1 \right)}^{-1} }
\newcommand{\Integer}{\mathbb{Z}}

\newcommand{\In}[1]{ \quad \mathrm{in} \quad #1 }
\newcommand{\LongBracket}[1]{ \left( #1 \right) }
\newcommand{\LongNorm}[2]{ { \left| \left| #1 \right| \right| }_{#2} }
\newcommand{\LongSet}[2]{\left \{ #1 : #2 \right \}}
\newcommand{\MRspace}{W^{1,p} (0, T ;L^q(\Omega)) \cap L^p (0, T; W^{2,q}(\Omega))}
\newcommand{\Napier}{ e }
\providecommand{\nrm}[2]{\lVert#1\rVert_{#2}}
\newcommand{\On}[1]{ \quad \mathrm{on} \quad #1 }
\newcommand{\ON}[1]{\quad \mathrm{on} \quad #1}
\newcommand{\Prime}[1]{ {#1}^{\prime} }
\newcommand{\Real}{ \mathbb{R} }
\newcommand{\RealPart}{ \mathrm{Re} }

\newcommand{\Set}[2]{\left \{#1 \, ; \, #2 \right \}}

\newcommand{\sLongBracket}[1]{ \left[ #1 \right] }
\newcommand{\SpaceTime}{ \Omega \times (0, \infty) }

\newcommand{\TimeInt}[4]{\int_{#1}^{#2} #3 \  d #4}

\newcommand{\Torus}{ \mathbb{T} }

\newcommand{\AmruComment}[1]{\textcolor{black}{#1}}
\newcommand{\FurukawaComment}[1]{\textcolor{black}{#1}}
\newcommand{\FurukawaCommentSecond}[1]{\textcolor{black}{#1}}
\newcommand{\FurukawaCommentThird}[1]{\textcolor{black}{#1}}
\newcommand{\FurukawaCommentFourth}[1]{\textcolor{black}{#1}}
\newcommand{\FurukawaCommentFifth}[1]{\textcolor{black}{#1}}
\newcommand{\FurukawaCommentSixth}[1]{\textcolor{black}{#1}}
\newcommand{\FurukawaCommentSeventh}[1]{\textcolor{black}{#1}}
\newcommand{\PGComment}[1]{\textcolor{black}{#1}}

\title{The Hydrostatic Approximation for the Primitive Equations by the Scaled Navier-Stokes Equations under the \FurukawaCommentSecond{No-Slip} Boundary Condition}
\author{Ken Furukawa, Yoshikazu Giga, Takahito Kashiwabara}

\keywords{ \FurukawaCommentThird{
	The first author was partly supported by the Program for Leading Graduate Schools, Leading Graduate Course for Frontiers of Mathematical Sciences and Physics,  Japan Society for the Promotion of Science (JSPS).
	The second author was partly supported by JSPS Grant-in-Aid for Scientific Research (Kiban) S (No. 26220702), A (No. \FurukawaCommentSeventh{17H}01091), A (No. 19H00639), B (No. 16H03948) and Challenging Pioneering Research (Kaitaku) (No. 18H05323).
	\FurukawaCommentSixth{The third author was partly supported by JSPS Grant-in-Aid for Young Scientists B (No. 17K14230)
and by Grant for The University of Tokyo Excellent Young Researchers.} \\ \\ 
	This work was partly supported by the DFG International Research Training Group IRTG 1529 and the JSPS Japanese- German Graduate Externship on Mathematical Fluid Dynamics.
	}}

\begin{document}
	\maketitle
	
	\thanks{
		\begin{center}
			{ \it \FurukawaCommentThird{Dedicated to Professor Matthias Hieber on the occasion of his 60th birthday}}
		\end{center}
	}
	
	\begin{abstract}
		\FurukawaCommentSeventh{In this paper we justify the hydrostatic approximation of the primitive equations in the maximal $L^p$-$L^q$-setting in the three-dimensional layer domain $\Omega = \Torus^2 \times (-1, 1)$ under the no-slip (Dirichlet) boundary condition in any time interval $(0, T)$ for $T>0$.
		We show that the solution to the scaled Navier-Stokes equations with Besov initial data} \FurukawaCommentSecond{$u_0 \in B^{s}_{q,p}(\Omega)$ for $s > 2 - 2/p + 1/ q$} converges to the solution to the primitive equations with the same initial data in $\mathbb{E}_1 (T) = W^{1, p}(0, T ; L^q (\Omega)) \cap L^p(0, T ; W^{2, q} (\Omega)) $ with order $O(\epsilon)$ where $(p,q) \in (1,\infty)^2$ satisfies $ \frac{1}{p} \leq \min \bracket{ 1 - 1/q, 3/2 - 2/q }$. 
		The global well-posedness of the scaled Navier-Stokes equations in $\mathbb{E}_1 (T)$ is also proved for sufficiently small $\epsilon>0$.	Note that $T = \infty$ is included.
	\end{abstract} 
	
	
\section{Introduction}

	\FurukawaCommentThird{We consider the primitive equations of the form}
		\begin{align*}
			(PE) \left \{
				\begin{array}{rll}
					\dt u - \Delta u + u \cdot \nabla v + \nabla_H \pi \quad
					 = 
					& 0 
					& \In{ \SpaceTime }, \\
				 	\partial_z \pi \quad
					 = 
				 	& 0 
				 	& \In{\SpaceTime}, \\
					\Div u \quad
					 = 
					&0 
					& \In{\SpaceTime}, \\
					u \quad
					= 
					& 0 
					& \On \partial \Omega	\times (0 , \infty) \FurukawaCommentThird{,}
				\end{array} 
			\right.
		\end{align*}
	\FurukawaCommentThird{where $u = (v, w) \in \euclid{2} \times \Real$ and $\pi$ are the unknown velocity field and pressure field,} respectively, $\nabla_H = \bracket{\partial_x, \partial_y}^T$, and $\Omega = \Torus^2 \times (-1, 1)$ \FurukawaCommentSixth{for $\Torus = \Real / 2 \pi \Integer$}. 
	By divergence-free condition $w$ is given by the formula
	\begin{align*}
		w (x^\prime, x_3, t) 
		= - \int_{-1}^{x_3}
			\mathrm{div}_H \, v (x^\prime, \FurukawaCommentSecond{\zeta}, t) d \zeta 
		d \zeta
		= \int_{x_3}^{1}
			\mathrm{div}_H \, v (\FurukawaCommentSecond{x^\prime , \zeta}, t) 
		\FurukawaCommentSecond{d \zeta} \FurukawaCommentThird{;}
	\end{align*}
	\FurukawaCommentThird{here we invoked physically reasonable condition $w (\cdot , \cdot , \pm 1 , \cdot) = 0$}.
	The primitive \FurukawaCommentSixth{equations are} fundamental model for geographic flow. 
	Existence of \PGComment{the global} weak solution to the primitive equations \PGComment{on the sphere} with $L^2$-initial data was proved by Lions, Temam and Wang \cite{LionsTemamWangShou1992}. 
	Local-in-time well-posedness was proved by \FurukawaComment{Guill\'{e}n-Gonz\'{a}lez, Masmoudi and Rodr\'{i}guez-Bellido} \cite{GuillenMasmoudiRodriguez2001}. 
	Although global well-posedness of the 3-dimensional Navier-Stokes \FurukawaCommentFifth{equations are} the well-known open problem, for the primitive \FurukawaCommentFourth{equations,} this problem \FurukawaCommentFourth{has} been solved by Cao and Titi \cite{CaoTiti2007}.
	Hieber and Kashiwabara \cite{HieberKashiwabara2016} extended this result to prove global well-posedness for the \FurukawaCommentSixth{primitive equations} in $L^p$-settings.
	\FurukawaComment{
		In these papers boundary conditions are imposed no-slip (Dirichlet) on the bottom and slip (Neumann) on the top.}
	Recently, \FurukawaCommentSecond{the second and last authors together with} Gries, Hieber and Hussein \cite{GigaGriesHieberHusseinKashiwabara2017_analiticity} obtained global-in-time well-posedness in the maximal regularity spaces (mixed Lebesgue-Sobolev \FurukawaCommentFifth{spaces}) $ W^{1,p} \bracket{ 0, T ; L^q \bracket{\Omega} } \cap L^p \bracket{ 0, T ; W^{2,q} \bracket{\Omega} } $ for $T>0$ \FurukawaCommentFifth{and appropriate $1 < p,q < \infty$} under various boundary conditions.
	
	\FurukawaCommentSeventh{Our aim} in this paper is to give a rigorous justification of the derivation of the primitive equations under the Dirichlet boundary condition. 
	We \PGComment{begin by explaining its derivation}. 
	\FurukawaCommentThird{Let us consider the anisotropic viscous Navier-Stokes equations in a thin domain of the form}
		\begin{align*} 
			(ANS) \left \{
				\begin{array}{rcll}
					\dt u - ( \Delta_H + \epsilon^2 \partial_z^2 ) u + u \cdot \nabla u + \nabla \pi 
					& = 
					& 0 
					& \In{ \Omega_{\epsilon} \times (0, \infty) }, \\
					\Div u 
					& = 
					& 0 
					& \In{ \Omega_{\epsilon} \times (0, \infty) },
				\end{array}
			\right.
		\end{align*}
	 where $\Omega_{\epsilon} = (- \epsilon, \epsilon) \times \mathbb{T}^2$. 
	 If $\epsilon = 1$, (ANS) is the usual Navier-Stokes equations. 
	\FurukawaCommentThird{
	The equations (ANS) \FurukawaCommentFifth{are} considered as a good model to describe motion of incompressible viscous fluid filled in a thin domain.
	}
	 Actually, if we put the Reynolds number $1$, since length and velocity are of $\epsilon$-order, apparent viscosity for vertical direction must be of $\epsilon^2$-order  \AmruComment{from the Reynolds number point of view}. 
	 The primitive equations \FurukawaCommentThird{are} formally derived from \FurukawaCommentThird{(ANS)}.
	\FurukawaCommentThird{We introduce new unknowns of (ANS) by rescaling as}
		\begin{itemize}
			\item $ u_{\epsilon} := (v_{\epsilon}, w_{\epsilon}) $
			\item $ v_{\epsilon} (x, y, z, t) := v  (x, y, \FurukawaCommentFifth{\epsilon z}, t) $
			\item $ w_{\epsilon} (x, y, z, t) := w (x, y, \FurukawaCommentFifth{\epsilon z}, t) / \epsilon$
			\item $ \pi_{\epsilon} (x, y, z, t) := \pi (x, y, \FurukawaCommentFifth{\epsilon z}, t) $,
		\end{itemize}
 	\FurukawaCommentFifth{where $x, y \in \Torus$, $z \in (-1, 1)$ and $t> 0$.}
 	Then, $(u_{\epsilon}, \pi_{\epsilon})$ satisfy the scaled Navier-Stokes equations in a fixed domain
		\begin{align*} 
			(SNS) 
			\left \{
				\begin{array}{rcll}
					\dt v_{\epsilon} - \Delta v_{\epsilon} + u_{\epsilon} \cdot \nabla v_{\epsilon} + \nabla_H \pi_{\epsilon}  
					& =
					& 0
					& \In{\SpaceTime}, \\ 
					\epsilon^2 \LongBracket{ \dt w_{\epsilon} - \Delta w_{\epsilon} + u_{\epsilon} \cdot \nabla w_{\epsilon} } + \partial_z \pi_{\epsilon} 
					& = 
					& 0 
					& \In{\SpaceTime}, \\ 
					\FurukawaCommentFifth{\mathrm{div}}\, u 
					& = 
					& 0 
					& \In{\SpaceTime}.
				\end{array}
			\right.
		\end{align*}
	Taking formally $\epsilon \rightarrow 0$ for the above equations, we get the primitive equations.
	
	\PGComment{
		The Navier-Stokes equations (SNS) \FurukawaCommentFifth{are} well-studied for $\epsilon = 1$ since the work of Leray \cite{Leray1934}, where a global weak solution is constructed in $\Omega = \Real^3$.
		For a general domain see Farwig, Kozono, Sohr \cite{FarwigKozonoSohr2005}.
		\FurukawaCommentThird{A local strong solution} is constructed by Fujita and Kato \cite{FujitaKato1964} \FurukawaCommentThird{when initial data is} in $H^{1/2}$.
		It is extended to various domains in various function spaces; see e.g. Ladyzenskaya \cite{Ladyzhenskaya1969}, Kato \cite{Kato1984}, Giga and Miyakawa \cite{GigaMiyakawa1989} for early development.
		\FurukawaCommentSixth{The reader refers} to a book of Lemari\'{e}-Rieusset \cite{Lemarie-Rieusset2016} and \FurukawaCommentFifth{review articles} by Farwig, Kozono and Sohr \cite{FarwigKozonoSohr2018} and Gallagher \cite{Gallagher2018} for resent development.
		Many results can be extended for general $\epsilon > 0$ but it is not often written explicitly \FurukawaCommentThird{except in} a book of Chemin, Desjardins, Gallagher and Grenier \cite{CheminDesjardinsGallagherGrenier2006}.
	}
	
	Rigorous justification of the primitive equations from the scaled Navier-Stokes equations was studied by Az\'{e}rad and Guill\'{e}n \cite{AzeradGuillen2001}. 
	They obtained weak* convergence in the natural energy space $L^{\infty}( 0, T; L^2(\Omega)) \cap L^2( 0, T; H^1(\Omega)) $ for $\Omega = \Torus^2 \times (-1, 1)$ and $T>0$. 
	Recently, Li and Titi \cite{LiTiti2017} improved their result to get strong convergence \PGComment{by energy method} with the aid of regularity of the solution to the primitive equations. 
	\FurukawaCommentSecond{The authors together with }Hieber, Hussein and Wrona \cite{FurukawaGigaHieberHusseinKashiwabaraWrona2018} extended Li and Titi's result in maximal-regularity \FurukawaCommentFifth{spaces} $W^{1, p} (0, T ; L^q (\Torus^3)) \cap L^{p} (0, T ; W^{2, q} (\Torus^3))$ \FurukawaCommentSeventh{with initial trace in the Besov space $B^{2 - 2 / p}_{q,p}$} for $T > 0$ and $1 / p \leq \min (1 - 1 / q, 3 / 2 - 2 / q)$ \FurukawaCommentThird{by an operator} theoretic approach.
	The case of $p = q = 2$ is corresponding to Li and Titi's result.
	 \AmruComment{
	 	Note the case of the torus corresponding to Neumann boundary conditions on top and bottom part, moreover, the work of Az\'{e}rad and Guill\'{e}n treats mixed boundary conditions with Dirichlet boundary conditions on the bottom, while Li and Titi deal with Neumann boundary conditions only.
	 }
	As we already mentioned, the \FurukawaCommentFifth{primitive equations are} a model for geographic flow. 
	Although it is more physically natural to consider the case of Dirichlet-Neumann and Dirichlet boundary conditions, there was no result of justification of derivation \FurukawaComment{in a strong topology} to the \FurukawaCommentFifth{primitive equations} from the Navier-Stokes equations.
	
	\AmruComment{
		Let
		\begin{align*}
			\mathbb{E}_1 (T) 
			&= \{ u \in \MRspace \, ;  \, \mathrm{div} \, u = 0, \, u |_{x = \pm 1} = 0 \}, \\
			 \mathbb{E}_0 (T) 
			 & = \{ u \in L^p (0, T ; L^q (\Omega)) \, ;  \, \mathrm{div} u = 0, \, u |_{x = \pm 1} = 0 \}, \\
			  \mathbb{E}_1^\pi (T) 
			 & = \{ \pi \in L^p (0, T ; W^{1, q} (\Omega)) \, ;  \, \int_\Omega \pi \, dx = 0 \},
		\end{align*}
		and 
		\begin{align*}
			 X_{\gamma}
			 & = \{ u \in B_{q, p}^{2 - 2 / p} \FurukawaCommentFourth{(\Omega)} \, ;  \, \mathrm{div} u = 0, \,  u |_{x = \pm 1} = 0 \}
		\end{align*}
	be the \FurukawaCommentFifth{initial} trace space of $\mathbb{E}_1 (T)$\FurukawaCommentSeventh{, where $B^s_{q,p}$ denotes the $L^q$-Besov space of order $s$}.
	In this paper, we frequently use $\Vert \cdot \Vert_{\mathbb{E}_0 (T)}$ as the norm of $L^p (0, T ; L^q (\Omega))$ \FurukawaCommentFifth{and $\Vert \cdot \Vert_{\mathbb{E}_1 (T)}$ as the norm of $\MRspace$} \FurukawaCommentThird{to simplify the notation}.
	}
	Let us seek the solution $U_\epsilon = (V_\epsilon, W_\epsilon)$ to
		\begin{equation}  \label{eq:A0502}
			\left \{
				\begin{array}{rcll}
					\dt \ep{V} - \Delta \ep{V} + \nabla_H \ep{P} 
					& = 
					& F_H 
					&\In{ \Omega \times (0, T) }, \\ 
					\dt \bracket{ \epsilon \ep{W} } - \Delta \bracket{ \epsilon \ep{W} } + \frac{\partial_z}{\epsilon} \ep{P} 
					& = 
					& \epsilon F_z + \epsilon F 
					&\In{ \Omega \times (0, T) }, \\
					\Div \ep{U} 
					& = 
					& 0 
					&\In{ \Omega \times (0, T) }, \\
					\ep{U} 
					& = 
					& 0 
					&\ON{\partial \Omega \times (0, T)},\\
					\ep{U}(0) 
					& = 
					& 0 
					&\In{\Omega},
				\end{array}
			\right .
		\end{equation}
	where
		\begin{itemize}
			\item $F_H = - \LongBracket{ \Convection{ \ep{U} }{ \ep{V} } + \Convection{u}{ \ep{V} } + \Convection{ \ep{U} }{v} } $
		
			\item $ F_z = - \LongBracket{ \Convection{ \ep{U} }{ \ep{W} } + \Convection{u}{ \ep{W} } + \Convection{ \ep{U} }{w} }$
		
			\item $F = - \LongBracket{ \dt w - \Delta w + \Convection{u}{w} }$.
		\end{itemize}
	\PGComment{The system} (\ref{eq:A0502}) is the equation of the difference between the solution to the (PE) and (SNS).
		\begin{theorem} \label{thm:A0105}
			Let $T > 0$ \FurukawaCommentFourth{and $0< \epsilon \leq 1$}. 
			Suppose $\FurukawaCommentSeventh{(p,q) \in (1, \infty)^2}$ satisfies $ \frac{1}{p} \leq \min \bracket{ 1 - 1/q, 3/2 - 2/q }$, \FurukawaComment{$u_0 = (v_0, w_0) \in X_{\gamma}$  and $v_0 \in B^{s}_{q, p} \FurukawaCommentFourth{(\Omega)}$ for $s > 2 - 2 / p + 1 / q$}.
			Let $u \in \mathbb{E}_1 (T)$ be a solution of (PE) with initial data $u_0 \in X_\gamma$.
			Then there exists constant $C = C (p, q, \Vert u \Vert_{\mathbb{E}_1 (T)})$ and a unique solution $U_\epsilon = (V_\epsilon, W_\epsilon)$ to (\ref{eq:A0502}) such that
			\begin{align}
				\Vert
					\left(
						V_\epsilon, \epsilon W_\epsilon
					\right)
				\Vert_{\mathbb{E}_1 (T)}
				\leq \epsilon C.
			\end{align}
			Moreover, $u_\epsilon = (v_\epsilon, w_\epsilon) := (v + V_\epsilon, w + W_\epsilon)$ is the unique solution to (SNS) in $\mathbb{E}_1 (T)$.
		\end{theorem}
	This theorem implies the justification of the hydrostatic approximation.
		\begin{corollary}  \label{cor:A0105}
			Let $T > 0$ \FurukawaCommentFourth{and $0 < \epsilon \leq 1$}. 
			Suppose $\FurukawaCommentSeventh{(p,q) \in (1, \infty)^2}$ satisfies $ \frac{1}{p} \leq \min \bracket{ 1 - 1/q, 3/2 - 2/q }$, \FurukawaCommentFourth{$u_0 = (v_0, w_0) \in X_{\gamma}$  and $v_0 \in B^{s}_{q, p} (\Omega)$ for $s > 2 - 2 / p + 1 / q$}.
			Let $u$ and $ \ep{u} $ be a solution of (PE) and (SNS) in $\mathbb{E}_1\FurukawaComment{(T)}$ under the Dirichlet boundary condition with initial data $u_0$, respectively, such that
				\begin{align}
					\LongNorm{u}{\mathbb{E}_1 \bracket{T}} 
					+ \LongNorm{ \bracket{ \ep{v}, \epsilon \ep{w} } }{\mathbb{E}_1 \bracket{T}} 
					\leq C_{\FurukawaCommentFourth{0}}
				\end{align}
			for some $C_{\FurukawaCommentFourth{0}} = C_{\FurukawaCommentFourth{0}} \bracket{u_0, p, q}$. 
			Then there exists a positive $ C = C(p, q, C_0 ) $ such that
				\begin{align*}
					\LongNorm{ \LongBracket{ \ep{v} - v, \epsilon \bracket{ \ep{w} - w } } }{ \mathbb{E}_1 \bracket{T} } \leq \epsilon C.
				\end{align*}
		\end{corollary}
	
	Our strategy to show Theorem \ref{thm:A0105} is based on the estimate for $(V_\epsilon, \epsilon W_\epsilon)$.
	\FurukawaCommentFifth{It consists of two key steps:} maximal regularity result of the anisotropic Stokes operator and improved regularity result for \FurukawaCommentFourth{the} vertical component of the solution to the \FurukawaCommentFifth{primitive equations}.
	\FurukawaCommentFifth{
	We consider the non-linear term in (SNS) as an external force term $f$ and set $u_\epsilon = (v_\epsilon, \FurukawaCommentSixth{\epsilon} w_\epsilon)$ to get
	\begin{equation} \label{eq_LE}
			\left \{
				\begin{array}{rcll}	
					\dt u_{\FurukawaCommentFifth{\epsilon}} - \Delta u_{\FurukawaCommentFifth{\epsilon}} + \nabla_{\epsilon} \pi_{\FurukawaCommentFifth{\epsilon}} 
					& = 
					& f 
					& \In{\Omega \times \bracket{0, T}}, \\
					\Dive u_{\FurukawaCommentFifth{\epsilon}} 
					& = 
					& 0 
					&\In{\Omega \times \bracket{0, T}} ,\\
					u_{\FurukawaCommentFifth{\epsilon}}
					& = 
					& 0 
					&\ON{ \partial \Omega  \times (0, T) }, \\
					u_{\FurukawaCommentFifth{\epsilon}} (0) 
					&= 
					& u_0
					& \In{\Omega},
				\end{array}
			\right .
		\end{equation}
		\FurukawaCommentSixth{where $\nabla_\epsilon = (\partial_1, \partial_2 , \partial_3 / \epsilon)^T$ and $\mathrm{div}_\epsilon = \nabla_\epsilon \cdot$.}
		We define the function space $\mathbb{E}_{\epsilon, j} (T)$ for $j = 0, 1$ similarly as $\mathbb{E}_j (T)$ by replacing $\mathrm{div}$ by $\mathrm{div_\epsilon}$.
		Although the space $\mathbb{E}_{\epsilon, j} \FurukawaCommentSeventh{(T)}$ depends on $\epsilon$, the norm is just the norm in $\MRspace$, so we shall write the norm in $\mathbb{E}_{\epsilon , j} (T)$ simply by $\Vert \cdot \Vert_{\mathbb{E}_j (T)}$.
		The space $X_{\epsilon, \gamma}$ is the initial trace space of $\mathbb{E}_{\epsilon, j}(T)$ and it is almost the same as $X_\gamma$ by replacing $\mathrm{div}$ by $\mathrm{div}_\epsilon$.
		Since the norm of $X_\gamma$ is that of $B^{2 - 2 / p}_{q,p} (\Omega)$ and is independent of $\epsilon$, we shall write the norm in $X_{\epsilon, \gamma}$ simply by $\Vert \cdot \Vert_{X_{ \gamma}}$.}
	\FurukawaCommentFifth{We recall} some known results on maximal regularity of the Stokes operator, which is corresponding to the case $\epsilon = 1$. 
	Solonnikov \cite{Solonnikov1964} first proved $\Lebesgue{q}$-$\Lebesgue{q}$ maximal regularity for the Stokes operator \PGComment{by a potential-theoretic approach}. 
	\PGComment{
		\FurukawaCommentSecond{The second author} \cite{Giga1985} established a \FurukawaCommentSecond{bound} for the pure imaginary power of the Stokes operator in a bounded domain.
		This type of property will be \FurukawaCommentThird{simply called} a bounded imaginary power\FurukawaCommentFifth{, shortly} BIP.
		This BIP implies the maximal regularity $\Lebesgue{p}$-$\Lebesgue{q}$ regularity via Dore-Venni theory \cite{DoreVenni1990}.
		Indeed, \FurukawaCommentSecond{the second author} and Sohr \cite{GigaSohr1991} established a global-in-time maximal regularity in an exterior domain by estimating BIP.
	}
	Further studies on maximal regularity were done by many researchers, for instance, Dore and Venni \cite{DoreVenni1990} and Weis \cite{Weis2001}. 
	See Denk, Hieber and Pr\"{u}ss \cite{DenkHieberPruss2003} for further comprehensive research.
	\PGComment{
		In our case, we have to clarify $\epsilon$-dependence in estimates for maximal regularity, which is a key point.
		Our key maximal regularity result is
	}
	\begin{lemma} \label{lem:A1001}
		Let $1 < p, q < \infty$, $0 < \epsilon \leq 1$ and $T > 0$.
		Let \FurukawaCommentFifth{$ f \in \mathbb{E}_{\epsilon, 0} (T)$ and $u_0 \in X_{\epsilon, \gamma}$}.
		Then \FurukawaCommentFourth{there exist} constants $C = C (p,q)>0$ and $C^\prime = C^\prime (p,q) > 0$, which are independent of $\epsilon$, and $(u, \pi)$ satisfying (\ref{eq_LE}) such that
		\begin{align} \label{eq:A1001} 
			\LongNorm{ \dt u }{\mathbb{E}_0\FurukawaComment{(T)}} 
			+ \LongNorm{ \nabla^2 u }{\mathbb{E}_0\FurukawaComment{(T)}} 
			+ \LongNorm{ \ep{\nabla} \pi }{\mathbb{E}_0\FurukawaComment{(T)}}
			\leq C \LongNorm{f}{\mathbb{E}_0\FurukawaComment{(T)}} 
			+ C^\prime \LongNorm{u_0}{ X_{\gamma} }.
		\end{align}
	\end{lemma}
	\PGComment{
	\FurukawaCommentFifth{Lemma \ref{lem:A1001} follows from a maximal regularity involving the Stokes operator, which follows from a bound for the pure imaginary power by Dore-Venni theory.}
	However, we need to clarify that $C$ and $C^\prime$ can be taken independent of $\epsilon$.
	For $\epsilon = 1$, \FurukawaCommentThird{a necessary BIP estimate} for the Stokes operator has been established \FurukawaCommentThird{by Abels \cite{Abels2002}, where a resolvent decomposition similar to \cite{Giga1985} is used.}
	Unfortunately, \FurukawaCommentFourth{the} $\epsilon$-dependent case is not discussed here.}
	However, the strategy in \cite{Abels2002} works \FurukawaComment{for} our problem.
	We construct the anisotropic Stokes operator by the method in \cite{Abels2002} and show the boundedness of imaginary power.
	Note that, in our previous paper \cite{FurukawaGigaHieberHusseinKashiwabaraWrona2018}, maximal regularity of the anisotropic Stokes operator is much easier since the corresponding Stokes operator is essentially the same as the Laplace operator on $\Torus^3$.
	In the case of the Dirichlet boundary condition, the corresponding Stokes operator becomes to be much more difficult by the effect of boundaries, \FurukawaCommentFifth{which is substantially different from the case} of the periodic boundary conditions.
	\FurukawaCommentThird{The maximal regularity was proved in a layer domain for the Stokes operator under various boundary conditions by Saito \cite{Saito2015} by proving $\mathcal{R}$-boundedness of the resolvent operator when $\epsilon =1$.
	Unfortunately, it seems very difficult to check the dependence of $\epsilon$, so we do not take this approach,}
	
	\FurukawaComment{The term} $F = \partial_t w - \Delta w + u \cdot \nabla w$ appears in \FurukawaCommentFourth{the right-hand side} of (\ref{eq:A0502}).
	Thus, we need to improve the regularity of $w$ and estimate this term in $L^p (0, T ; L^q (\Omega))$.
	\begin{lemma} \label{lem:A1002}
		Let $T > 0$ and $u_0 = (v_0, w_0) \in X_\gamma$ with $w_0 = - \TimeInt{-1}{x_3}{\mathrm{div}_H \, v_0 }{\zeta}$ and \FurukawaComment{$v_0 \in B^{s}_{q, p}(\Omega)$  for $s > 2 - 2 / p + 1 / q$} and $u = (v, w)$ be the solution to (PE).
		Assume $v \in \mathbb{E}_1 (T)$.
		Then there \FurukawaCommentFourth{exists} a \FurukawaComment{constant $C> 0$} such that
		\begin{align} \label{eq:B0800}
			\Vert
				w
			\Vert_{\mathbb{E}_1 (T)}
			\leq C.
		\end{align}
	\end{lemma}
	Since $v \in \mathbb{E}_1 (T)$, which is the \FurukawaCommentFifth{horizontal} component of the solution to the primitive equations, has already proved, it follows $w (\cdot, x_3) =  - \int_{-1}^{x_3} \mathrm{div}_H v (\cdot, \zeta) d\zeta \in \FurukawaComment{W}^{1, p}( 0, T ; \FurukawaComment{W}^{-1, q} (\Omega)) \cap L^p (0, T; \FurukawaComment{W}^{1, p} (\Omega))$ .
	This derivative loss is due to the absence of the equation of time-evolution of $w$ in the \FurukawaCommentFifth{primitive equations}.
	In our previous paper \cite{FurukawaGigaHieberHusseinKashiwabaraWrona2018}, \AmruComment{which treats the periodic boundary condition}, we recover the regularity of $w$ by deriving the equation which $w$ satisfies and applying maximal regularity of the Laplace operator to the equation.
	However, in the case of the Dirichlet boundary condition, this method is not applicable directly \FurukawaCommentThird{because of the presence of the second-order derivative term} at the boundary, which vanishes in the case of periodic boundary condition.
	\FurukawaComment{
	Thus, \FurukawaCommentThird{we are \FurukawaCommentFourth{forced} to impose} additional regularity for initial data to get regularity for $v$.
	If \FurukawaCommentSeventh{$v_0 \in B^{s}_{q,p}\FurukawaCommentSeventh{(\Omega)}$} for $s > 2 - 2 /p + 1 /q$, then we obtain $v \in L^p (0 , T; W^{s + 2 / p , q} \FurukawaCommentSeventh{(\Omega)})$ and the trace of \FurukawaCommentFourth{the} second derivative belongs to $\mathbb{E}_0 (T)$.
	}Let us \FurukawaCommentFifth{explain} our strategy to show Theorem \ref{thm:A0105}.
	\PGComment{
		By Lemma \ref{lem:A1002}, our main result Theorem \ref{thm:A0105}, can be proved \FurukawaCommentFourth{the} same way as  \cite{FurukawaGigaHieberHusseinKashiwabaraWrona2018}.
		The proof we give here is slightly different from that of \cite{FurukawaGigaHieberHusseinKashiwabaraWrona2018} in the sense of the constant $C$ \FurukawaCommentFourth{in} Theorem \ref{thm:A0105}
 is clarified.
 	}
	\FurukawaCommentThird{We first show the boundedness of non-linear terms $F_H$ and $F_z$ in (\ref{eq:A0502}) in the space $\mathbb{E}_0\FurukawaComment{(T)}$.}
	\FurukawaCommentThird{We know that $F$} is also bounded \FurukawaCommentThird{in} $\mathbb{E}_0 (T)$ by Lemma \ref{lem:A1001}.
	\FurukawaCommentThird{We next} apply Lemma \ref{lem:A1002} to (\ref{eq:A0502}) to get a quadratic inequality, which leads \FurukawaCommentSixth{to} $ \LongNorm{ \bracket{ \ep{V}, \epsilon \ep{W} } }{\mathbb{E}_1 \bracket{T^{\ast}}} \leq C \epsilon $ for some short time \FurukawaComment{$T^{\ast}> 0$} and $\epsilon$-independent constant $C>0$. 
	Since $C$ depends only on $p$, $q$, $\Vert u_0 \Vert_{X_\gamma}$, $\Vert u \Vert_{\mathbb{E}_1 (T)}$ and $T$, if we take $\epsilon$ small, \FurukawaCommentThird{we are able to extend} the time to all finite time $T$ by finite step.

\FurukawaComment{
	This paper is organized as follows.
	In section \ref{section_BIP}, \FurukawaCommentFourth{the} boundedness of pure imaginary power is proved.
	The resolvent operator of the anisotropic Stokes operator is decomposed into three parts, and for each part uniform bound on $\epsilon$ is proved.
	In section \ref{section_improved_regularity_w}, improved regularity for $w$ is proved.
	In section \ref{section_proof_of_main_thoerem}, we give a proof of our main theorem by iteration.
}
	

	In this paper, $|| \cdot||_{X \rightarrow Y}$ denotes the operator norm from a Banach space $X$ to a Banach space $Y$. 
	\FurukawaCommentFifth{We denote by $C_0^\infty (\Omega)$ the set of compactly supported smooth functions in $\Omega$.
	We denote $L^q (\Omega)$ is the Lebesgue space for $1 \leq q \leq \infty$ \FurukawaCommentSixth{equipped} with the norm
	\begin{align*}
		\Vert
			f 
		\Vert_{L^q (\Omega)}
		= \left(
			\int_\Omega
				\left \vert
					f (x)
				\right \vert^q
			d x
		\right)^{1/q}.
	\end{align*}
	We use the usual modification when $q = \infty$.
	For $m \in \Integer_{\geq 0}$ and $1 \leq q \leq \infty$ we denote by $W^{m, q} (\Omega)$ the $m$-th order Sobolev space \FurukawaCommentSixth{equipped} with the norm 
	\begin{align*}
		\Vert 
			 f 
		\Vert_{W^{m, q} (\Omega)}
		= \Vert 
			\nabla^m f 
		\Vert_{L^q (\Omega)}.
	\end{align*}
	We define the fractional Sobolev spaces \FurukawaCommentSeventh{$W^{s, q} (\Omega) (= B^s_{qq} (\Omega))$ for $s \notin \Integer$} and $1 < q < \infty$ by the real interpolation $\left( W^{[s], q} (\Omega), W^{[s] + 1, q} (\Omega) \right)_{s - [s], q}$, where $[\cdot]$ denotes the Gauss symbol.
	}
	We define the Fourier transform by 
	\begin{align*}
			\Fourier f \bracket{\xi} = \TimeInt{\Real^d}{}{ \Napier^{ - i x \cdot \xi } f \bracket{x} }{x}
	\end{align*} 
	\FurukawaCommentFourth{and} 
	the Fourier inverse transform by 
	\begin{align*}
		\FourierInverse f \bracket{x} = \frac{1}{ \bracket{ 2 \pi }^d} \TimeInt{\Real^d}{}{ \Napier^{ i x \cdot \xi } f \bracket{\xi} }{\xi}.
	\end{align*}
	The Fourier transform on the \FurukawaComment{$d$-dimensional} torus $\Torus^d$ and its inverse transform are \FurukawaCommentSecond{defined by $[\Fourier_d f ] (n) = \int_{\Torus^d} e^{ - i x \cdot n} f (x) dx$ and $[\Fourier_d^{-1} g ]  (x) = \frac{1}{(2 \pi )^d} \sum_{n} g_n e^{i n \cdot x}$}, respectively.
	\FurukawaCommentFifth{We denote by $\Fourier_{\Prime{x}}$ the partial Fourier transform with respect to $\Prime{x} \in \Real^2$ and by $\FourierInverse_{\Prime{\xi}}$ the partial Fourier inverse transform with respect to $\Prime{\xi}$.}
	\FurukawaCommentFifth{We denote by $\Fourier_{d,\Prime{x}}$ the partial Fourier transform \FurukawaCommentSixth{with} respect to $\Prime{x} \in \Torus^2$ and by $\FourierInverse_{d, \Prime{n}}$ the partial Fourier inverse transform with respect to $\Prime{n} \in \Integer^2$.}
	Define $\Sigma_{\theta} := \LongSet{ \lambda \in \Complex}{ \abs{ \arg \lambda } < \pi - \theta }$. 
	For a Fourier multiplier operator $\FourierInverse_{\xi} m \bracket{\xi} \Fourier_x$ in $\Real^3$, we denote by $[m]_{\mathcal{M}}$ the Mikhlin constant.
	$\FourierInverse_{\Prime{\xi}} m \bracket{\xi} \Fourier_{\FurukawaComment{x^\prime}}$ is a Fourier multiplier operator in $\Real^2$ with Mikhlin constant $[m]_{\Prime{\mathcal{M}}}$.
	For $0 < \epsilon \leq 1$, $\ep{\Delta} = \partial_1^2 + \partial_2^2 + \partial_3^2 / \epsilon^2$ denotes the anisotropic Laplace operator. 
	\FurukawaCommentFifth{We denote by $E_0$ the zero-extension operator with respect to the vertical variable from $(-1 , 1)$ to $\Real$.
	We denote by $R_0$ the restriction operator with respect to the vertical variable from $\Real$ to $(-1, 1)$.}
	For \FurukawaCommentFourth{an} integrable function $f$ defined on $\Omega$, we write its vertical average by $\overline{f} = \frac{1}{2} \TimeInt{-1}{1}{ f (\cdot, \cdot , \zeta)}{\zeta}$.
	 	
	
	\section{\FurukawaCommentFifth{A uniform bound for pure imaginary power of the anisotropic Stokes operator and its maximal regularity}} \label{section_BIP}
	
	 \PGComment{In this section, we \FurukawaCommentFifth{first} establish a uniform bound  independent of $\epsilon$ for the pure imaginary power to the anisotropic Stokes operator along with \cite{Abels2002}.} 
	\FurukawaCommentFifth{Then we shall give the proof of Lemma \ref{lem:A1001}.}
	\subsection{Boundedness of Fourier multipliers}
	
	Although the case of infinite \FurukawaCommentSeventh{the} layer $\Real^2 \times (-1. 1)$ is considered in \cite{Abels2002}, his method also works in \FurukawaCommentFifth{the} case of \FurukawaCommentFifth{the} periodic layer $\Omega = \Torus^2 \times (- 1, 1)$ thanks to Fourier multiplier theorem on the torus, e.g. Proposition 4.5 in \cite{HeckKimKozono2009} \FurukawaCommentFifth{and Section 4 of Grafakos's book \cite{Grafakos2008}.}
	\begin{proposition}[\cite{HeckKimKozono2009}] \label{prop_Fourier_Multiplier_Theorem_Discrete}
		Let $1 < p < \infty$ and $m \in C^{\FurukawaCommentSeventh{d}} (\Real^d \setminus \{ 0 \})$ satisfies the Mikhlin condition:
			\begin{align}
				[m]_{\PGComment{\mathcal{M}}}
				:= \sup_{\alpha \in \{ 0, 1 \}^d} \sup_{\xi \in \Real^d \setminus \{ 0 \} }	\left|
					\xi^\alpha
					\partial_{\xi}^\alpha m (\xi)
				\right|
				< \infty.
			\end{align}
		Let $a_k = m (k)$ for $k \in \mathbb{Z}^d \setminus \{ 0 \}$ and $a_0 \in \Complex$.
		\PGComment{
		For $f (x) = \sum_{n \in \mathbb{Z}^d} \widehat{f}_n e^{i n \cdot x} \in L^{\FurukawaCommentFifth{q}} (\Torus^d)$ \FurukawaCommentFifth{and a sequence $a = \{ a_n \}_{n \in \Integer^d}$}, we set  the Fourier multiplier operator of discrete type by
		\begin{align} 
			[T f] (x) 
			\FurukawaCommentFifth{
			:= \FourierInverse_{d} a \Fourier_d f
			}
			= \sum_{n \in \mathbb{Z}^d} a_n  \hat{f}_n e^{i n \cdot x}.
		\end{align}
		Then there exits a constant $C = C(p,d) > 0$ such that
		}
		\begin{align}
			\Vert
				T f
			\Vert_{L^{\FurukawaCommentFifth{q}} (\Torus^d)}
			\leq C \max ( [m]_{\PGComment{\mathcal{M}}}, a_0) \Vert
				f
			\Vert_{L^{\FurukawaCommentFifth{q}} (\Torus^d)}.
		\end{align}
	\end{proposition}

	Let us consider the resolvent problem to (\ref{eq_LE}) ;
		\begin{equation} \label{eq_RLE}
			\left \{
				\begin{array}{rrl}
					\lambda u - \Delta u + \ep{\nabla} \pi 
					& = f 
					&\In \Omega, \\
					\Dive u 
					&= 0 
					&\In \Omega , \\
					u 
					&= 0 
					&\ON \partial \Omega,
				\end{array}	
			\right.
		\end{equation}
	for $\lambda \in \Sigma_\theta$ ($0 < \theta < \pi / 2$) and $f \in \Lebesgue{q} \bracket{\Omega}$.
	Let 
	\begin{align*}
		H_\epsilon : L^{\FurukawaCommentFifth{q}}  (\Omega) \rightarrow L^{\FurukawaCommentFifth{q}}_{\sigma, \epsilon} (\Omega) 
		= \Set{ u \in L^{\FurukawaCommentFifth{q}} (\Omega)}{ \mathrm{div}_\epsilon u = 0, u|_{x_3 = \pm 1 } = 0}, 
		\, (1 < p <\infty) 
	\end{align*}
	 be the anisotropic Helmholtz projection on $\Omega$, its $L^{\FurukawaCommentFifth{q}}$-boundedness is proved later.
	Let $A_\epsilon = H_\epsilon (- \Delta)$ be the Stokes operator with the domain $D (A_\epsilon) = L^{\FurukawaCommentFifth{q}}_{\sigma, \epsilon} (\Omega) \cap W^{2, \FurukawaCommentFifth{q}} (\Omega)$.
	 \FurukawaCommentFourth{For $0 < a < 1 / 2$ and $ - a < \RealPart \, z < 0$, the fractional power of $A_{\epsilon}$ is defined via the Dunford calculus}
		\begin{align*}
			\ep{A}^z
			 = \frac{1}{ 2 \pi i } \TimeInt{\Gamma_{\theta}}{}{ \bracket{ - \lambda }^z \Inverse{ \lambda + \ep{A} } }{\lambda},
		\end{align*}
	where $ 0 < \theta < \pi /2 $ and $\Gamma_{\epsilon} =  \Real \Napier^{ i \bracket{ - \pi + \theta } } \cup \Real \Napier^{ i \bracket{ \pi - \theta } } $. 
	Our aim in this section is to prove 
		\begin{lemma} \label{lem:C02}
			Let $1 < q < \infty$, $\FurukawaCommentFifth{0 < \epsilon \leq 1}$, $ 0 < a < 1/2$, $z \in \Complex$ satisfying $ - a < \RealPart \, z < 0 $ and $ 0 < \theta < \pi /2 $. 
			Then there  \FurukawaCommentFourth{exists} a constant $C = C \bracket{ q, a, \theta }$ such that
				\begin{align}
					\LongNorm{ \ep{A}^z }{ \Lebesgue{q} \bracket{\Omega} \rightarrow \Lebesgue{q} \bracket{\Omega} } 
					\leq C \Napier^{ \theta \abs{ \ImaginaryPart z} }.
				\end{align}
		\end{lemma}
	\PGComment{Once} the above lemma is proved, then we obtain \FurukawaCommentFourth{the maximal regularity} of the anisotropic Stokes operator via the formula
		\begin{align}
			\left(
				\frac{
					d}{
					d t
					} 
				+ A_\epsilon 
			\right)^{-1} 
			= \int_{c + i \infty}^{ c - i \infty}
				\frac{				
					(d/dt)^z A_\epsilon^{1 - z}
				}{
					\sin \pi z
				}
			d z
		\end{align}
	for $0 < c < 1$ and the Dore-Venni theory \cite{DoreVenni1990}.
	
	To show Lemma \ref{lem:C02}, we decompose the solution $(u, \pi)$ \FurukawaCommentFifth{to} (\ref{eq_RLE}) into three parts;
		\begin{align} \label{def_v_j_pi_j}
			u 
			& = R_0 v_1 - v_2 + \ep{\nabla} \pi_3, \\
			\ep{\nabla} \pi 
			& = \ep{\nabla} \pi_1 + \ep{\nabla} \pi_2,
		\end{align}
	where $v_j$ and $\pi_j$ are solutions to	
		\begin{align*}
			(I) \left \{
				\begin{array}{rcll}
					\lambda v_1 - \Delta v_1 + \ep{\nabla} \pi_1 
					&= 
					& E_0 f 
					& \In \Torus^2 \times \Real,\\
					\Dive v_1 
					&= 
					& 0 
					& \In \Torus^2 \times \Real,
				\end{array}
			\right.
		\end{align*}
		
		\begin{align*}
			(II)	\left \{
				\begin{array}{rcll}
					\lambda v_2 - \Delta v_2 + \ep{\nabla} \pi_2 
					&= 
					& 0 
					& \In \Omega,\\
					\Dive v_2 
					&= 
					& 0 
					& \In \Omega, \\
					v_2 
					& = 
					& \gamma v_1 - \bracket{ \gamma v_1 \cdot \nu } \nu 
					& \On \partial \Omega,
				\end{array}
			\right.
		\end{align*}
	and
		\begin{align*}
			(III) \left \{
				\begin{array}{rcll}
					\ep{\Delta} \pi_3 
					& = 
					& 0 
					& \In \Omega,\\
					\ep{\nabla} \pi_3 \cdot \nu  
					& = 
					& \bracket{ \gamma v_1 \cdot \nu } \nu 
					& \On \partial \Omega,
				\end{array}
			\right.
		\end{align*}	
	respectively, where $\gamma = \gamma_{\pm}$ is the trace operator to the upper and lower boundary, respectively, and $\nu$ is the unit outer normal. To show Lemma \ref{lem:C02}, we need to obtain
		\begin{align*}
			&\LongNorm{ \frac{1}{2 \pi i}
				\TimeInt{\Gamma}{}{ 
					\bracket{- \lambda}^z R_0 v_1
				}{\lambda}
			}{\Lebesgue{q} \bracket{\Omega}}
			+
			\LongNorm{  \frac{1}{2 \pi i}
				\TimeInt{\Gamma}{}{ 
					\bracket{- \lambda}^z v_2
				}{\lambda}
			}{\Lebesgue{q} \bracket{\Omega}} \\
			& +
			\LongNorm{ \frac{1}{2 \pi i}
				\TimeInt{\Gamma}{}{
					\bracket{- \lambda}^z \ep{\nabla} \pi_3
				}{\lambda}
			}{\Lebesgue{q} \bracket{\Omega}}
			 \leq C \Napier^{ \theta \abs{ \ImaginaryPart \, z } }
			\LongNorm{f}{\Lebesgue{q} \bracket{\Omega} }
		\end{align*}
	  \FurukawaCommentFourth{for some Constant $C >0$, which is independent of $\epsilon$.}
	  
	  \FurukawaCommentFifth{
	  \begin{remark} \label{rem_average_free}
		For $f \in L^q (\Omega)$ and $\tilde{f} := \int_{\Torus^2} f d x^\prime$, we can solve the resolvent problem (\ref{eq_RLE}) with external force $\tilde{f}$ to get $\FurukawaCommentFifth{\Prime{u}} = \left( (\lambda - \partial_3^2)^{-1} \tilde{f}_H , 0 \right)^T$ and $\FurukawaCommentFifth{\Prime{\pi}} = \epsilon \int_{-1}^{x_3} \tilde{f}_3 d \zeta / \mathbb{R}$, where $\tilde{f}_H$ is \FurukawaCommentFourth{the} horizontal component of $\tilde{f}$ and $ / \mathbb{R}$ means average-free.
		Since $- \partial_3^2$ has BIP and the resolvent operator is linear, by taking the difference between the solution to (\ref{eq_RLE}) and \FurukawaCommentFifth{$(\Prime{u}, \Prime{\pi})$}, we can always assume without loss of generality that $f$ is \FurukawaCommentFifth{horizontal} average-free.
	\end{remark}
	}
	\FurukawaCommentFifth{We define the space of horizontally average-free $L^q$-vector fields by
	\begin{align*}
		L^q_{\FurukawaCommentSixth{\mathrm{af}}} (\Omega) 
		: = \LongSet{
			f \in L^q (\Omega)
		}{
		\tilde{f} = 0}.
	\end{align*}
	Similarly we define
	\begin{align*}
		W^{s, q}_{\FurukawaCommentSixth{\mathrm{af}}} (\Omega) 
		: = \LongSet{
			f \in W^{s, q} (\Omega)
		}{
		\tilde{f} = 0}.
	\end{align*}
	}	
	Throughout this section we frequently use partial Fourier transform to construct solutions and estimate these partial Fourier multipliers.
	\FurukawaComment{
		\begin{proposition}[\cite{Abels2002}] \label{prop:C03}
			Let $1 < q < \infty$ and $a, \, b \in \{-1 , 1 \}$.
			Set a integral operator $M$ by 
				\begin{align*}
					M f (x^\prime , x_3) 
					=  \TimeInt{-1}{1}{
						\frac{
							f (x^\prime , \zeta)
						}{
							\abs{x_3 - a} 
							+ \abs{\zeta - b}
						} 
					}{\zeta}
				\end{align*}
			\FurukawaCommentSecond{for $f \in L^q (\Omega)$.}
			Then there exists a constant $C>0$ such that
				\begin{align*}
					\LongNorm{ M f }{\Lebesgue{q}} 
					\leq C \LongNorm{f}{\Lebesgue{q}}.
				\end{align*}
		\end{proposition}
		\begin{proof}
			See Lemma 3.3 in \cite{Abels2002}.
		\end{proof}
	}
%
%
	Rescaled $L^{\FurukawaCommentFifth{q}}$-Fourier multipliers are also bounded $L^{\FurukawaCommentFifth{q}}$ multiplier by the direct consequence of the Mikhlin theorem.
	\begin{proposition} \label{prop_scaled_multiplier}
		\FurukawaCommentFifth{Let $1 < q  < \infty$ and $0 < \epsilon \leq 1$.}
		Let $m \in \FurukawaCommentSeventh{C^d} (\Real^d \setminus \{ 0 \})$ be a $L^{\FurukawaCommentFifth{q}}$-Fourier multiplier with the Mikhlin constant $[m]_\mathcal{M} \leq C$ for some $C>0$.
		Then rescaled one $m_\epsilon (\xi) := m (\epsilon \xi)$ is also bounded from $L^{\FurukawaCommentFifth{q}}$ into itself such that
		\begin{align*}
			[ m_\epsilon ]_{\PGComment{\mathcal{M}}} \leq C.		
		\end{align*}
	\end{proposition} 
	The above proposition is frequently used in this section to get $\epsilon$-independent estimate for scaled multipliers.
	We show boundedness of some Fourier multiplier operators in advance.
	We set 
		\begin{align*}
			s_{\lambda} 
			= \bracket{ \lambda + \abs{\xi^{\prime}}^2 }^{1/2}
		\end{align*}
	for $\Prime{\xi} \in \mathbb{R}^2$.
	\FurukawaCommentFourth{In this paper we use $s_\lambda$ to denote $(\lambda + \abs{n^\prime}^2)^{1/2}$ for $n^\prime \in \mathbb{Z}^2$ to simplify notation.}
		\begin{proposition} \label{prop:C05}
			\hspace{1mm}
			\begin{itemize}
				\item Let $0 < \theta < \pi /2$, $\lambda \in \Sigma_{\theta}$, $t > 0$ and $\alpha$ be a positive integer.
				Then there exist constants $c > 0$ and $C>0$ such that
 					\begin{align} \label{eq:C0501}
 						\sLongBracket{ 
						\abs{\xi^{\prime}}^{\alpha} \Napier^{- t s_{\lambda} } 
					}_{
						\mathcal{M}^{\prime}
					} 
					\leq C \frac{ 
						\Napier^{{ - c t \abs{ \lambda } }^{1/2} } 
					}{ 
						t^{\alpha} 
					}, \quad
						\sLongBracket{ \frac{ \Napier^{ - s_{\lambda} } }{ s_{\lambda} } }_{\mathcal{M}^{\prime}} 
						\leq C \abs{ \lambda }^{ - 1/2 } \Napier^{ - c \abs{\lambda}^{1/2} }.
					\end{align}
				
				
				\item  Let $ -1 \leq x_3 \leq 1 $. 
				Then there exists a constant $C>0$ which is independent of $\epsilon$, such that
					\begin{align} \label{eq:C0504}
						\sLongBracket{
							\frac{ \sinh \bracket{ \epsilon \abs{\xi^{\prime}} x_3 } }{ \sinh \bracket{ \epsilon \abs{\xi^{\prime}} } } 
							\frac{ \epsilon \abs{\xi^{\prime}} }{ 1 + \epsilon  \abs{\xi^{\prime}} } 
						}_{\mathcal{M}^{\prime}} \leq C, \quad
						\sLongBracket{ 
							\frac{ \cosh \bracket{ \epsilon \abs{\xi^{\prime}} x_3 } }{ \sinh \bracket{ \epsilon \abs{\xi^{\prime}} } } 
							\frac{ \epsilon  \abs{\xi^{\prime}} }{ 1 + \epsilon  \abs{\xi^{\prime}} } }_{\mathcal{M}^{\prime}} 
							\leq C
					\end{align}
				for all $0 < \epsilon \leq 1$.
				
				\item  Let $-1 \leq x_3 \leq 1$. 
				Then there exists a constant $C>0$, which is independent of $\epsilon$, such that
					\begin{align} \label{eq:C0505}
						\sLongBracket{ \frac{ \sinh \bracket{ \epsilon \abs{\xi^{\prime}} x_3 } }{ \cosh \bracket{ \epsilon \abs{\xi^{\prime}} } } }_{\mathcal{M}^{\prime}} \leq C, \quad
						\sLongBracket{ \frac{ \sinh \bracket{ \epsilon \abs{\xi^{\prime}} x_3 } }{ \cosh \bracket{ \epsilon \abs{\xi^{\prime}} } } }_{\mathcal{M}^{\prime}}  
						\leq C
					\end{align}
				for all $0 < \epsilon \leq 1$.
			\end{itemize}
		\end{proposition}

		
		\begin{proof}
			\FurukawaCommentFourth{The estimate} (\ref{eq:C0501}) is a direct consequence of Lemma 3.5 in \cite{Abels2002} and the Mikhlin theorem. 
			By definition of $\sinh$ and $\cosh$, we find the formula
				\begin{align} \label{eq:C0526}
					\frac{ 
						\sinh \bracket{ \epsilon \abs{\Prime{\xi}} x_3 } 
					}{
						\sinh \bracket{ \epsilon \abs{ \Prime{\xi} } } 
					} 
					= \frac{ 
						\Napier^{ \epsilon \abs{\Prime{\xi}} x_3 } 
						- \Napier^{ - \epsilon \abs{\Prime{\xi}} x_3 } 
					}{ 
						\Napier^{ \epsilon \abs{\Prime{\xi}} } 
						- \Napier^{ - \epsilon \abs{\Prime{\xi}} } 
					} 
					= \frac{ 
						\Napier^{ - \epsilon \abs{\Prime{\xi}} \bracket{ x_3 - 1 } } 
					}{ 
						1 
						- \Napier^{ - 2 \epsilon \abs{\Prime{\xi}} } 
					} 
					- \frac{ 
						\Napier^{ 
							- \epsilon \abs{\Prime{\xi}} \bracket{ 
								x_3 + 1 
							} 
						} 
					}{ 
						1 
						- \Napier^{ 
							- 2 \epsilon \abs{\Prime{\xi}} 
						} 
					}
				\end{align}
			and
			\begin{align} \label{eq:C0526_2}
				\FurukawaComment{
				\frac{ 
						\cosh \bracket{ \epsilon \abs{\Prime{\xi}} x_3 } 
					}{
						\sinh \bracket{ \epsilon \abs{ \Prime{\xi} } } 
					} 
				}
				= \frac{ 
						\Napier^{ - \epsilon \abs{\Prime{\xi}} \bracket{ x_3 - 1 } } 
					}{ 
						1 
						- \Napier^{ - 2 \epsilon \abs{\Prime{\xi}} } 
					} 
					+ \frac{ 
						\Napier^{ 
							- \epsilon \abs{\Prime{\xi}} \bracket{ 
								x_3 + 1 
							} 
						} 
					}{ 
						1 
						- \Napier^{ 
							- 2 \epsilon \abs{\Prime{\xi}} 
						} 
					} \FurukawaCommentFourth{.}
				\end{align} 
			Thus, multiplying $ \frac{ \epsilon \abs{\Prime{\xi}} }{ 1 + \epsilon \abs{\Prime{\xi}} }$ by both sides of (\ref{eq:C0526}), \FurukawaCommentSeventh{we find from Proposition \ref{prop_scaled_multiplier} that}
				\begin{align}
					&\sLongBracket{ 
							\frac{ 
								\sinh \bracket{ \epsilon \abs{\Prime{\xi}} x_3 } 
							}{ 
								\sinh \bracket{ \epsilon \abs{ \Prime{\xi} } } 
							} 
							\frac{ 
								\epsilon  \abs{\Prime{\xi}} 
							}{ 
								1 
								+ \epsilon \abs{\Prime{\xi}} 
							} 
						}_{\Prime{\mathcal{M}}} \notag \\ \label{eq:C0527}
					& \leq C 
					\sLongBracket{ 
						\Napier^{ - \epsilon \abs{\Prime{\xi}} \bracket{ x_3 - 1 } } 
						\frac{ 
							\epsilon  \abs{\Prime{\xi}} 
						}{ 
							\bracket{ 
								1 
								- \Napier^{ 
									- 2 \epsilon \abs{\Prime{\xi}} } 
							}
						} \frac{
							1
						}{ 
							\bracket{ 
								1 + \epsilon  \abs{\Prime{\xi}} 
							} 
						} 
					}_{\mathcal{M}^{\prime}} \notag \\
					& + \sLongBracket{ 
						\Napier^{ - \epsilon \abs{\Prime{\xi}} \bracket{ x_3 + 1 } } 
						\frac{ 
							\epsilon \abs{\Prime{\xi}} 
						}{ 
							\bracket{ 1 - \Napier^{ - 2 \epsilon \abs{\Prime{\xi}} } }
						} \frac{
							1 
						}{
							\bracket{ 1 + \epsilon  \abs{\Prime{\xi}} } 
						} 
					}_{\mathcal{M}^{\prime}} \notag \\
					& \leq C.
				\end{align}
			\FurukawaCommentSeventh{The second} inequality of (\ref{eq:C0504}) is proved by the same as above using (\ref{eq:C0526_2}).
			\FurukawaCommentSeventh{Similarly, by definition of $\sinh$ and $\cosh$, the estimate (\ref{eq:C0505}) follows.}
		\end{proof}
					
	\subsection{Estimate for $v_1$}
	
	Let us consider the equations (I).	
	For \FurukawaCommentSeventh{$a \in \Real$} we denote by $\tau_a f = f (a \cdot)$ the rescaling operator by $a$.
	The anisotropic Helmholtz projection $\mathbb{\ep{P}}^{\Real^3}$ on $\Real^3$ with symbols
		\begin{align*}
			\mathcal{F} \mathbb{P}_{\epsilon}^{\Real^3}
			= I_3 
			- \xi_{\epsilon} \otimes \xi_{\epsilon}, \quad
			\xi_{\epsilon} 
			= \LongBracket{
				\xi_1, \xi_2, \frac{\xi_3}{\epsilon} 
			} \in \Real^3,
		\end{align*}
	is bounded in $L^q(\Real^3)$ by boundedness of the Riesz operator and the formula
		\begin{align} \label{eq:0530}
			\FourierInverse_{\xi} m \bracket{ a \xi} \Fourier_x f
			= \tau_{a^{-1}}
				\sLongBracket{
					\FourierInverse_{\xi} m \bracket{ \xi } \Fourier_x \tau_a f
				}\FurukawaCommentFifth{.}
		\end{align}
		\FurukawaCommentFourth{Actually apply (\ref{eq:0530}) with respect to the third variable, then, \FurukawaCommentFifth{the symbol is no longer dependent on $\epsilon$}.
		Changing the variable with respect to and using boundedness of the Riesz operator, we find}
		\begin{align*}
			\Vert
				\mathbb{P}_\epsilon^{\Real^3} f
			\Vert_{L^{\FurukawaCommentFifth{q}} (\Real^3)}
			= \epsilon^{-1} \Vert
				\mathbb{P}_1 \sLongBracket{
					\tau_{1 / \epsilon}^3 f
				} 
			\Vert_{L^{\FurukawaCommentFifth{q}} (\Real^3)}
			\leq C \Vert
				f
			\Vert_{L^{\FurukawaCommentFifth{q}} (\Real^3)},
		\end{align*}
		where $\tau^3_a$ is the rescaled operator with respect to the third variable for $a > 0$.	
		We define the anisotropic Helmholtz projection \FurukawaComment{$\mathbb{P}_\epsilon^{\Torus^2 \times \Real}$ } on $\Torus^2 \times \Real$ with symbols \FurukawaCommentFourth{by}
		\begin{align*}
			\Fourier_{x_3} \Fourier_{d, \Prime{x} }\mathbb{P}_\epsilon^{\Torus^2 \times \Real}
			= I_3
				- \left( 
					\begin{array}{c}
						n_1 \\
						n_2 \\
						\xi_3 / \epsilon
					\end{array}
				\right)
				\otimes \left( 
					\begin{array}{c}
						n_1 \\
						n_2 \\
						\xi_3 / \epsilon
					\end{array}
				\right), \quad
				n_1, n_2 \in \mathbb{Z} , \, \, \xi_3 \in \Real.
		\end{align*}
		We find $\mathbb{P}_\epsilon^{\Torus^2 \times \Real}$ is bounded from $L^q (\Torus^2 \times \Real)$ into itself by boundedness of \FurukawaCommentThird{$\mathbb{P}^{\Real^3}_\epsilon$} and Proposition \ref{prop_Fourier_Multiplier_Theorem_Discrete} \FurukawaCommentSeventh{uniformly in $\epsilon \in (-1, 1)$}.

	
		\begin{proposition} \label{prop:C0510}
			Let $ 1 < q < \infty $, $0 < a < 1/2$, \FurukawaComment{$0<\epsilon \leq 1$}, $z \in \Complex$ satisfying $ - a < \RealPart z < 0$ and $0 < \theta < \pi / 2$. Then there exists a constant $C = C \bracket{ q, a, \theta }$ such that
				\begin{align} 
					& \LongNorm{					
						\frac{1}{ 2 \pi i } R_0 \TimeInt{
							\Gamma_\theta
						}{}{ 
							\bracket{- \lambda}^z \Inverse{ \lambda - \Delta_{\Torus^2 \times \Real} } \ep{ \mathbb{P} }^{\Torus^2 \times \Real}  E_0 f
						}{
							\lambda
						}
					}{\Lebesgue{q} \bracket{\Torus^2 \times \Real} } \notag \\ \label{eq:C0560}
					& \leq C \Napier^{\theta \abs{\ImaginaryPart \, z}} \LongNorm{
						f
					}{
						\Lebesgue{q}\bracket{\Omega}
					}
				\end{align}
			 for all $ f \in \Lebesgue{q} \bracket{ \Omega } $.
		\end{proposition}	
		\begin{proof}
			It is known that the Laplace operator on a cylinder $\Torus^2 \times \Real$ has BIP.
			Combining this fact and $L^q$-boundedness of $\mathbb{P}_\epsilon^{\Torus^2 \times \Real}$, we have (\ref{eq:C0560}).
		\end{proof}
		
		\FurukawaCommentFifth{
		\begin{proposition} \label{prop_nabla2_resolvent_estimate}
			Let $1< q < \infty$, $0 < \epsilon \leq 1$, $0 < \theta < \pi /2$ and $\lambda \in \Sigma_\theta$.
			Then there exists a constant $C = C(q) > 0$, which is independent of $\epsilon$, \FurukawaCommentSeventh{such that}
			\begin{align*}
				\left \Vert
					\nabla^2 \left(
						\lambda - \Delta_{\Torus^2 \times \Real} 
					\right)^{-1} \mathbb{P}_\epsilon^{\Torus^2 \times \Real} E_0 f
				\right \Vert_{L^q (\Omega)}
				\leq C \Vert
					f
				\Vert_{L^q (\Omega)}.
			\end{align*}
		\end{proposition}
		\begin{proof}
			\FurukawaCommentSeventh{This follows from Propositions \ref{prop_Fourier_Multiplier_Theorem_Discrete} and \ref{prop_scaled_multiplier} since $P_\epsilon^{\Torus^2 \times \Real}$ is uniformly bounded from $L^q (\Torus^2 \times \Real)$ into itself.}
		\end{proof}	}
	Let us calculate the partial Fourier transform for $v_1$ with respect to the horizontal variable.
	\FurukawaComment{This} is needed to obtain representation formula for $v_2$ later.
	Let $g \in L^q ( \FurukawaCommentFifth{\Torus^2 \times \Real})$.
	The solution $\FurukawaComment{\tilde{v}}$ to the equation 
		\begin{align*}
			\left \{
				\begin{array}{rcll}
					\lambda \tilde{v} - \Delta \tilde{v} + \ep{\nabla} \tilde{\pi}
					&= 
					& g 
					& \In \FurukawaCommentFifth{\Torus^2 \times \Real},\\
					\Dive \tilde{v}
					&= 
					& 0 
					& \In \FurukawaCommentFifth{\Torus^2 \times \Real},
				\end{array}
			\right.
		\end{align*}
	is given by
		\begin{align*}
			\tilde{v} 
			= \LongBracket{ 
				\lambda - \Delta_{\Real^3}
			}^{-1} \mathbb{P}_\epsilon^{\FurukawaCommentFifth{\Torus^2 \times \Real}} \FurukawaComment{g}.
		\end{align*}
	Moreover,
		\begin{align}
			K_{\lambda, \epsilon} g \notag
			& := \LongBracket{ 
				\lambda - \Delta_{\FurukawaCommentFifth{\Torus^2 \times \Real}}
				}^{-1} \mathbb{P}_\epsilon^{\FurukawaCommentFifth{\Torus^2 \times \Real}}  g \notag \\
			& = \mathcal{F}^{-1} 
				\bracket{ 
					\lambda 
					+ \FurukawaCommentFifth{ \abs{n^\prime}^2 + \xi_3^2}
				}^{- 1} \LongBracket{ 
					I_3 - \frac{\ep{\xi} \otimes \ep{\xi} }{ \abs{ \ep{\xi} }^2 } 
				} 
			\mathcal{F}  g \notag \\  \label{eq:C0550_discrete}
			& = \mathcal{F}^{- 1}_{\FurukawaCommentFifth{n^\prime}}  \TimeInt{\Real}{}{ 
					k_{\lambda, \epsilon} \bracket{ 
						\FurukawaCommentFifth{n^\prime}, x_3 - \zeta 
					}  
				\mathcal{F}_{\FurukawaCommentFifth{n^\prime}} g \bracket{\FurukawaCommentFifth{n^\prime}, \zeta } 
			}{\zeta},
		\end{align}
	where \FurukawaCommentFifth{$\xi_\epsilon = (n^\prime, \xi_3 / \epsilon) \in \Integer^2 \times \Real$ and}
		\begin{align}
			& \Prime{k}_{\lambda, \epsilon} \bracket{\FurukawaCommentFifth{n^\prime}, x_3 } \notag \\
			& = \FourierInverse_{\xi_3} \sLongBracket{
				\bracket{ 
					\lambda 
					+ \FurukawaCommentFifth{\abs{n^\prime}^2 + \xi_3^2}
				}^{- 1} \LongBracket{ 
						I_3 
						- \frac{\ep{\xi} \otimes \ep{\xi} }{ \abs{ \ep{\xi} }^2 } 
				}
			} \notag \\
			& = 
			\frac{ \Napier^{ - s_{\lambda} } }{ 2 s_{\lambda} }
				\left( 
					\begin{array}{cc}
						I & 0 \\
						0 & 0
					\end{array}
				\right) \notag \\
			&	-
				\left( 
					\begin{array}{c}
						\FurukawaCommentFifth{n^\prime} \otimes \FurukawaCommentFifth{n^\prime} \frac{
							\epsilon^2}{ 
							\lambda 
							+ \bracket{ 1 - \epsilon^2 } \abs{ \FurukawaCommentFifth{n^\prime} }^2 
							}  \frac{
								- \epsilon \abs{\FurukawaCommentFifth{n^\prime}} \Napier^{ - \abs{x_3} s_{\lambda} } 
								+ s_{\lambda} \Napier^{ - \abs{x_3} \epsilon \abs{ \FurukawaCommentFifth{n^\prime} } } 
							}{
								2 s_{\lambda} \epsilon \abs{\FurukawaCommentFifth{n^\prime}} 
							} \\
						- i { \FurukawaCommentFifth{n^\prime} }^T \frac{
							\epsilon^2 
						 }{ 
						 	\lambda 
						 	+ \bracket{ 1 - \epsilon^2 } \abs{ \FurukawaCommentFifth{n^\prime} }^2 
						 } \frac{
						 	\Napier^{- \abs{x_3} s_{\lambda}} - \Napier^{ - \abs{x_3} \epsilon \abs{\FurukawaCommentFifth{n^\prime}} }
						 }{ 
						 	2 
						 } 
					\end{array}
				\right. \notag \\
			& \quad \quad \quad \quad \quad \quad \quad \left.
					\begin{array}{c}
						- i \FurukawaCommentFifth{n^\prime}  \frac{ 
							\epsilon^2 
						}{ 
							\lambda 
							+ \bracket{ 1 - \epsilon^2 } \abs{ \FurukawaCommentFifth{n^\prime} }^2 
						} \frac{ 
							\Napier^{- \abs{x_3} s_{\lambda}} 
							- \Napier^{ 
								- \abs{x_3} \epsilon \abs{\FurukawaCommentFifth{n^\prime}} 
							}
						}{ 
							2 
						} \notag \\
						- \abs{ \FurukawaCommentFifth{n^\prime} }^2 \frac{
						 	\epsilon^2
						 }{ 
						 	\lambda 
						 	+ \bracket{ 1 - \epsilon^2 } \abs{ \FurukawaCommentFifth{n^\prime} }^2 
						 } \frac{
						 	- \epsilon \abs{\FurukawaCommentFifth{n^\prime}} \Napier^{ - \abs{x_3} s_{\lambda} } + s_{\lambda} \Napier^{ - \abs{x_3} \epsilon \abs{ \FurukawaCommentFifth{n^\prime} } } 
						 }{
							2 s_{\lambda} \epsilon \abs{\FurukawaCommentFifth{n^\prime}} 
						 }
					\end{array}
				\right) \notag \\
			& = : \frac{ 
				\Napier^{ - s_{\lambda} } 
			}{ 
				2 s_{\lambda} 
			}
				\left( 
					\begin{array}{cc}
						I_2 & 0 \\
						0 & 0
					\end{array}
				\right)
				-
				\left( 
					\begin{array}{cc}
						\FurukawaCommentFifth{n^\prime} \otimes \FurukawaCommentFifth{n^\prime} \eta_{\lambda, \epsilon}^{\prime} \bracket{ \FurukawaCommentFifth{n^\prime}, x_3 } 
						& - i \FurukawaCommentFifth{n^\prime} \partial_3 \eta_{\lambda, \epsilon}^{\prime} \bracket{ \FurukawaCommentFifth{n^\prime}, x_3 } \\ \label{eq:C0555}
						- i { \FurukawaCommentFifth{n^\prime} }^T \partial_3 \eta_{\lambda, \epsilon}^{\prime} \bracket{ \FurukawaCommentFifth{n^\prime}, x_3 }  
						& - \abs{\xi^{\prime}}^2 \eta_{\lambda, \epsilon}^{\prime} \bracket{\FurukawaCommentFifth{n^\prime}, x_3 }
					\end{array}
				\right).
		\end{align}
	\FurukawaCommentFifth{The kernel function} $k_{\lambda, \epsilon} \bracket{\FurukawaCommentFifth{n^\prime}, x_3}$ is calculated by the residue theorem. 
	Actually, since poles of $\Inverse{ \lambda + \FurukawaCommentFifth{\abs{n^\prime}^2 + \xi_3^2} }$ are $\xi_3 = \pm i s_{\lambda} $, the residue theorem implies the partial Fourier inverse transform of $ \Inverse{ \lambda + \FurukawaCommentFifth{\abs{n^\prime}^2 + \xi_3^2}  } $ with respect to $\xi_3$ is given by inserting $\xi_3 = i s_{\lambda} \, \mathrm{or} \, - i s_{\lambda}$ into $\Napier^{ i x_3 \xi_3}$ so that the real part become to be negative. 
	Thus, we have
		\begin{align} \label{eq_definition_of_e_prime_lamda}
			\Prime{e}_\lambda (\FurukawaCommentSixth{n^\prime}, x_3)
			: = \FourierInverse_{\xi_3} \LongBracket{
				\lambda + \FurukawaCommentSixth{\abs{n^\prime}^2 + \abs{x_3}^2}
			}^{-1}
			= \frac{
				e^{- \abs{x_3} s_\lambda}}{
				s_\lambda
			}.
		\end{align}
	Moreover, this formula leads \FurukawaCommentSixth{to}
		\begin{align*}
			\FourierInverse_{\xi_3} \sLongBracket{
				\abs{\xi_\epsilon}^2
			}^{-1}
			= \FourierInverse_{\xi_3} \sLongBracket{
				\frac{
					\epsilon^2
				}{
					\epsilon^2 \abs{\FurukawaCommentFifth{n^\prime}}^2 + \xi_3^2
				}
			}
			= \frac{
				\epsilon e^{- \abs{x_3} \epsilon \abs{\FurukawaCommentFifth{n^\prime}} }
			}{
				\abs{\FurukawaCommentFifth{n^\prime}}
			}.
		\end{align*}
	Combining the above two calculations and the formula
		\begin{align*}
			I_3 - \frac{ \xi_\epsilon \otimes \xi_\epsilon }{ \abs{\xi_\epsilon}^2 }
			=	\left( 
					\begin{array}{cc}
						I_2 & 0 \\
						0   & 0
					\end{array}
				\right)
			- 
				\left(
					\begin{array}{cc}
						\frac{ 
							\FurukawaCommentFifth{n^\prime} \otimes \FurukawaCommentFifth{n^\prime} 
						}{ 
							\abs{\FurukawaCommentFifth{\xi_\epsilon}}^2
						} 
						& \frac{ 
							\xi_3 \FurukawaCommentFifth{n^\prime} / \epsilon
						}{ 
								\abs{\FurukawaCommentFifth{\xi_\epsilon}}^2 
						} \\
						\frac{ 
							\xi_3 \FurukawaCommentFifth{n^\prime}^T / \epsilon
						}{ 
							\abs{\FurukawaCommentFifth{\xi_\epsilon}
						}^2 
						} 
						& - \frac{ 
							\abs{ \FurukawaCommentFifth{n^\prime}_\epsilon }^2 
						}{ 
							\abs{\FurukawaCommentFifth{\xi_\epsilon}}^2 
						}
					\end{array}
				\right)	,
		\end{align*}
	\FurukawaCommentFifth{we obtain (\ref{eq:C0555})}.

	\subsection{Boundedness of the anisotropic Helmholtz projection}
	
	Next, we consider the equation (III) with boundary data $\phi = \bracket{\phi_+, \phi_-} \FurukawaCommentFifth{\in C_0^\infty (\Omega)}$. 
	Applying the partial Fourier transform to (III), we have
		\begin{align} \label{eq:C0701}
			\left \{
				\begin{array}{rcl}
				\LongBracket{ 
						\frac{
							\partial_z^2
						}{
						\epsilon^2
						} 
					- \abs{ n^\prime }^2 } \mathcal{F}_{d, x^{\prime} } \pi_3 (\Prime{n}, x_3) 
					& =
					& 0 , \\
					\frac{\partial_z }{ \epsilon} \mathcal{F}_{d, x^{\prime} } \pi_3 \bracket{ n^\prime, \pm 1}
					& = 
					& \Fourier_{d, \Prime{x}} {\phi}_{\pm} \bracket{\Prime{n}} \FurukawaCommentFourth{,}
				\end{array}
			\right .
		\end{align}
	for $n^\prime \in \mathbb{Z}^2 \FurukawaComment{ \setminus \{ 0 \}}$ and $x_3 \in (-1 , 1)$.
	The solution to (\ref{eq:C0701}) is of the form 
		\begin{align*}
			\Fourier_{d, \Prime{x}} \pi_3  (n^\prime, x_3)
			= C_1 \Napier^{\epsilon x_3 \abs{\Prime{n}} } 
			+ C_2 \Napier^{- \epsilon x_3 \abs{\Prime{n}} }
		\end{align*} 
	for some constant $C_1$ and $C_2$. 
	Take the constants so that (\ref{eq:C0701}) satisfied, namely 
		\begin{align*}
			& C_1 
			= \frac{
					\Fourier_{d, \Prime{x}} \phi_+ 
					+ \Fourier_{d, \Prime{x}} \phi_-  
				}{ 
					4 \abs{\Prime{n}} \cosh \bracket{ \epsilon \abs{\Prime{n}} }
				}
			+ \frac{
					\Fourier_{d, \Prime{x}} \phi_+ 
					- \Fourier_{d, \Prime{x}} \phi_-  
				}{ 
					4 \abs{\Prime{n}} \sinh \bracket{ \epsilon \abs{\Prime{n}} } 
				}, \quad \\
			& C_2 
			= - 
				\frac{
					\Fourier_{d, \Prime{x}} \phi_+ 
					+ \Fourier_{d, \Prime{x}} \phi_-  
				}{ 
					4 \abs{\Prime{\xi}} \cosh \bracket{ \epsilon \abs{\Prime{n}} } 
				}
			+ \frac{
					\Fourier_{d, \Prime{x}} \phi_+ - \Fourier_{d, \Prime{x}} \phi_-  
				}{ 4 \abs{\Prime{\xi}} \sinh \bracket{ \epsilon \abs{\Prime{n}} }
				},
		\end{align*}
	then the solution to (\ref{eq:C0701}) is given by
		\begin{align*}
			& \pi_3 \bracket{ x^{\prime}, x_3 } \\
			 & = \mathcal{F}^{ - 1 }_{d,  n^{\prime} } \LongBracket{
			 	\frac{
			 		\sinh \bracket{ \epsilon x_3 \abs{ \Prime{n}}} 
			 	}{ 
			 		\abs{\Prime{n} } \cosh \bracket{\epsilon \abs{ \Prime{n}} } 
			 	} \frac{ 
			 		\Fourier_{d, \Prime{x}} {\phi}_{+} 
			 		+ \Fourier_{\Prime{x}} {\phi}_{-} 
			 	}{
			 		2
			 	} 
			 + \frac{
			 	\cosh \bracket{ \epsilon x_3 \abs{ \Prime{n}}} 
			 }{ 
				\abs{\Prime{n} } \sinh \bracket{\epsilon \abs{ \Prime{n}} } 
			 } \frac{ 
			 	\Fourier_{d, \Prime{x}} {\phi}_{+} 
			 	- \Fourier_{d, \Prime{x}} {\phi}_{-} 
			 }{
			 	2
			 } }.
		\end{align*}
	Moreover, its anisotropic gradient given by
		\begin{align}
			\nabla_{\epsilon} \pi_3 
			& = \mathcal{F}^{ - 1 }_{d,  n^{\prime} } 
				\left(
					\begin{array}{c}
						\frac{
							i n^{\prime} \sinh \bracket{ \epsilon x_3 \abs{ n^{\prime} } } 
						}{ 
							\abs{ n^{\prime} } \cosh \bracket{ \epsilon \abs{ n^{\prime} } }
						} 
						\frac{ 
							\Fourier_{d, \Prime{x}} {\phi}_{+} 
							+ \Fourier_{d, \Prime{x}} {\phi}_{-} 
						}{2} 
					+ \frac{ 
						i n^{\prime} \cosh \bracket{ \epsilon x_3 \abs{ n^{\prime} } } 
					}{ 
						\abs{ n^{\prime} } \sinh \bracket{ \epsilon \abs{ n^{\prime} } } 
					} \frac{ 
						\Fourier_{d, \Prime{x}} {\phi}_{+} 
						- \Fourier_{d, \Prime{x}} {\phi}_{-} 
						}{2} \\
					\frac{  
						\cosh \bracket{ \epsilon x_3 \abs{ n^{\prime} } } 
					}{  
						\cosh \bracket{ \epsilon \abs{ n^{\prime} } } 
					} \frac{ 
						\Fourier_{d, \Prime{x}} {\phi}_{+} 
						+ \Fourier_{d, \Prime{x}} {\phi}_{-} 
					}{2} 
					+ \frac{ 
						\sinh \bracket{ \epsilon x_3 \abs{ n^{\prime} } } 
					}{ 
						\sinh \bracket{ \epsilon \abs{ n^{\prime} } } 
					} \frac{ 
						\Fourier_{d, \Prime{x}}{\phi}_{+} 
						- \Fourier_{d, \Prime{x}} {\phi}_{-} }{2}
					\end{array}
				\right) \notag \\ \label{eq:C0712}
			& =: \mathcal{F}^{ - 1}_{d,  n^{\prime} } 
				\alpha_{ \epsilon, +} \bracket{ n^{\prime}, x_3 } 
			\mathcal{F}_{d,  x^{\prime} } \phi_{+} 
			+ \mathcal{F}^{ - 1}_{d,  n^{\prime} } 
				\alpha_{ \epsilon, - } \bracket{ n^{\prime}, x_3} 
			\mathcal{F}_{d,  x^{\prime} } \phi_{-}. 
		\end{align}
	We \FurukawaComment{apply the trace} to (\ref{eq:C0712}) to get
		\begin{align*}
			\gamma_{ \pm } \ep{\nabla} \pi_3 
			&= \mathcal{F}^{ - 1 }_{d,  n^{\prime} }
			\left(
				\begin{array}{c}
					\frac{ \pm i n^{\prime} \sinh \bracket{ \epsilon \abs{ n^{\prime} } } }{ \abs{ n^{\prime} } \cosh \bracket{ \epsilon \abs{ n^{\prime} } } } 
					\frac{ 
						\Fourier_{d, \Prime{x}} {\phi}_{+} 
						+ \Fourier_{d, \Prime{x}} {\phi}_{-} 
					}{2}
					+ \frac{ 
						i n^{\prime} \cosh \bracket{ \epsilon \abs{ n^{\prime} } } 
					}{ 
						\abs{ n^{\prime} } \sinh \bracket{ \epsilon \abs{ n^{\prime} } } 
					} \frac{ 
						\Fourier_{d, \Prime{x}} {\phi}_{+} -
						\Fourier_{d, \Prime{x}} {\phi}_{-} }{2} \\
					\frac{ 
						\Fourier_{d, \Prime{x}} {\phi}_{+} 
						+ \Fourier_{d, \Prime{x}} {\phi}_{-} 
					}{2}
					\pm \frac{ 
						\Fourier_{d, \Prime{x}} {\phi}_{+} 
						- \Fourier_{d, \Prime{x}} {\phi}_{-} }{2}
				\end{array}
			\right).
		\end{align*}
	We insert \FurukawaCommentFifth{$\phi_{\pm} = \gamma_{\pm} \ep{\mathbb{P}}^{\Torus^2 \times \Real} f$ to (\ref{eq:C0712}) for $f \in C_0^\infty (\Omega)$} \FurukawaCommentSeventh{satisfying $\tilde{f}=0$} and set
		\begin{align*}
			\Pi_{\epsilon} f 
			& := \mathcal{F}^{-1}_{d, n^{\prime}}
				\sLongBracket{
					\alpha_{ \epsilon, + } \bracket{ n^{\prime}, x_3 } \gamma_+ \mathcal{F}_{d, x^{\prime}} \LongBracket{  
						\Napier_3 \cdot \mathbb{P}_{\epsilon}^{\Torus^2 \times \Real} E_0 f 
					} 
				} \\
			 & +  \mathcal{F}^{-1}_{d, n^{\prime}}
				\sLongBracket{				
					\alpha_{ \epsilon, - } \bracket{ n^{\prime}, x_3 } \gamma_-  \mathcal{F}_{d, x^{\prime}} \LongBracket{ 
						\Napier_3 \cdot \mathbb{P}_{\epsilon}^{\Torus^2 \times \Real} E_0 f 
					}
				}.
		\end{align*}
	
	
		\begin{lemma} \label{lem:C0710}
			 \FurukawaCommentFifth{Let $1 < q < \infty$, $0 < \epsilon \leq 1$ and $s \geq 0$.}
			 Then there \FurukawaComment{exists} a constant $C = C \bracket{q}$, \FurukawaCommentFifth{which is independent of $\epsilon$, the operator  $\Pi_\epsilon$ can be extended to a bounded operator from $W^{s,q}_{\FurukawaCommentSeventh{\mathrm{af}}} (\Omega)$ into itself} such that
				\begin{align} \label{eq_bound_Pi}
					\LongNorm{
						\Pi_\epsilon f
					}{\FurukawaCommentFifth{W^{s, q}} \bracket{\Omega} }
					\leq C \LongNorm{f}{\FurukawaCommentFifth{W^{s, q}} \bracket{\Omega}}
				\end{align}
			for all \FurukawaCommentFifth{$f \in W^{s, q}_{\FurukawaCommentSixth{\mathrm{af}}} \bracket{\Omega}$}.
		\end{lemma}

		
		\begin{proof}
			\FurukawaCommentSixth{Let $f \in C_0^\infty (\Omega)$ satisfy $\tilde{f} = 0$.}
			We seek the multiplier of $\Pi_\epsilon$ by a direct calculation.
			Recall that the symbol of $\ep{\mathbb{P}}^{\Torus^2 \times \Real}$ is of the form
				\begin{align} \label{eq:C0715}
					\Fourier_{x_3} \Fourier_{d, x^\prime} \ep{\mathbb{P}}^{\Torus^2 \times \Real} 
					=
						\left(
							\begin{array}{cc}
								I_2 & 0 \\
								0 & 0 \\							
							\end{array}
						\right)
						-
							\left(
							\begin{array}{cc}
								\frac{ \Prime{n} \otimes \Prime{n} }{ \abs{ \ep{ \Prime{n}}}^2 } 
								& \frac{ \Prime{n} \xi_3 / \epsilon }{ \abs{ \ep{ \Prime{n}}}^2 } \\
								\frac{{\Prime{n}}^T \xi_3 / \epsilon }{ \epsilon \abs{ \ep{ \Prime{n}}}^2 } 
								& - \frac{ \abs{ \Prime{n}}^2 }{ \abs{ \ep{ \Prime{n}}}^2 } \\		
							\end{array}
						\right).
				\end{align}
			Since the symbol of $\ep{\mathbb{P}}^{\Torus^2 \times \Real} $ have poles at $\xi_3 = \pm i \epsilon \abs{\Prime{n}}$, we apply $e_3 \cdot$ to (\ref{eq:C0715}) by the left hand side and use the residue theorem so that the power of $e$ is negative to get
				\begin{align} \label{eq:C0717}
					\Fourier_{d, x^{\prime}} 
						\bracket{ \Napier_3 \cdot \mathbb{P}_{\epsilon}^{\Torus^2 \times \Real}  E_0 f }
					& = 
					- \TimeInt{-1}{1}{ 
							\frac{ 
								i \Napier^{ - \abs{ x_3 - \zeta} \epsilon \abs{ n^{\prime} } } 
							}{2} \epsilon n^{\prime} \cdot 
							\Fourier_{d, x^{\prime}} f^{\prime} \bracket{ n^{\prime}, \zeta } 
						}{\zeta}  \notag \\
					& + 
					\TimeInt{-1}{1}{ 
						\frac{ 
							\Napier^{ - \abs{ x_3 - \zeta} \epsilon \abs{ \Prime{n} } } 
						}{2}  \epsilon \abs{ \Prime{n} } \Fourier_{d, x^{\prime}} f_3 \bracket{ n^{\prime}, \zeta} 
					}{\zeta}.
				\end{align}						
			Note that the integration is due to the relationship between the Fourier transform and convolution. 
			Applying trace operators $\gamma_{\pm}$ and $\alpha_{\epsilon, \pm} \bracket{\Prime{n}, x_3} $, respectively, and taking Fourier inverse transform with respect to $\Prime{n}$, we find 
				\begin{align}
					& \Pi_{\epsilon} f \bracket{ {\Prime{x}, x_3} } \notag \\
					& = 
					- \FourierInverse_{d, \Prime{n}} 
							\TimeInt{-1}{1}{ 
								\alpha_{\epsilon, +} \bracket{ n^{\prime}, x_3 } \frac{ i \Napier^{ - \abs{ 1 - \zeta} \epsilon \abs{ n^{\prime} } } }{2} \epsilon n^{\prime} \cdot \Fourier_{\FurukawaCommentFifth{d,} x^{\prime}} f^{\prime} \bracket{ n^{\prime}, \zeta } 
							}{\zeta} \notag \\
					& + \FourierInverse_{d, \Prime{n}} 
							\TimeInt{-1}{1}{
								\alpha_{\epsilon, +} \bracket{ n^{\prime}, x_3 } 
								\frac{ \Napier^{ - \abs{ 1 - \zeta} \epsilon \abs{ \Prime{n} } } }{2} \epsilon \abs{ \Prime{n} } \Fourier_{d, x^{\prime}} f_3 \bracket{ n^{\prime}, \zeta} 
							}{\zeta}
					 \notag \\ 
					& 
					- \FourierInverse_{d, \Prime{n}}
							\TimeInt{-1}{1}{ 
								\alpha_{\epsilon, - } \bracket{ n^{\prime}, x_3 } 
								\frac{ i \Napier^{ - \abs{ -1 - \zeta} \epsilon \abs{ n^{\prime} } } }{2} 
								\epsilon n^{\prime} \cdot \Fourier_{d, x^{\prime}} f^{\prime} \bracket{ n^{\prime}, \zeta } 
							}{\zeta} \notag \\
					& + \FourierInverse_{d, \Prime{n}} 
							\TimeInt{-1}{1}{ 
								\alpha_{\epsilon, -} \bracket{ n^{\prime}, x_3 } 
								\frac{ \Napier^{ - \abs{ - 1 - \zeta} \epsilon \abs{ \Prime{n} } } }{2} 
								\epsilon \abs{ \Prime{n} } \Fourier_{d, x^{\prime}} f_3 \bracket{ n^{\prime}, \zeta} 
							}{\zeta} \notag \\ \label{eq:0720}
					& = : I_1 + I_2 + I_3 + I_4\FurukawaCommentFifth{.}
				\end{align}
			\FurukawaCommentFifth{By} the definition of $\alpha_{\pm, \epsilon}$,
				\begin{align*}
					I_1
					& = - \FourierInverse_{d, \Prime{n}} 
						\int_{-1}^{1}
							\frac{1}{2}
							\left(
								\begin{array}{c}
									\frac{ 
										i n^{\prime} \sinh \bracket{ \FurukawaCommentSeventh{\epsilon} x_3 \abs{ n^{\prime} } } 
									}{ 
										\abs{ n^{\prime} } \cosh \bracket{ \epsilon \abs{ n^{\prime} } } 
									} 
									+ \frac{ 
										i n^{\prime} \cosh \bracket{ \epsilon x_3 \abs{ n^{\prime} } } 
									}{ 
										\abs{ n^{\prime} } \sinh \bracket{ \epsilon  \abs{ n^{\prime} } } 
									} \\
									\frac{
										\cosh (\epsilon x_3 \abs{\Prime{n}})
									}{
										\cosh (\epsilon \abs{\Prime{n}})
									}	
									+	\frac{
										\sinh (\epsilon x_3 \abs{\Prime{n}})
									}{
										\sinh (\epsilon \abs{\Prime{n}})
									}
								\end{array}
							\right) \\
							& \quad \quad \quad \quad \quad \quad \quad \quad 
							\times \frac{ 
									i \Napier^{ - \abs{ 1 - \zeta} \epsilon \abs{ n^{\prime} } } 
								}{2} \epsilon n^{\prime} \cdot 
					\Fourier_{d, x^{\prime}} f^{\prime} \bracket{ n^{\prime}, \zeta } 
							d\zeta.
				\end{align*}
			Symbols in the integral can be written by $A \bracket{\epsilon \Prime{n}} \bracket{ 1 + \epsilon \abs{\Prime{n}} } \Napier^{- \epsilon \abs{\Prime{n}} \bracket{ \abs{x_3 \pm 1} + \abs{\zeta \pm 1} }} $ for a symbol $A$ with \FurukawaCommentFourth{an} $\epsilon$-independent Mikhlin constant by \FurukawaCommentSeventh{Proposition \ref{prop:C05}}.
			The same argument is valid for $I_j$ ($j = 2, 3, 4$).
			Thus, we find from \FurukawaCommentSeventh{Propositions \ref{prop:C03} and \ref{prop_scaled_multiplier}} that
				\begin{align*}
					\LongNorm{
						\ep{\Pi} f
					}{L^q \bracket{\Omega} }
					& \leq C \LongNorm{
						f
					}{L^{\FurukawaCommentFourth{q}} (\Omega)}
					+ C\LongNorm{ \TimeInt{-1}{1}{ 
						\frac{ \LongNorm{f \bracket{ \cdot, \zeta }}{\Lebesgue{q} \bracket{\Torus^2}} }{ \abs{ x_3 \pm 1 } + \abs{ \zeta \pm 1  } } }{\zeta} 
					}{\Lebesgue{q} \bracket{-1, 1}} \notag \\
					& \leq C \LongNorm{f}{\Lebesgue{q} \bracket{\Omega} }
				\end{align*}				 
			for all \FurukawaCommentSixth{$f \in C_0^\infty (\Omega)$ satisfying $\tilde{f} = 0$}, where the constant $C$ is independent of $\epsilon$.
			\FurukawaCommentFifth{Thus the estimate (\ref{eq_bound_Pi}) holds for $m = 0$.
			We find from the formula (\ref{eq:0720}) that $\partial_j$ commutes with $\Pi_\epsilon$ for $j = 1, 2$.
			Moreover, the equation (\ref{eq:C0701}) implies
			\begin{align*}
				\partial_3^2 \, \Pi_\epsilon f
				& = - \epsilon^2  (\partial_1^2 + \partial_2^2) \, \Pi_\epsilon f \\
				& = - \epsilon^2   \Pi_\epsilon \, (\partial_1^2 + \partial_2^2) f
			\end{align*}
			Thus we find (\ref{eq_bound_Pi}) holds for all positive even number $m$.
			We can obtain (\ref{eq_bound_Pi}) for all \FurukawaCommentSixth{$s>0$} by interpolation.}
		\end{proof}
	
	 \FurukawaCommentFifth{ We set the operator
	\begin{align*}
		P_{N, \epsilon} 
		:= R_0 \left( \mathbb{P}_{\epsilon}^{\Torus^2 \times \Real} - \Pi_{\epsilon} \right) f
	\end{align*}
	for all $f \in C_0^\infty (\Omega)$ \FurukawaCommentSeventh{satisfying $\tilde{f}=0$}.}
	Then \FurukawaCommentThird{Lemma} \ref{lem:C0710} implies 
	\begin{corollary} \label{cor_bound_P_N_epsilon}
		 Let $1 < q < \infty$, \FurukawaCommentFifth{$0 < \epsilon \leq 1$ and $s > 0$.} 
		 Then there \FurukawaCommentFourth{exists} a constant $C > 0$, which is independent of $\epsilon$, \FurukawaCommentFifth{the operator  $P_{N, \epsilon}$ can be extended to a bounded operator from $W^{s,q}_{\FurukawaCommentSeventh{\mathrm{af}}} (\Omega)$ into itself} such that \FurukawaCommentFifth{
		 	\begin{align*}
		 		\LongNorm{
		 			P_{N, \epsilon} f
		 		}{
		 			W^{s, q} (\Omega)
		 		}
		 		\leq C \LongNorm{
		 			f
		 		}{
		 			W^{s, q} (\Omega)
		 		}
		 	\end{align*}
		} for all \FurukawaCommentFifth{$f \in W^{s, q}_{\FurukawaCommentSixth{\mathrm{af}}} (\Omega)$}.
	\end{corollary}
	
		\FurukawaComment{Note that} $P_{N, \epsilon}$ is not the anisotropic Helmholtz projection on $\Omega$.
		$P_{N, \epsilon}$ is the operator which maps from the \FurukawaCommentSeventh{$L^{\FurukawaCommentFifth{q}}$}-vector fields into \FurukawaCommentSeventh{$L^{\FurukawaComment{q}}$}-divergence-free vector fields  with tangential trace.
		However,  we find that the anisotropic Helmholtz projection is bounded from $L^{\FurukawaComment{q}} (\Omega)$ into itself by the same method of Lemma \ref{lem:C0710}.
		Let $u \in \FurukawaCommentSixth{C^\infty_0} (\Omega)$.
		Then, we obtain the solution $\pi_\epsilon$ to the Neumann problem
			\begin{align} \label{eq:C0725}
				\left \{
					\begin{array}{rll}
						\Delta_\epsilon \pi_{\epsilon}
						& = \mathrm{div}_\epsilon \, u
						& \mathrm{in} \quad \Omega, \\
						\gamma_\pm \frac{\partial_3 \pi_{\epsilon}}{\epsilon}
						& = u \cdot \nu_\pm
						& \mathrm{on} \quad \partial \Omega.				
					\end{array}
				\right.
			\end{align}
		The anisotropic Helmholtz projection $H_\epsilon$ is defined by
			\begin{align*}
				H_\epsilon u = u - \nabla_\epsilon \pi_\epsilon.
			\end{align*}
		In the case of the Dirichlet boundary condition, i.e. $\gamma_\pm u = 0$, the right hand side of the second equality of (\ref{eq:C0725}) is zero.
		Let us consider the $L^{\FurukawaCommentFourth{q}}$-boundedness of $\nabla_\epsilon \pi_\epsilon$, which implies the boundedness of the anisotropic Helmholtz projection.
		\FurukawaCommentSixth{For the solution $\pi^0$ to the equation
		\begin{align*}
			\left \{
				\begin{array}{rcll}
					\partial_3^2 \pi^0 (x_3) / \epsilon^2 
					& =
					& \partial_3 \tilde{u}_3 (x_3) / \epsilon, 
					& x_3 \in (-1 , 1), \\
					\partial_3 \pi^0 (\pm 1) / \epsilon
					& =
					& 0,
				\end{array}
			 \right.
		\end{align*}
		where $\tilde{u}_3 = \int_{\Torus^2} u_3 \, dx^\prime$, we have 
		\begin{align} \label{eq_C0730_0}
			\nabla_\epsilon \pi^0 = (0, 0, \tilde{u}_3 )^T.
		\end{align}
		}
		Let $\pi_\epsilon^1$ and $\pi_\epsilon^2$ be the solutions to
			\begin{align} \label{eq:C0730}
						\Delta_\epsilon \pi_{\epsilon}^1
						 = E_0 \mathrm{div}_\epsilon u \quad
						 \mathrm{in} \quad \Torus^2 \times \Real,
			\end{align}
		and 
			\begin{align*}
				\left \{
					\begin{array}{rll}
						\Delta_\epsilon \pi_{\epsilon} ^2
						& = 0
						& \mathrm{in} \quad \Omega \FurukawaCommentFourth{,}\\
						\FurukawaCommentSixth{\gamma_\pm \partial_3 \pi_{\epsilon}^2 / \epsilon}
						& = - \gamma_\pm \nu_\pm \cdot \nabla_\epsilon \Delta_\epsilon^{-1} E_0 \mathrm{div}_\epsilon \, u
						& \mathrm{on} \quad \partial \Omega,				
					\end{array}
				\right.
			\end{align*}
		respectively, \FurukawaCommentSixth{for $u \in C^\infty_0 (\Omega)$ satisfying $\tilde{u} = 0$.}
		Let us first consider (\ref{eq:C0730}).
		It follows from integration by parts
		\begin{align*}
			\Fourier_{x_3} \Fourier_{d, x^\prime} E_0 \mathrm{div}_\epsilon u
			& = \Fourier_{d, \Prime{x}} \TimeInt{-1}{1}{
				e^{- i x_3 \xi_3} \left(
					\mathrm{div}_H \Prime{u} (\cdot, x_3)
					+ \frac{\partial_3 u_3 (\cdot, x_3)}{\epsilon}
				\right)
			}{x_3} \\
			& = i \left(
				\begin{array}{cc}
					n^\prime
					& \xi_3 / \epsilon
				\end{array}
			\right)^T \cdot \Fourier_{d, x^\prime} ( E_0 u) \\
			& = \Fourier_{x_3} \Fourier_{d, x^\prime} \mathrm{div}_\epsilon (E_0 u).
		\end{align*}	
		This formula, the Mikhlin theorem and Proposition \ref{prop_Fourier_Multiplier_Theorem_Discrete} imply
			\begin{align} \label{eq:CC0735}
				\LongNorm{
					\nabla_\epsilon \pi_\epsilon^1
				}{
					L^{\FurukawaCommentFourth{q}} (\Omega)					
				}
				\leq C \LongNorm{
					u
				}{
					L^{\FurukawaCommentFourth{q}} (\Omega)
				},
			\end{align}
		where $C>0$ is independent of $\epsilon$.
		Moreover, since $e_3 \cdot \nabla_\epsilon \Delta_\epsilon^{-1} \mathrm{div}_\epsilon$ is given by the \FurukawaCommentFourth{left-hand} side of (\ref{eq:C0717}), we can use the same method as Lemma \ref{lem:C0710} to \FurukawaComment{get}
			\begin{align} \label{eq:CC0736}
				\LongNorm{
					\nabla_\epsilon \pi_\epsilon^2
				}{
					L^{\FurukawaCommentFourth{q}} (\Omega)					
				}
				\leq C \LongNorm{
					u
				}{
					L^{\FurukawaCommentFourth{q}} (\Omega)
				},
			\end{align}
		where $C>0$ is also independent of $\epsilon$.
		\FurukawaCommentSixth{The formula (\ref{eq_C0730_0}) and} \FurukawaCommentSixth{estimates} (\ref{eq:CC0735}) and (\ref{eq:CC0736}) \FurukawaCommentFourth{imply} $L^p$-boundedness of the anisotropic Helmholtz projection on $\Omega$.
	\FurukawaCommentFifth{
	Summing up the above argument, we have
	\begin{lemma} \label{lem_bound_H_epsilon}
		 Let $1 < q < \infty$ and $0 < \epsilon \leq 1$. Then there \FurukawaCommentFifth{exists} a constant $C > 0$, which is independent of $\epsilon$, such that
		 	\begin{align*}
		 		\LongNorm{
		 			H_{ \epsilon} f
		 		}{
		 			L^q (\Omega)
		 		}
		 		\leq C \LongNorm{
		 			f
		 		}{
		 			L^q (\Omega)
		 		}
		 	\end{align*}
		 for all $f \in L^q (\Omega)$.
	\end{lemma}
	}
		\begin{proposition} \label{prop:C0715}
			Let $0 <\epsilon \leq 1$, $ 1 < q < \infty $, $0 < a < 1/2$, $z \in \Complex $ satisfying $- a < \RealPart z < 0$ and $ 0 < \theta < \pi / 2$. 
			Then there exists a constant \FurukawaCommentFifth{$C = C \bracket{q, a, \theta}$, which is independent of $\epsilon$,} the solution $ \pi_3 $ to (III) with boundary data $\bracket{ \gamma K_{\lambda, \epsilon} f \cdot \nu } \nu$ satisfies
				\begin{align*}
					\LongNorm{				
						\frac{1}{ 2 \pi i } \TimeInt{\Gamma_\theta}{}{
							\bracket{- \lambda }^{z} \ep{\nabla} \pi_3 
						}{\lambda}
					}{\Lebesgue{q} \bracket{\Omega} }
					\leq C \Napier^{\theta \abs{\ImaginaryPart \, z}} \LongNorm{f}{\Lebesgue{q}\bracket{\Omega}}
				\end{align*}
			\FurukawaCommentFifth{for all $f \in \Lebesgue{q} \bracket{\Omega} $.}
		\end{proposition}

		\begin{proof}
			\FurukawaCommentFifth{In view of Remark \ref{rem_average_free}, we \FurukawaCommentSixth{may assume} $\tilde{f} = 0$ without loss of generality.}
			Since
				\begin{align} \label{eq_formula_nabla_epsilon_pi_3}
					\ep{\nabla} \pi_3 = \ep{\Pi} K_{\lambda, \epsilon} E_0 f
				\end{align}
			and the Cauchy integral commutes with $\ep{\Pi}$, \FurukawaCommentSixth{the conclusion is obtained from} Proposition \ref{prop:C0510} and Lemma \ref{lem:C0710}.
		\end{proof}
	\FurukawaCommentFifth{
	\begin{proposition} \label{prop_nabla2_resolvent_estimate_2}
			Let $0 <\epsilon \leq 1$, $ 1 < q < \infty $, $ 0 < \theta < \pi / 2$ and $\lambda \in \Sigma_\theta$. 
			Then there exists a constant $C = C \bracket{q, \theta}$, which is independent of $\epsilon$, the solution $ \pi_3 $ to (III) with boundary data $\bracket{ \gamma K_{\lambda, \epsilon} f \cdot \nu } \nu$ satisfies
				\begin{align} \label{eq_nabla2_resolvent_estimate_2}
					\LongNorm{				
						\nabla^2	 \ep{\nabla} \pi_3 
					}{\Lebesgue{q} \bracket{\Omega} }
					\leq C  \LongNorm{f}{\Lebesgue{q}\bracket{\Omega}}
				\end{align}
			for all $f \in \Lebesgue{q} \bracket{\Omega} $. 
		\end{proposition}
		\begin{proof}
			The estimate (\ref{eq_nabla2_resolvent_estimate_2}) is a direct consequence of (\ref{eq:C0550_discrete}), (\ref{eq_formula_nabla_epsilon_pi_3}),  \FurukawaCommentSeventh{Lemma \ref{lem:C0710} and Proposition \ref{prop_nabla2_resolvent_estimate}.}	
		\end{proof}
		}
	\subsection{Estimate for $v_2$}
		
	Let us consider the equation (II) with tangential boundary data $g = \bracket{ g_{+}, g_{-}}$. 
	Set
		\begin{align}
			y_{ \lambda, \epsilon}^{\prime} \bracket{n^{\prime}} 
			& = 2 s_{\lambda} \LongBracket{ I_2 + \frac{ \epsilon \abs{n^{\prime}} }{s_{\lambda}} \frac{ n^{\prime} \otimes n^{\prime} }{\abs{ n^\prime }^2} }, \notag\\ \label{eq:y_lambda_epsilon}
			y_{\lambda, \epsilon} \bracket{\Prime{n}} 
			& = 
			\left(
				\begin{array}{cc}
					\Prime{y}_{ \lambda, \epsilon } \bracket{n^{\prime}} & 0 \\
					0 & 0 
				\end{array}
			\right).
		\end{align}
	Then, $y_{ \lambda, \epsilon }$ satisfies
		\begin{align} \label{eq:C1505}
			\Prime{k}_{ \lambda, \epsilon } \bracket{ n^{\prime}, 0 } y_{\lambda, \epsilon} \bracket{\Prime{n}}
			= J_2
			:=
			\left(
				\begin{array}{cc}
					I_2 & 0 \\
					0 & 0				
				\end{array}
			\right),
		\end{align}
	\FurukawaCommentFifth{where $k^\prime_{\lambda , \epsilon}$ is defined by (\ref{eq:C0555}).}
	Let us define a multiplier operator $L_{\lambda, \epsilon}$ as
		\begin{align} 
			L_{\lambda, \epsilon} g (\Prime{n}, x_3)
			& = \ep{ \mathbb{P}}^{\Torus^2 \times \Real}  \FourierInverse_{d, n^{\prime}} \sLongBracket{
				e_{ \lambda }^{\prime} \bracket{ n^{\prime}, 1 - x_3 } y_{\lambda, \epsilon} \bracket{n^{\prime}} \Fourier_{d,  x^{\prime}} g_{+} \bracket{n^{\prime}} 
			}  \notag \\ \label{eq:C1506}
			& + \ep{ \mathbb{P}}^{\Torus^2 \times \Real} \FourierInverse_{d, n^{\prime}} \sLongBracket{ 
				e_{ \lambda }^{\prime} \bracket{ n^{\prime}, - 1 - x_3 } y_{\lambda, \epsilon} \bracket{n^{\prime}} \Fourier_{d, x^{\prime}} g_{-} \bracket{n^{\prime}} 
			},
		\end{align}
	where $e^\prime_\lambda$ is defined by (\ref{eq_definition_of_e_prime_lamda}).	
	Let $p_{\FurukawaCommentFourth{\epsilon}}^{\prime} \bracket{\Prime{n}, x_3 }$ be a partial Fourier transform of the symbol of $\ep{\mathbb{P}}^{\Torus^2 \times \Real}$ with respect to $\xi_3$.
	Then
		\begin{align}
			L_{\lambda, \epsilon} g (\Prime{n}, \cdot)
			& = \FourierInverse_{d, \Prime{n}} 
				\sLongBracket{
					\Prime{p}_{\epsilon} \bracket{ \Prime{n}, \cdot } \ast_{3} \Prime{e}_{\lambda} \bracket{ \Prime{n}, 1 - \cdot } \Prime{y}_{\lambda, \epsilon} \bracket{\Prime{n}} \Fourier_{d, \Prime{x}} g_{+} \bracket{\Prime{n}}
				} \notag \\ \label{eq:C1507}
			& + \FourierInverse_{d, \Prime{n}} 
				\sLongBracket{
					\Prime{p}_{\epsilon} \bracket{ \Prime{n}, \cdot } \ast_{3} \Prime{e}_{\lambda} \bracket{ \Prime{n}, - 1 - \cdot } \Prime{y}_{\lambda, \epsilon} \bracket{\Prime{n}} \Fourier_{d, \Prime{x}} g_{-} \bracket{\Prime{n}}
				},
		\end{align}
	where $ \cdot \ast_3 \cdot$ is convolution with respect to $x_3$. 
	We set
		\begin{align} \label{eq:C1509}
			W_{\lambda, \epsilon} = P_{N, \epsilon} L_{\lambda, \epsilon}.
	 	\end{align}
	Then, $W_{\lambda, \epsilon} g$ is a solution to (II) with boundary data $ \gamma W_{\lambda, \epsilon} g$.
	We first get \FurukawaComment{the} Fourier multiplier of $ \gamma W_{\lambda, \epsilon}$. 
	Next, we show the map $S_{\lambda,\epsilon} : g \mapsto \gamma W_{\lambda,  \epsilon} g$ has \FurukawaComment{a} bounded inverse for large $\lambda$. 
	Put
		\begin{align} \label{eq:C1510}
			V_{\lambda, \epsilon} g 
			= W_{\lambda, \epsilon} S^{-1}_{\lambda, \epsilon} g,
		\end{align}
	then, $V_{\lambda, \epsilon} g$ gives the solution to (II) with boundary data $g$. 
	

		\begin{proposition} \label{prop:C1500}
			\FurukawaCommentFifth{
			Let $r > 0$ be sufficiently large.
			Let $ 1 < q < \infty $, $0 < \epsilon \leq 1$, $s>0$, $0 < \theta < \pi/2 $ and $ \lambda \in \Sigma_{\theta} $.}
			Then, for $\abs{\lambda} > r $, there exists a bounded operator $R_{\lambda, \epsilon}$ from $\Lebesgue{q} (\Omega)$ into itself satisfying
				\begin{align}  \label{eq_existence_of_inverse}
					\LongNorm{R_{\lambda, \epsilon}}{
						\FurukawaCommentFifth{
						W^{s , q}_{\FurukawaCommentSixth{\mathrm{af}}} (\Torus^2)
						\rightarrow W^{s , q}_{\FurukawaCommentSixth{\mathrm{af}}} (\Torus^2)
						}
					} 
					\leq \frac{C}{\abs{\lambda}^{1/2}},
				\end{align}
			and \FurukawaCommentFifth{
			\begin{align} \label{eq_existence_of_inverse_2}
				\LongNorm{
					R_{\lambda, \epsilon}}{
						\FurukawaCommentFifth{
						W^{s , q}_{\FurukawaCommentSixth{\mathrm{af}}} (\Torus^2)
						\rightarrow W^{s + 1 , q}_{\FurukawaCommentSixth{\mathrm{af}}} (\Torus^2)
					}
				}
				\leq C,
			\end{align}
			} where $C>0$ is independent of $\epsilon$, such that 
				\begin{align} \label{eq:C1533}
					- S_{\lambda, \epsilon}^{-1} = I + R_{\lambda, \epsilon}.
				\end{align}
		\end{proposition}
		
		
		\begin{proof}
			 \FurukawaCommentFifth{Let $g \in C^\infty (\Torus^2)$ be horizontal average-free}.
			Since $e^\prime_{\lambda}$ is an even function with respect to $x_3$, we find from \FurukawaCommentFifth{the} change of variable that
				\begin{align*}
					\Prime{p}_{\epsilon} \bracket{\Prime{n}, \cdot} \ast_3 e^\prime_{\lambda} \bracket{\Prime{n}, 1 - \cdot} \notag
					& = \TimeInt{\Real}{}{
							\Prime{p}_{\epsilon} \bracket{\Prime{n}, \cdot - \zeta} e^\prime_{\lambda} \bracket{\Prime{n}, 1 - \zeta} 
						}{\zeta} \notag \\
					& = - \TimeInt{\Real}{}{
							 		\Prime{p}_{\epsilon} \bracket{\Prime{n}, \eta} e^\prime_{\lambda} \bracket{\Prime{n}, - 1 + \cdot - \eta} 
						 		}{\eta} \notag \\ 
					& = - \Prime{k}_{\lambda, \epsilon} \bracket{\Prime{n}, - 1 + \cdot },
				\end{align*}
			and similarly 
				\begin{align*}
					\Prime{p}_{\epsilon} \bracket{\Prime{n}, \cdot} \ast_3 e^\prime_{\lambda} \bracket{\Prime{n}, - 1 - \cdot}
					= - \Prime{k}_{\lambda, \epsilon} \bracket{\Prime{n}, 1 + \cdot } \FurukawaCommentFourth{.}
				\end{align*}
			Thus, we find from (\ref{eq:C1507}) that
				\begin{align}
					L_{\lambda, \epsilon} g (\Prime{n}, \FurukawaCommentFifth{x_3})
					& = \FourierInverse_{d, {\FurukawaCommentFifth{n^\prime}}} 
						\sLongBracket{
							- \Prime{k}_{\lambda, \epsilon} (\FurukawaCommentFifth{n^\prime}, -1 + \FurukawaCommentFifth{x_3})
							\Prime{y}_{\lambda, \epsilon} \bracket{\FurukawaCommentFifth{n^\prime}} \Fourier_{d, \Prime{x}} g_{+} \bracket{\FurukawaCommentFifth{n^\prime}}
						} \notag \\  \label{eq:C1520}
					& + \FourierInverse_{d, \FurukawaComment{\Prime{n}}} 
						\sLongBracket{
							- \Prime{k}_{\lambda, \epsilon} (\FurukawaCommentFifth{n^\prime}, 1 + \FurukawaCommentFifth{x_3})
							\Prime{y}_{\lambda, \epsilon} \bracket{\FurukawaCommentFifth{n^\prime}} \Fourier_{d, \Prime{x}} g_{-} \bracket{\FurukawaCommentFifth{n^\prime}}
						}\FurukawaCommentFourth{.}
				\end{align}
			We apply $P_{N, \epsilon}$ to (\ref{eq:C1520}) to get
				\begin{align}
					S_{\lambda, \epsilon} g 
					& = \gamma_{\pm} W_{\lambda, \epsilon} g \notag \\
					& = - 
					\FourierInverse_{d, \Prime{n}} 
						\sLongBracket{ \Prime{k}_{\lambda, \epsilon} \bracket{ \Prime{n}, - 1 \pm 1 } y_{\lambda, \epsilon} \bracket{\xi^{\prime}} 
					\Fourier_{d, \Prime{x}} g_{+} \bracket{\Prime{n}} } \notag \\
					& - 
					\FourierInverse_{d, n^{\prime}} 
						\sLongBracket{ 
							\alpha_{+, \epsilon} \bracket{ \Prime{n }, \pm 1
							} e_3 \cdot
							k_{ \lambda, \epsilon }^{\prime} 
							\bracket{
								n^{\prime}, 2 
							} y_{\lambda, \epsilon} \bracket{n^{\prime}} 
					\Fourier_{d, x^{\prime}} g_{-} \bracket{n^{\prime}} } \notag \\	
					& - 
					\FourierInverse_{d, \Prime{n}} 
						\sLongBracket{ 
							\Prime{k}_{\lambda, \epsilon} 
							\bracket{ 
								\Prime{n}, 1 \pm 1 
							} 
							y_{\lambda, \epsilon} \bracket{\xi^{\prime}} 
					\Fourier_{d, \Prime{x}} g_{-} \bracket{\Prime{n}} 
						} \notag \\ 
					& - 
					\FourierInverse_{d, n^{\prime}} 
						\sLongBracket{ 
							\alpha_{-, \epsilon} \bracket{\Prime{n}, \pm 1 } e_3 \cdot
							k_{ \lambda, \epsilon }^{\prime} \bracket{ n^{\prime}, - 2 } y_{\lambda, \epsilon} \bracket{n^{\prime}} 
					\Fourier_{d, x^{\prime}} g_{+} \bracket{n^{\prime}} 
						} \notag \\ \label{eq:C1525}
						& = I_1 + I_2 + I_3 + I_4.
				\end{align}
			Let us estimate $I_1$ and $I_3$.
			\FurukawaCommentFourth{The identity} (\ref{eq:C1505}) implies
				\begin{align} \label{eq:C1522}
					\FourierInverse_{d, \Prime{n}} 
						\sLongBracket{ 
							\Prime{k}_{\lambda, \epsilon} \bracket{ \Prime{n}, 0 } y_{\lambda, \epsilon} \bracket{n^{\prime}} 
					\Fourier_{d, \Prime{x}} g_{\pm} \bracket{\Prime{n}} 
						} 
					= g_{\pm}.
				\end{align}
			We need to show the other terms are $O(1 / \abs{\lambda}^{1/2})$.
			By (\ref{eq:C0555}) and (\ref{eq:y_lambda_epsilon}), we have
				\begin{align*}
					 & \Prime{k}_{\lambda, \epsilon} \bracket{ \Prime{\xi}, \pm 2 } 
					 y_{\lambda, \epsilon} \bracket{ \Prime{n} } \notag \\
				 	 & =
					 \Napier^{- s_{\lambda}} 
					 	\LongBracket{ 
						 	J_2 + \frac{ \epsilon \abs{\Prime{n}} }{ s_{\lambda} } \frac{ J_2 n \otimes J_2 n }{ \abs{\Prime{n}}^2 }
						} \notag \\ 
						& \, -
					 J_2 n \otimes J_2 n
					 \frac{ \epsilon^2 }{ \lambda + \bracket{ 1 - \epsilon^2 } \abs{\Prime{n}}^2 } 
					 	\Napier^{  - 2 s_{\lambda} } 
					  \LongBracket{ 
					  	 J_2 + 
					  	\frac{ \epsilon \abs{\Prime{n}} }{ s_{\lambda} } \frac{  J_2 n \otimes  J_2 n }{ \abs{\Prime{n}}^2 } 
					  } \\
					  & \, -
					 J_2 n \otimes J_2 n 
					 \frac{ \epsilon^2 }{ \lambda + \bracket{ 1 - \epsilon^2 } \abs{\Prime{n}}^2 } 
					 \frac{ 
							\Napier^{ - 2 \epsilon \abs{\Prime{n}}  }  
						}{ 
							\epsilon \abs{\Prime{n}} 
						} s_{\lambda}
					  \LongBracket{ 
					  	 J_2 + 
					  	\frac{ \epsilon \abs{\Prime{n}} }{ s_{\lambda} } \frac{  J_2 n \otimes  J_2 n }{ \abs{\Prime{n}}^2 } 
					  } \notag \\
						& =: II_1 + II_2 + II_3. \notag
				\end{align*}
			We find from (\ref{eq:C0501}) and (\ref{eq:C0504}) in Proposition \ref{prop:C05} and 
			\begin{align} \label{eq_C1530_multiplier}
				\sLongBracket{
					\frac{
						\abs{\xi^\prime}
					}{
						s_\lambda
					}
				}_{\mathcal{M}^\prime}
				+ \sLongBracket{
					\frac{
						J_2 \xi \otimes J_2 \xi
					}{
						\abs{\xi^\prime}^2
					}
				}_{\mathcal{M}^\prime}
				\leq C
				\FurukawaCommentFifth{, \quad\xi = (\xi^\prime , \xi_3) \in \Real^3,}
			\end{align}
		 	that
				\begin{align} \label{eq:boundedness_Mikhlin_constant_A_1}
					\sLongBracket{
						II_1
					}_{\mathcal{M}^\prime}
					\leq C e^{ - c \abs{\lambda}^\frac{1}{2}},
					\FurukawaCommentFifth{
					\quad \sLongBracket{
						\, \abs {\xi^\prime} \, II_1
					}_{\mathcal{M}^\prime}
					\leq C e^{ - c \abs{\lambda}^\frac{1}{2}},
					}
				\end{align}
			\FurukawaCommentSixth{where we interpret that the multiplier $II_1$ is extended from $\Integer^2$ to $\Real^2$ in the canonical way.}
			\FurukawaComment{
			Since 
				\begin{align} \label{eq_C1530_multiplier_2}
					\sLongBracket{
					\frac{
						1
					}{
						\lambda + (1 - \epsilon^2) \abs{\xi^\prime}^2
					}
				}_{\mathcal{M}^\prime}
				\leq \frac{C}{\abs{\lambda}},
				\end{align}
			by the same way as above we find
			}
				\begin{align} \label{eq:boundedness_Mikhlin_constant_A_2}
					\sLongBracket{
						II_2
					}_{\mathcal{M}^\prime}
					\leq \frac{C e^{- c \abs{\lambda}^{1/2}}}{\abs{\lambda}}, \quad
					\FurukawaCommentFifth{
					\sLongBracket{
						\, \abs {\xi^\prime} \, II_2
					}_{\mathcal{M}^\prime}
					\leq C e^{ - c \abs{\lambda}^\frac{1}{2}},}
				\end{align}
			for \FurukawaComment{$\lambda \in \Sigma_\theta$}, where constants $c, C > 0$ are independent of $\epsilon$.
			Note that \FurukawaCommentFifth{$II_3$} \FurukawaComment{has a} little bit problem near $\epsilon = 0$ since, at this point, we can not use \FurukawaCommentFourth{the} decay of $\Napier^{- 2 \epsilon \abs{\Prime{\xi}}}$ to obtain uniform boundedness of the Mikhlin constant. 
			However, we can use \FurukawaCommentFourth{the} decay of $1/ \bracket{ \lambda + (1 - \epsilon^2) \abs{\Prime{\xi}}^2}$ around $\epsilon = 0$.
			On the other hand, when $\epsilon$ is away from $0$, we have no problem \FurukawaComment{to use decay of $e^{- 2 \epsilon \abs{\xi^\prime}}$}.
			Thus, combining \FurukawaComment{this} observation with \FurukawaCommentSeventh{Proposition \ref{prop_scaled_multiplier}, (\ref{eq_C1530_multiplier}) and (\ref{eq_C1530_multiplier_2})}, we conclude that
				\begin{align} \label{eq:boundedness_Mikhlin_constant_A_3}
					\sLongBracket{
						II_3
					}_{\mathcal{M}^\prime}
					\leq \frac{C}{\abs{\lambda}^\frac{1}{2}}, \quad
					\FurukawaCommentFifth{
					\sLongBracket{
						\, \abs{\xi^\prime} \, II_3
					}_{\mathcal{M}^\prime}
					\leq C,}	
				\end{align}
			where $C>0$ is independent of $\epsilon$.
			Thus we find from (\ref{eq:boundedness_Mikhlin_constant_A_1}), (\ref{eq:boundedness_Mikhlin_constant_A_2}) and \FurukawaComment{(\ref{eq:boundedness_Mikhlin_constant_A_3}) }that \FurukawaCommentFifth{
				\begin{align} 
					& 
					\LongNorm{
						I_1 + I_3 - g_+ - g_-
					}{L^q (\Torus^2)} \notag \\
					& \leq  \LongNorm{
						\FourierInverse_{d, \Prime{n}} 
							\sLongBracket{ 
								\Prime{k}_{\lambda, \epsilon} \bracket{ \Prime{n}, 2 } y_{\lambda, \epsilon} \bracket{n^{\prime}} 
						\Fourier_{d, \Prime{x}} g \bracket{\Prime{n}} 
							} 
					}{\FurukawaCommentFifth{L^q (\Omega)}} \notag \\
					& + \LongNorm{
						\FourierInverse_{d, \Prime{n}} 
							\sLongBracket{ 
								\Prime{k}_{\lambda, \epsilon} \bracket{ \Prime{n}, - 2 } y_{\lambda, \epsilon} \bracket{n^{\prime}} 
						\Fourier_{d, \Prime{x}} g \bracket{\Prime{n}} 
							} 
					}{\FurukawaCommentFifth{L^q (\Omega)}} \notag \\
					& \leq \frac{C}{\abs{\lambda}^\frac{1}{2}} \LongNorm{
						g
					}{
						L^{\FurukawaCommentFifth{q}} (\FurukawaCommentFifth{\Torus^2})
					}.\label{eq:C1531}
				\end{align}
				}
			Next, we estimate $I_2$ and $I_4$.
			It follows from (\ref{eq:C0555}) that
 				\begin{align}
					& e_3 \cdot
							k_{ \lambda, \epsilon }^{\prime} \bracket{ n^{\prime}, \pm 2 } y_{\lambda, \epsilon} \bracket{n^{\prime}} 
					\Fourier_{d, x^{\prime}} g_{\mp} \bracket{n^{\prime}} \notag \\
					& = e_3 \cdot \sLongBracket{
						\frac{e^{- s_\lambda}}{2 s_\lambda} y_{\lambda, \epsilon} (\Prime{n}) \Fourier_{d, \Prime{x}} g_\mp
						- \left( 
						\begin{array}{cc}
							\eta_{\lambda, \epsilon}^{\prime} \bracket{ n^{\prime}, \pm 2 }
							n^{\prime} \otimes n^{\prime} 
							\Prime{y}_{\lambda, \epsilon} (\Prime{n})
							& 0 \\
							- \partial_3 \eta_{\lambda, \epsilon}^{\prime} \bracket{ n^{\prime}, \pm 2 }
							i { n^{\prime} }^T
							\Prime{y}_{\lambda, \epsilon} (\Prime{n})
							& 0
						\end{array}
						\right)
						\Fourier_{d, \Prime{x}} g_\mp
					} \notag \\ \label{eq_e_k_y_F_g}
					& = - \left( 
						\begin{array}{ccc}
							- \partial_3 \eta_{\lambda, \epsilon}^{\prime} \bracket{ n^{\prime}, \pm 2}
							i { n^{\prime} }^T
							\Prime{y}_{\lambda, \epsilon} (\Prime{n})
							& 0
						\end{array}
						\right)
						\Fourier_{d, \Prime{x}} g_\mp.
				\end{align}
			Recall 
				$ \partial_3 \eta_{\lambda, \epsilon} (\Prime{n}, \pm 2) = \frac{
							\epsilon^2 
						 }{ 
						 	\lambda 
						 	+ \bracket{ 1 - \epsilon^2 } \abs{ n^{\prime} }^2 
						 } \frac{
						 	\Napier^{- 2 s_{\lambda}} - \Napier^{ - 2 \epsilon \abs{n^{\prime}} }
						 }{ 
						 	2 
						 } 
				$.
			Then, we find from \FurukawaCommentSeventh{the first inequality of (\ref{eq:C0501})} and \FurukawaComment{(\ref{eq_C1530_multiplier_2})} that
				\begin{align*}
					& \sLongBracket{
						\LongBracket{
							1 + \epsilon \abs{\Prime{\xi}}
						}
						\partial_3 \eta_{\lambda, \epsilon} (\Prime{\xi}, \pm 2)
						\Prime{y}_{\lambda} (\Prime{\xi})
					}_{\mathcal{M}^\prime} \notag \\
					& = \sLongBracket{
						\frac{
							\epsilon^2 
						 }{ 
						 	\lambda 
						 	+ \bracket{ 1 - \epsilon^2 } \abs{ \Prime{\xi} }^2 
						 }
						\LongBracket{
							1 + \epsilon \abs{\Prime{\xi}}
						} 
						s_{\lambda} 
						 \LongBracket{ 
						 	\Napier^{- 2 s_{\lambda}} 
						 	- \Napier^{ - 2 \epsilon \abs{\xi^{\prime}} }
						 	}
						 	\LongBracket{ 
						 		I_2 + \frac{ \epsilon \abs{\Prime{\xi}} }{s_{\lambda}} \frac{ \Prime{\xi} \otimes \Prime{\xi} }{\abs{ \Prime{\xi} }^2} 
						 	}
					}_{\mathcal{M}^\prime} \notag \\
					& \leq \frac{C}{\abs{\lambda}^\frac{1}{2}},
				\end{align*}
			\FurukawaCommentFifth{
			and
			\begin{align} \label{eq_regularity_R_epsilon_lamda}
				\sLongBracket{
					\abs{\xi^\prime}
						\LongBracket{
							1 + \epsilon \abs{\Prime{\xi}}
						}
						\partial_3 \eta_{\lambda, \epsilon} (\Prime{\xi}, \pm 2)
						\Prime{y}_{\lambda} (\Prime{\xi})
					}_{\mathcal{M}^\prime} 
				\leq C,
			\end{align}}
			\hspace{-1.5mm}where $C>0$ is independent of $\epsilon$.
			\FurukawaCommentSeventh{The formula (\ref{eq:C0712}), estimates (\ref{eq:C0504}) and (\ref{eq:C0505})} lead \FurukawaCommentSixth{to}
				\begin{align*}
						\sLongBracket{ 
							\frac{
								\alpha_{\pm, \epsilon} \bracket{\Prime{\xi}, \pm \FurukawaCommentFifth{1}} \epsilon \abs{\Prime{\xi}} 
							}{
								1 + \epsilon \abs{\Prime{\xi}} 
							}
						}_{\mathcal{M}^\prime} 
						< \infty,
				\end{align*}
		uniformly on $\epsilon$.
		\FurukawaCommentSixth{We find from Proposition \ref{prop_Fourier_Multiplier_Theorem_Discrete} that}
				\begin{align}
						& \FurukawaCommentFifth{\LongNorm{
							I_2 + I_4 
						}{L^q (\Torus^2)}} \notag \\
						& \FurukawaCommentFifth{\leq} \LongNorm{
							\FourierInverse_{d, \Prime{n}} 
								\alpha_{+, \epsilon} \bracket{\Prime{n}, \pm 1 } e_3 \cdot
								k_{ \lambda, \epsilon }^{\prime} \bracket{ \Prime{n},  2 } 
								y_{\lambda, \epsilon} \bracket{\Prime{n}} 
							\Fourier_{d, x^{\prime}} g_{-}
						}{\Lebesgue{q} (\Torus^2)} \notag \\
						& + \LongNorm{
							\FourierInverse_{d, \Prime{n}} 
								\alpha_{-, \epsilon} \bracket{\Prime{n}, \pm 1 } e_3 \cdot
								k_{ \lambda, \epsilon }^{\prime} \bracket{ \Prime{n}, - 2 } 
								y_{\lambda, \epsilon} \bracket{\Prime{n}} 
							\Fourier_{d, x^{\prime}} g_{+}
						}{\Lebesgue{q} (\Torus^2)} \notag \\ \label{eq:C1532}
						& \leq \frac{
							C
						}{ 
							\abs{\lambda}^\frac{1}{2} 
						} \LongNorm{
							g
						}{
							L^q (\Torus^2)
						},
				\end{align}
			where $C$ is independent of $\epsilon$. 
			Thus, taking $\abs{\lambda}$ sufficiently large, clearly \FurukawaCommentFourth{the} choice of $\lambda$ is also independent of $\epsilon$, we can conclude by \FurukawaComment{(\ref{eq:C1525}), (\ref{eq:C1522}\FurukawaCommentFourth{)}, (\ref{eq:C1531}) and (\ref{eq:C1532}) }that
				\begin{align*}
					- S_{\lambda,\epsilon} 
					= I + O \bracket{ \abs{\lambda}^{ - 1/2} }.
				\end{align*}
			\FurukawaCommentFifth{By the Neumann series argument we obtain (\ref{eq_existence_of_inverse}) for $s = 0$.
			Moreover, we find from (\ref{eq:boundedness_Mikhlin_constant_A_1}), (\ref{eq:boundedness_Mikhlin_constant_A_2}), (\ref{eq:boundedness_Mikhlin_constant_A_3}) and (\ref{eq_regularity_R_epsilon_lamda}) that (\ref{eq_existence_of_inverse_2}) holds for $s = 0$.
			Since $\partial_j$ ($j = 1,2$) commutes Fourier multiplier operators, we obtain (\ref{eq_existence_of_inverse}) and (\ref{eq_existence_of_inverse_2}) for $s > 0$. 
			}
		\end{proof}
	


	\begin{proposition} \label{prop:C1507}
		Let \FurukawaCommentFifth{$1 < q < \infty$, $0 < \epsilon \leq 1$,} $0 < \theta < \pi / 2$ and $\lambda \in \Sigma_\theta$.
		Then there exist $r > 0$ and a constant $C > 0$, which is independent of $\epsilon$ and $\lambda$, if $\abs{\lambda} \geq r$, $V_{\lambda , \epsilon}$ defined by (\ref{eq:C1510}) satisfies
		\begin{align} \label{eq:C1534}
			\LongNorm{V_{\lambda, \epsilon} g}{\Lebesgue{q} \bracket{\Omega}}
			\leq C \abs{\lambda}^{- 1 / 2q} \LongNorm{g}{\Lebesgue{q}  \bracket{\partial \Omega}}
		\end{align}
		for all \FurukawaCommentSixth{$g \in \Lebesgue{q} \bracket{\partial \Omega}$ satisfying $\tilde{g} = 0$.}
	\end{proposition}
	\begin{proof}
		We take $r>0$ so that $R_{\lambda, \epsilon}$ exists.
		Then $S_{\lambda, \epsilon}^{-1}$ is bounded on $L^q (\Omega)$.
		\FurukawaCommentSeventh{
		 We find from the resolvent estimate for the Dirichlet Laplacian on $\Omega$, see Lemma 5.3 in \cite{Abels2002}, and Proposition \ref{prop:C1500} that
		 \begin{align*}
				\LongNorm{ 
					L_{\lambda, \epsilon} S_{\lambda, \epsilon}^{-1} g
				}{L^q (\Omega)} 
				\leq C \abs{\lambda}^{- 1 / 2 q} \LongNorm{ 
					g 
				}{L^q (\Omega)},
			\end{align*}
			where $C>0$ is independent of $\epsilon$.
		By Corollary \ref{cor_bound_P_N_epsilon}, we obtain (\ref{eq:C1534}).
		}

	\end{proof}

	\begin{proposition} \label{prop:C1510}
		Let $1 < q < \infty$ and $0 < \epsilon \leq 1$.
		Then there exists a constant $C > 0$, which is independent of $\epsilon$, such that
			\begin{align} \label{eq:C1535}
				\LongNorm{	
					\FourierInverse_{d, \Prime{n}}		
						\frac{ 1 + \epsilon \abs{\Prime{n}} }{ \epsilon \abs{\Prime{n}} } 
						\LongBracket{ e_3 \cdot 
					\Fourier_{\FurukawaCommentFifth{d,\Prime{x}}} 
					\LongBracket{ \ep{\mathbb{P}}^{\Torus^2 \times \Real} E_0 f }}
				}{\Lebesgue{q} (\Omega)}
				\leq C 
				\LongNorm{f}{\Lebesgue{q} (\Omega)}
			\end{align}
		for all \FurukawaCommentFifth{$f \in \Lebesgue{q}_{\FurukawaCommentSixth{\mathrm{af}}} (\Omega )$}.
	\end{proposition}


	\begin{proof}
		Since the symbol have poles at $\xi_3 = \pm i \epsilon \abs{\Prime{n}}$, we obtain its partial Fourier transform with  respect to $\xi_3$ by the residue theorem.
		Thus, we have
			\begin{align*}
				& \FourierInverse_{d, \Prime{n}}		
					\frac{ 1 }{ \epsilon \abs{\Prime{n}} } 
					\LongBracket{ e_3 \cdot 
				\Fourier_{\FurukawaCommentFifth{d,\Prime{x}}} 
				\LongBracket{ \ep{\mathbb{P}}^{\Torus^2 \times \Real}  E_0 f }} \notag \\
				& = \frac{1}{2} \FourierInverse_{d, \Prime{n}}
				\int_{-1}^{1}
					\frac{ 1 }{ \epsilon \abs{\Prime{n}} }
					\left[
						\Napier^{- \abs{x_3 - \zeta} \epsilon \abs{\Prime{n}} } 
						i \epsilon \Prime{n} \cdot 
						\Fourier_{d, \Prime{x}} \Prime{f} \bracket{\Prime{n}, \zeta}
					\right. \\
					& \quad \quad \quad \quad \quad \quad \left.
					+ \Napier^{- \abs{x_3 - \zeta} \epsilon \abs{\Prime{n}} } 
						\epsilon \abs{\Prime{n}}
						\Fourier_{d, \Prime{x}} f_3 \bracket{\Prime{n}, \zeta}
					\right]
				d \zeta.
			\end{align*}
		\FurukawaCommentSeventh{This formula and Proposition \ref{prop_Fourier_Multiplier_Theorem_Discrete} imply}
			\begin{align*}
				\LongNorm{
					\FourierInverse_{d, \Prime{n}}		
					\frac{ 1 }{ \epsilon \abs{\Prime{n}} } 
					\LongBracket{ e_3 \cdot 
				\Fourier_{d, \Prime{x}} 
				\LongBracket{ \ep{\mathbb{P}}^{\Torus^2 \times \Real} E_0 f }}
				}{
					L^q (\Torus^2)
				}
				\leq C \LongNorm{
					f
				}{
					L^q (\Real^2)
				},
			\end{align*}
		where $C$ is independent of $\epsilon$.
		\FurukawaCommentFifth{
		Combining \FurukawaCommentSixth{this estimate with} the boundedness of $\mathbb{P}_\epsilon^{\Torus^2 \times \Real}$, we obtain (\ref{eq:C1535}).
		}
	\end{proof}


	
	Let us show BIP for the solution operator for \FurukawaComment{the equation (II)}.
	

		\begin{proposition} \label{prop:C1515}
			Let \FurukawaCommentFifth{$1 < q < \infty$, $0 < \epsilon \leq 1$, $0 < \theta < \pi /2$,} $\lambda \in \Sigma_\theta$,  $0 < a < 1/2$, \FurukawaCommentFifth{ and $z$ satisfying $ - a < \RealPart z <0 $.}
			Then there exists a constant $C = C \bracket{ q, a, \theta }$, it holds that
				\begin{align} \label{eq:C1550}
					\LongNorm{ \frac{
						1
					}{ 
						2 \pi i 
					} \TimeInt{
							\Gamma_\theta
						}{}{ 
							\bracket{ 
								- \lambda
							}^z V_{\lambda, \epsilon} \sLongBracket{ 
								\gamma v_1 - \bracket{ \gamma v_1 \cdot \nu } \nu 
							} 
						}{\lambda} 
					}{\Lebesgue{q}(\Omega)} 
					\leq C \Napier^{ \abs{ \ImaginaryPart z } \theta } \LongNorm{
						f
					}{
						\Lebesgue{q} (\Omega)
					}
				\end{align}
			\FurukawaCommentFourth{for all $f \in \Lebesgue{q} (\Omega)$, where $ v_1 = K_{\lambda, \epsilon} f$.}
		\end{proposition}
		
		
		\begin{proof}
			\FurukawaCommentFifth{In view of Remark \ref{rem_average_free}, we \FurukawaCommentSixth{may} assume $\tilde{f} = 0$ without loss of generality.}
			It holds \FurukawaComment{by (\ref{eq:C0550_discrete}) that}
				\begin{align*}
					\gamma v_1 - \bracket{ \gamma v_1 \cdot \nu } \nu 
					= \gamma K_{\lambda, \epsilon} E_0 f 
					- \gamma \Pi_{\epsilon} K_{\lambda, \epsilon} E_0 f.
				\end{align*}
			We find from this formula, \FurukawaCommentSeventh{(\ref{eq:C0712}), (\ref{eq:C1506}), (\ref{eq:C1509}), (\ref{eq:C1510}) and (\ref{eq:C1533})} that the integrand of the left hand side of (\ref{eq:C1550}) can be essentially written as
				\begin{align}
					& P_{N, \epsilon} \ep{\mathbb{P}}^{\Torus^2 \times \Real}
					\FourierInverse_{d, \Prime{n}}
						\int_{-1}^{1}
							\Prime{e}_{\lambda} \bracket{\Prime{\xi}, \pm 1 - x_3}
							y_{\lambda, \epsilon} \bracket{\Prime{n}}
							\Prime{e}_{\lambda} \bracket{\Prime{\xi}, \pm 1 - \zeta} \notag \\
				& \quad \quad \quad \quad 	\quad \quad \quad \quad 
					\times \Fourier_{d, \Prime{x}} \LongBracket{ \ep{\mathbb{P}}^{\Torus^2 \times \Real} E_0 f } \bracket{\Prime{n}, \zeta}
						d \zeta \notag \\
				+ & P_{N, \epsilon} \ep{\mathbb{P}}^{\Torus^2 \times \Real}
					\FourierInverse_{d, \Prime{n}}
						\int_{-1}^{1}
							\Prime{e}_{\lambda} \bracket{\Prime{n}, \pm 1 - x_3}
							y_{\lambda, \epsilon} \bracket{\Prime{n}}
							\alpha_{\epsilon, \pm} \bracket{\Prime{n}, \pm 1} 
							\Prime{e}_{\lambda} \bracket{\Prime{n}, \pm 1 -\zeta} \notag \\
				&	 \quad\quad\quad\quad\quad\quad\quad\quad	
							\times \frac{\epsilon \abs{\Prime{n}}}{ 1 + \epsilon \abs{\Prime{n}} }
							\frac{1 + \epsilon \abs{\Prime{n}} }{ \epsilon \abs{\Prime{n}}}
					\Fourier_{d, \Prime{x}} \LongBracket{  \ep{\mathbb{P}}^{\Torus^2 \times \Real} E_0 f } \bracket{\Prime{n}, \zeta}
						d \zeta \notag \\
				+ & W_{\lambda, \epsilon} R_{\lambda, \epsilon} 
					\sLongBracket{
							\gamma K_{\lambda, \epsilon} E_0 f 
							- \gamma \Pi_{\epsilon} K_{\lambda, \epsilon} E_0 f
					}
					 \notag \\
				=: & I_1 + I_2 + I_3, \label{eq_formula_V_lambda_epsilon_tangential_term}
				\end{align}
			where $\pm$ should be take properly. 
			It follows from \FurukawaCommentSeventh{(\ref{eq:C0501}) and (\ref{eq_C1530_multiplier})} that
				\begin{align} \label{eq:C1560}
					& \sLongBracket{
						\Prime{e}_{\lambda} \bracket{\FurukawaCommentFifth{\FurukawaCommentSixth{\xi}^\prime}, \pm 1 - x_3 }
						y_{\lambda, \epsilon} \bracket{\FurukawaCommentFifth{\FurukawaCommentSixth{\xi}^\prime}}
						\Prime{e}_{\lambda} \bracket{\FurukawaCommentFifth{\FurukawaCommentSixth{\xi}^\prime}, \pm 1 - \zeta }
					}_{\Prime{\mathcal{M}}} \notag \\
					& = 2 \sLongBracket{
							e^{ \FurukawaCommentFifth{-} \abs{\pm 1 - x_3} s_\lambda} 
						\LongBracket{
							I_2 
							+ \frac{ 
								\epsilon \abs{\FurukawaCommentFifth{\FurukawaCommentSixth{\xi}^\prime}} 
							}{
								s_{\lambda}
							} \frac{ 
								\FurukawaCommentFifth{\FurukawaCommentSixth{\xi}^\prime} \otimes \FurukawaCommentFifth{\FurukawaCommentSixth{\xi}^\prime} 
							}{
								\abs{\FurukawaCommentFifth{\FurukawaCommentSixth{\xi}^\prime}}^2
							} 
						}
						\frac{
							e^{ \FurukawaCommentFifth{-} \abs{\pm 1 - \zeta} s_\lambda}
						}{
							s_\lambda
						}
					}_{\mathcal{M}^\prime} \notag \\
					& \leq C 
					\frac{
						\Napier^{- c \abs{\lambda}^{1/2} \bracket{ \abs{ \pm 1 - x_3 } + \abs{ \pm 1 - \zeta } }}
					}{\abs{\lambda}^{1/2}}.
				\end{align}
			\FurukawaCommentFifth{Let $R>0$ be large enough so that $S_{\lambda, \epsilon}^{-1}$ in Proposition \ref{prop:C1500} exists.
			Then we find from the change of integral curve  around the origin to ensure $\abs{\lambda} > R$ and Proposition \ref{prop_Fourier_Multiplier_Theorem_Discrete} that}
				\begin{align*}
					& \LongNorm{ \frac{1}{ 2 \pi i}
						 \TimeInt{\Gamma}{}{ 
							 \bracket{- \lambda}^z I_1
							 }{\lambda}
					 }{\Lebesgue{q} (\FurukawaCommentFifth{\Omega)}} \notag \\
					& \leq C \LongNorm{ 
							\TimeInt{-1}{1}{
								\TimeInt{\Gamma}{}{	\abs{ \lambda^z}
										 \frac{ \Napier^{ - c \abs{\lambda}^{1/2} \bracket{ \abs{ x_3 - a}  + \abs{ \zeta - b } } } }{ \abs{\lambda}^{1/2} } \LongNorm{
										\mathbb{P}_\epsilon^{\FurukawaCommentFifth{\Torus^2 \times \Real}} E_0 f \bracket{\cdot, \zeta} }{ \Lebesgue{q} \bracket{\Torus^2} }
								}{\lambda}
							}{\zeta}
							}{\Lebesgue{q} \bracket{-1, 1} } \notag \\
					& \leq C \FurukawaComment{e^{\theta \abs{\ImaginaryPart z}}}  \LongNorm{f}{\Lebesgue{q} \bracket{\Omega}} \notag \\
					& + C \left \Vert
						\int_{-1}^{1}
								\int_{R}^{\infty}
									\Napier^{ \theta \abs{ \ImaginaryPart z } }
									r^{ \RealPart z - 1/2 } \Napier^{ - c r^{1/2} \bracket{ \abs{ x_3 - a}  + \abs{ \zeta - b } } } \LongNorm{f \bracket{\cdot, \zeta}}{\Lebesgue{q} \bracket{\Torus^2}}
							dr
						d \zeta
					\right \Vert_{\Lebesgue{q} \bracket{-1, 1} } \\
					& \leq C \FurukawaComment{e^{\theta \abs{\ImaginaryPart z}}} \LongNorm{
						f
					}{L^q (\Omega)} 
					+ C_R e^{\theta \abs{\ImaginaryPart z}} \LongNorm{
						\int_{-1}^{1}
							\frac{\LongNorm{f (\cdot , \zeta)}{L^q (\Torus^2)}}{\abs{ x_3 - a}  + \abs{ \zeta - b }}
						d r
					}{L^q (\Omega)}
				\end{align*}
			\FurukawaComment{for some $a , b \in \{ -1 , 1 \}$,} where \FurukawaCommentSeventh{$C$ and $C_R$ are} independent of $\epsilon$. 
			Applying \FurukawaCommentSeventh{Proposition} \ref{prop:C03}, we obtain
				\begin{align} \label{eq:C1570}
					\LongNorm{ 
						 \TimeInt{\Gamma}{}{ 
							 \bracket{- \lambda}^z I_1
							 }{\lambda}
					 }{\Lebesgue{q} \bracket{\Omega} }
					\leq C \Napier^{ \theta \abs{\ImaginaryPart z} } \LongNorm{f}{\Lebesgue{q} \bracket{\Omega}} \FurukawaCommentFourth{.}
				\end{align}
			\FurukawaCommentSeventh{It follows from (\ref{eq_definition_of_e_prime_lamda}), (\ref{eq:C0712}), (\ref{eq:y_lambda_epsilon}) and Proposition \ref{prop:C05} that}
				\begin{align}
					&\sLongBracket{
						\Prime{e}_{\lambda} \bracket{\FurukawaCommentFifth{\Prime{\FurukawaCommentSixth{\xi}}}, \pm 1 - x_3}
							y_{\lambda, \epsilon} \bracket{\FurukawaCommentFifth{\Prime{\FurukawaCommentSixth{\xi}}}}
							\alpha_{\epsilon, \pm} \bracket{\FurukawaCommentFifth{\Prime{\FurukawaCommentSixth{\xi}}}, \pm 1} 
							\Prime{e}_{\lambda} \bracket{\FurukawaCommentFifth{\Prime{\FurukawaCommentSixth{\xi}}}, \pm 1 -\zeta}
							\frac{\epsilon \abs{\FurukawaCommentFifth{\Prime{\FurukawaCommentSixth{\xi}}}}}{ 1 + \epsilon \abs{\FurukawaCommentFifth{\Prime{\FurukawaCommentSixth{\xi}}}} }
					}_{\Prime{\mathcal{M}}} \notag \\  \label{eq_Resolvent_estimate_2}
					& \leq C 
					\frac{
						\Napier^{- c \abs{\lambda}^{1/2} \bracket{ \abs{ \pm 1 - x_3 } + \abs{ \pm 1 - \zeta } }}
					}{\abs{\lambda}^{1/2}}. 
				\end{align}							
			Thus we find from Proposition \ref{prop:C1510} \FurukawaCommentFifth{that}
				\begin{align*}
					& \LongNorm{ \frac{1}{2 \pi i }
						\TimeInt{\Gamma}{}{
							 \bracket{ - \lambda }^z I_2
						}{\lambda}
					}{
						\Lebesgue{q} \bracket{\Omega}
					} \\
					& \leq C \Napier^{ \theta \abs{\ImaginaryPart z} } 
					\LongNorm{
						\FourierInverse_{\FurukawaCommentFifth{d,} \Prime{n}}
						\frac{
							1 + \epsilon \abs{\FurukawaComment{\Prime{n}}} 
						}{ 
							\epsilon \abs{\Prime{n}}
						}
						\Fourier_{\FurukawaCommentFifth{d,} \Prime{x}} \LongBracket{
							 \mathbb{P}_\epsilon^{\FurukawaCommentFifth{\Torus^2 \times \Real}} E_0 f 
						} \bracket{\Prime{n}, \zeta}
					}{
						\Lebesgue{q}\bracket{\Omega}
					} \notag \\
					& \leq C \Napier^{ \theta \abs{\ImaginaryPart z} } 
					\LongNorm{
						f
					}{\Lebesgue{q} \bracket{\Omega} }.
				\end{align*}
			\FurukawaCommentSeventh{By Proposition \ref{prop:C1500}, the trace theorem, Lemma \ref{lem:C0710} and the resolvent estimate for the Laplace operator on $\Torus^2 \times \Real$, we have
			\begin{align*}
					& \LongNorm{ R_{\lambda ,\epsilon}
					 \gamma	\sLongBracket{
							K_{\lambda, \epsilon} E_0 f 
							- \gamma \Pi_{\epsilon} K_{\lambda, \epsilon} E_0 f
					}
					}{\Lebesgue{q} \bracket{ \partial \Omega}}  \\
					& \leq C \abs{\lambda}^{ -3/2 + 1/ 2 q + \delta} 
					\LongNorm{
						f
					}{\Lebesgue{q} \bracket{ \Omega}} 
				\end{align*}
			for some small $\delta>0$.
			 The resolvent estimate for the Dirichlet Laplacian on $\Omega$, see Lemma 5.3 in \cite{Abels2002}, \FurukawaCommentSeventh{and Lemma \ref{lem:C0710}} imply
			 \begin{align*}
					\LongNorm{
						P_{N, \epsilon} L_{\lambda, \epsilon} 
					}{\Lebesgue{q} \bracket{\partial  \Omega} \rightarrow \Lebesgue{q} \bracket{\Omega}}
					 \leq C \abs{\lambda}^{- 1 /2q } 
				\end{align*}
			for some small $\delta>0$. 
			We find from the above \FurukawaCommentSeventh{two inequalities}
			\begin{align} \label{eq:C1575}
			\LongNorm{I_3}{\Lebesgue{q} \bracket{\Omega}}
					& \leq C \abs{\lambda}^{- 3/2 + \delta} 
					\LongNorm{
						f
					}{\Lebesgue{q} \bracket{ \Omega}}.
				\end{align}
			}
			Thus \FurukawaCommentThird{we find from the change of integral line around the origin that}
				\begin{align*}
					\LongNorm{
						\frac{1}{2 \pi i} 
						\TimeInt{\Gamma}{}{
							\bracket{ - \lambda}^z I_3
						}{\lambda}
					}{\Lebesgue{q} \bracket{\Omega}} 
					\leq C \LongNorm{f}{\Lebesgue{q} \bracket{\Omega}},
				\end{align*}
		\end{proof}
		\FurukawaCommentSeventh{where $C>0$ is independent of $\epsilon$.}
		\begin{proof}[Proof of Lemma \ref{lem:C02}]
			Lemma \ref{lem:C02} is a direct consequence of \AmruComment{Propositions \ref{prop:C0510}, \ref{prop:C0715} and \ref{prop:C1515}}.
		\end{proof}
			\FurukawaCommentSixth{We next prove Lemma \ref{lem:A1001} from Lemma \ref{lem:C02}.
		For this purpose we need further uniform estimate for the resolvent to compare $\Vert \nabla^2 u \Vert_{L^q (\Omega)}$ and $\Vert A_\epsilon  u \Vert_{L^q (\Omega)}$.
		For resolvent estimates we begin with
		}
		\FurukawaCommentFifth{
		\begin{proposition} \label{prop_nabla_2_V_lambda_epsilon}
			Let $1 < q < \infty$, $0 < \epsilon \leq 1$, $0 < \theta < \pi /2$.
			Let $\lambda \in \Sigma_\theta$ be sufficiently large so that $S_{\lambda, \epsilon}^{-1}$ exists in Proposition \ref{prop:C1500}.
			Then there exists a constant $C = C \bracket{ q, \theta }$, it holds that
				\begin{align*}
					\LongNorm{ 
						 \nabla^2 V_{\lambda, \epsilon} \sLongBracket{ 
								\gamma v_1 - \bracket{ \gamma v_1 \cdot \nu } \nu 
						} 
					}{\Lebesgue{q}(\Omega)} 
					\leq C \LongNorm{
						f
					}{
						\Lebesgue{q} (\Omega)
					}
				\end{align*}
			for all $f \in \Lebesgue{q} (\Omega)$, where $ v_1 = K_{\lambda, \epsilon} f$.
		\end{proposition}
		}\FurukawaCommentFifth{
		\begin{proof}
			\FurukawaCommentFifth{In view of Remark \ref{rem_average_free}, we \FurukawaCommentSixth{may} assume $\tilde{f} = 0$ without loss of generality.}
			It is enough to estimate the second derivative of the left-hand side of (\ref{eq_formula_V_lambda_epsilon_tangential_term}) in $L^q (\Omega)$.
			We find from (\ref{eq:C0501}) and (\ref{eq:C1560}) that
			\begin{align*}
					& \sLongBracket{
						\abs{n^\prime}^2 \, \Prime{e}_{\lambda} \bracket{\FurukawaCommentFifth{\FurukawaCommentSixth{\xi}^\prime}, \pm 1 - x_3 }
						y_{\lambda, \epsilon} \bracket{\FurukawaCommentFifth{\FurukawaCommentSixth{\xi}^\prime}}
						\Prime{e}_{\lambda} \bracket{\FurukawaCommentFifth{\FurukawaCommentSixth{\xi}^\prime}, \pm 1 - \zeta }
					}_{\Prime{\mathcal{M}}} \notag \\
					& = 2 \sLongBracket{ 
							\abs{\FurukawaCommentSixth{\xi}^\prime}^2 \, e^{ \FurukawaCommentFifth{-} \abs{\pm 1 - x_3} s_\lambda} 
						\LongBracket{
							I_2 
							+ \frac{ 
								\epsilon \abs{\FurukawaCommentFifth{\FurukawaCommentSixth{\xi}^\prime}} 
							}{
								s_{\lambda}
							} \frac{ 
								\FurukawaCommentFifth{\FurukawaCommentSixth{\xi}^\prime} \otimes \FurukawaCommentFifth{\FurukawaCommentSixth{\xi}^\prime} 
							}{
								\abs{\FurukawaCommentFifth{\FurukawaCommentSixth{\xi}^\prime}}^2
							} 
						}
						\frac{
							e^{ \FurukawaCommentFifth{-} \abs{\pm 1 - \zeta} s_\lambda}
						}{
							s_\lambda
						}
					}_{\mathcal{M}^\prime} \notag \\
					& \leq \frac{C}{
						 	\abs{ \pm 1 - x_3 } 
						 	+ \abs{ \pm 1 - \zeta }
					},
				\end{align*}
			where $C> 0$ is independent of $\epsilon$.
			 Similarly, it follows from  (\ref{eq:y_lambda_epsilon}), (\ref{eq:C0712}) and Proposition \ref{prop:C05} that
				\begin{align*}
					&\sLongBracket{
						\abs{\FurukawaCommentSixth{\xi}^\prime}^2 \, \Prime{e}_{\lambda} \bracket{\FurukawaCommentFifth{\Prime{\FurukawaCommentSixth{\xi}}}, \pm 1 - x_3}
							y_{\lambda, \epsilon} \bracket{\FurukawaCommentFifth{\Prime{\FurukawaCommentSixth{\xi}}}}
							\alpha_{\epsilon, \pm} \bracket{\FurukawaCommentFifth{\Prime{\FurukawaCommentSixth{\xi}}}, \pm 1} 
							\Prime{e}_{\lambda} \bracket{\FurukawaCommentFifth{\Prime{\FurukawaCommentSixth{\xi}}}, \pm 1 -\zeta}
							\frac{\epsilon \abs{\FurukawaCommentFifth{\Prime{\FurukawaCommentSixth{\xi}}}}}{ 1 + \epsilon \abs{\FurukawaCommentFifth{\Prime{\FurukawaCommentSixth{\xi}}}} }
					}_{\Prime{\mathcal{M}}} \notag \\  
					&  \leq \frac{C}{
						 	\abs{ \pm 1 - x_3 } 
						 	+ \abs{ \pm 1 - \zeta }
					}.
				\end{align*}
				Thus we find from Corollary \ref{cor_bound_P_N_epsilon} and Proposition \ref{prop:C03} that
				\begin{align*}
					\LongNorm{
						\nabla_H \otimes \nabla_H I_j
					}{L^q(\Omega)}
					\leq C \Vert
						f
					\Vert_{L^q (\Omega)}, \quad j = 1,2,
				\end{align*}
				where $\nabla_H = (\partial_1, \partial_2)^T$, $I_j$ is defined in (\ref{eq_formula_V_lambda_epsilon_tangential_term}) and $C>0$ is independent of $\epsilon$.
				Since 
				\begin{align*}
					\partial_3 e_{\lambda}^\prime (n^\prime, \pm  1- x_3)
					& = \frac{\pm e^{- (1 \mp x_3)s_\lambda}}{2}, \\
					\partial_3^2 e_{\lambda}^\prime (n^\prime, \pm 1- x_3)
					& = \frac{ s_\lambda e^{- (1 \mp x_3)s_\lambda}}{2},
				\end{align*}
				\FurukawaCommentSeventh{we} can use the same way as above to get
				\begin{align*}
					\LongNorm{
						\nabla_H \partial_3 I_j
					}{L^q(\Omega)}
					+ \LongNorm{
						\partial_3^2 I_j
					}{L^q(\Omega)}
					\leq C \Vert
						f
					\Vert_{L^q (\Omega)}, \quad j = 1,2,
				\end{align*}
				where $I_j$ is defined in (\ref{eq_formula_V_lambda_epsilon_tangential_term}) and $C>0$ is independent of $\epsilon$.
				Propositions \ref{prop_nabla2_resolvent_estimate} \FurukawaCommentSeventh{and} \ref{prop:C1500}, Lemma \ref{lem:C0710} and the trace theorem imply 
				\begin{align*}
					R_{\lambda, \epsilon} 
					\sLongBracket{
							\gamma K_{\lambda, \epsilon} E_0 f 
							- \gamma \Pi_{\epsilon} K_{\lambda, \epsilon} E_0 f
					} \in W^{3 - 1 / q, q} (\Torus^2)
				\end{align*}
				and its norm is bounded uniformly on $\epsilon$.
				\FurukawaCommentSeventh{By} the definition of the operator $L_{\lambda, \epsilon}$, see (\ref{eq:C1506}), \FurukawaCommentSeventh{we have}
				\begin{align*}
					L_{\lambda, \epsilon} R_{\lambda, \epsilon} 
					\sLongBracket{
							\gamma K_{\lambda, \epsilon} E_0 f 
							- \gamma \Pi_{\epsilon} K_{\lambda, \epsilon} E_0 f
					}
				\end{align*}
				solves the elliptic equations $\lambda u - \Delta u =0$.
				Moreover, the boundary data belongs to $W^{3 - 1 / q, q} (\Torus^2)$ by (\ref{eq_definition_of_e_prime_lamda}), (\ref{eq:y_lambda_epsilon}) and Proposition \ref{prop_Fourier_Multiplier_Theorem_Discrete}.
				Thus we find from (\ref{eq:C1510}), Corollary \ref{cor_bound_P_N_epsilon} and smoothing effect of the solution operator to the elliptic equation that
				\begin{align*}
					& \LongNorm{W_{\lambda, \epsilon} R_{\lambda, \epsilon} 
					\sLongBracket{
							\gamma K_{\lambda, \epsilon} E_0 f 
							- \gamma \Pi_{\epsilon} K_{\lambda, \epsilon} E_0 f
					}}{W^{2, q} (\Omega)} \\
					& \leq C \LongNorm{L_{\lambda, \epsilon} R_{\lambda, \epsilon} 
					\sLongBracket{
							\gamma K_{\lambda, \epsilon} E_0 f 
							- \gamma \Pi_{\epsilon} K_{\lambda, \epsilon} E_0 f
					}}{W^{2, q} (\Omega)} \\
					& \leq C \LongNorm{R_{\lambda, \epsilon} 
					\sLongBracket{
							\gamma K_{\lambda, \epsilon} E_0 f 
							- \gamma \Pi_{\epsilon} K_{\lambda, \epsilon} E_0 f
					}}{W^{2 - 1/ q + \delta, q} (\Torus^2)} \\
					& \leq C \LongNorm{
						f
					}{L^q (\Omega)},
				\end{align*}
			where $\delta>0$ is small and $C$ is independent of $\epsilon$.
		\end{proof}}
		\FurukawaCommentFifth{
		\begin{lemma} \label{lem_nabla2_resolvent_f}
			Let $1 < q < \infty$, $0 < \epsilon \leq 1$, \FurukawaCommentSeventh{$0< \theta < \pi / 2$} and $\lambda \in \Sigma_\theta$ satisfying $\abs{\lambda} > R$ for sufficiently large $R>0$.
			Then there exists a constant $C = C \bracket{ q, \theta }$ \FurukawaCommentSeventh{such that}
				\begin{align*}
					\LongNorm{ 
						\nabla^2 \left(
							\lambda + A_\epsilon
						\right)^{-1} f
					}{\Lebesgue{q}(\Omega)} 
					\leq C \LongNorm{
						f
					}{
						\Lebesgue{q} (\Omega)
					}
				\end{align*}
			for all $f \in \Lebesgue{q} (\Omega)$.
		\end{lemma}
		\begin{proof}
			This is a direct consequence of Propositions \ref{prop_nabla2_resolvent_estimate}, \ref{prop_nabla2_resolvent_estimate_2} and \ref{prop_nabla_2_V_lambda_epsilon}.
		\end{proof}
		}
		
	\FurukawaCommentFifth{
	\begin{lemma} \label{lem_apriori_estimate_nabla2}
		Let $1 < q < \infty$, $0 < \epsilon \leq 1$.
		Then there exists a constant \FurukawaCommentSixth{$C = C \bracket{ q}$} \FurukawaCommentSeventh{such that}
		\begin{align*}
			\Vert
				\nabla^2 u
			\Vert_{\FurukawaCommentSixth{L^q (\Omega)}}
			\leq C \Vert
				A_\epsilon u
			\Vert_{L^q (\Omega)}
		\end{align*}
		for all $u \in D(A_\epsilon)$.
	\end{lemma}}
	\FurukawaCommentFifth{
	\begin{proof}[Proof of Lemma \ref{lem:A1001}]
		Let $u$ be a solution of (\ref{eq_LE}).
		Our uniform BIP yields
		\begin{align*}
			\Vert
				\partial_t u
			\Vert_{\mathbb{E}_0 (T)}
			+ \Vert
				A_\epsilon u
			\Vert_{\mathbb{E}_0 (T)}
			\leq C \left(
				\Vert
					f
				\Vert_{\mathbb{E}_0 (T)}
				+ \Vert
					u_0
				\Vert_{X_\gamma}
			\right)
		\end{align*}
		by the Dore-Venni theory, where $C>0$ is independent of $\epsilon$ and $T$.
		Applying an a priori estimate Lemma \ref{lem_apriori_estimate_nabla2}, we can replace $\Vert A_\epsilon u \Vert_{\mathbb{E}_0 (T)}$ by $\Vert \nabla^2 u \Vert_{\mathbb{E}_0 (T)}$.
		Since $(u, \pi)$ solves (\ref{eq_LE}) and $\partial_t u$ and $\nabla^2 u$ are controlled, we are able to estimate $\Vert \nabla_\epsilon \pi \Vert_{\mathbb{E}_0 (T)}$.
		This completes the proof of Lemma \ref{lem:A1001}.
	\end{proof}
	}
	\FurukawaCommentFifth{
	It remains to prove Lemma \ref{lem_apriori_estimate_nabla2}.
	We first observe an a priori estimate slightly weaker than Lemma \ref{lem_apriori_estimate_nabla2}, which can be proved by using the resolvent estimate Lemma \ref{lem_nabla2_resolvent_f}.
	}
	\FurukawaCommentFifth{
	\begin{proposition} \label{prop_anisotropic_Stokes_lambda_equal_0}
		Let $1 < q < \infty$ and $0 < \epsilon \leq 1$.
		There exists a unique solution $(u, \pi) \in D (A_\epsilon) \times L^q (\Omega) / \Real$ to
		\begin{align} \label{eq_anisotropic_Stokes_lambda_equal_0}
			\begin{array}{rcll}
				 - \Delta u + \nabla_\epsilon \pi 
				&= 
				& f
				&  \quad \mathrm{in} \quad \Omega , \\
				\mathrm{div}_\epsilon \, u
				&= 
				& 0
				& \quad \mathrm{in} \quad \Omega, \\
				u
				&= 
				& 0
				& \quad  \mathrm{on} \quad \partial \Omega, 
			\end{array}
		\end{align}
		\FurukawaCommentSixth{for $f \in L^q (\Omega)$,} such that
		\begin{align*}
			\Vert
				\nabla^2 u
			\Vert_{L^q (\Omega)}
			+ \Vert
				\nabla_\epsilon \pi
			\Vert_{L^q (\Omega)}
			\leq C \Vert
				f
			\Vert_{L^q (\Omega)}
			+ C \Vert
				u
			\Vert_{L^q (\Omega)},
		\end{align*}
		where $C > 0$ is independent of $\epsilon$ \FurukawaCommentSixth{and $f$}.
	\end{proposition}} 
	\FurukawaCommentFifth{
	\begin{proof}
		The equations are equivalent to
		\begin{align*}
			\begin{array}{rcll}
				\lambda_0 u - \Delta u + \nabla_\epsilon \pi 
				&= 
				& f + \lambda_0 u
				&\quad \mathrm{in} \quad \Omega,  \\
				\mathrm{div}_\epsilon \, u
				&= 
				& 0
				&\quad \mathrm{in} \quad \Omega, \\
				u
				&= 
				& 0
				& \quad \mathrm{on} \quad \partial \Omega,
			\end{array}
		\end{align*}
		\FurukawaCommentSixth{for sufficiently large $\lambda_0 > 0$.}
		We find from Lemma \ref{lem_nabla2_resolvent_f} that
		\begin{align*}
			\Vert
				\nabla^2 u
			\Vert_{L^q (\Omega)}
			& \leq C \Vert
				\nabla^2 (\lambda_0 + A_\epsilon)^{-1}
			\Vert_{L^q (\Omega) \rightarrow L^q (\Omega)}
			\Vert
				f + \lambda_0 u
			\Vert_{L^q (\Omega)} \\
			& \leq C \left(
			 \Vert
					f
				\Vert_{L^q (\Omega)}
				+ \lambda_0 \Vert
					u
				\Vert_{L^q (\Omega)}
			\right)
		\end{align*}
		for some constant $C> 0$, which is independent of $\epsilon$.
		The first equation in (\ref{eq_anisotropic_Stokes_lambda_equal_0}) implies
		\begin{align*}
			\Vert
				\nabla_\epsilon \pi
			\Vert_{L^q (\Omega)}
			& \leq\Vert
				\nabla^2 u
			\Vert_{L^q (\Omega)}
			+ \Vert
				f
			\Vert_{L^q (\Omega)} \\
			& \leq C \left(
			 \Vert
					f
				\Vert_{L^q (\Omega)}
				+ \lambda_0 \Vert
					u
				\Vert_{L^q (\Omega)}
			\right).
		\end{align*}
		\FurukawaCommentSixth{
	For uniqueness we multiply $u$ with the first equation and integrating by parts yields $\nabla u = 0$.
	By the Poincar\'{e} inequality it implies $u = 0$.
	This argument works for $q \geq 2$ since $\Omega$ is bounded.
	Since $(\lambda_0 + A_\epsilon)^{-1}$ is compact in $L^q (\Omega)$, the Riesz-Schauder theorem implies that $0$ is in resolvent since $\mathrm{ker} A_\epsilon = \{ 0 \}$.
	In particular, (\ref{eq_anisotropic_Stokes_lambda_equal_0}) is uniquely solvable for any $f \in L^q (\Omega)$ for $q \geq 2$.
	By duality argument the solvability of $q \geq 2$ implies the uniqueness of (\ref{eq_anisotropic_Stokes_lambda_equal_0}) for $1 < q < 2$.
	Again by compactness of $(\lambda_0 + A_\epsilon)^{-1}$ the solvability for (\ref{eq_anisotropic_Stokes_lambda_equal_0}) follows.
	}
	\end{proof}	}
	\FurukawaCommentFifth{
	\begin{proof}[Proof of Lemma \ref{lem_apriori_estimate_nabla2}]
		Assume that the statement were false then there would exist a sequence $\{ \epsilon_k \}_{k\in \Integer_{\geq 1}}$, ($0 < \epsilon_k \leq 1$) and $u_k \in D(A_\epsilon)$ such that
		\begin{align*}
			\Vert
				\nabla^2 u_k
			\Vert_{L^q (\Omega)}
			> k \Vert
				f_k
			\Vert_{L^q (\Omega)}, \quad
			f_k 
			= A_{\epsilon_k} u_k.
		\end{align*}
		Since the problem is linear we may assume that
		\begin{align*}
			\Vert
				\nabla^2 u_k
			\Vert_{L^q (\Omega)}
			\equiv 1, \quad
			\Vert
				f_k
			\Vert_{L^q (\Omega)}
			\leq \frac{1}{k} 
			\rightarrow 0,  \,
			(k \rightarrow 0).
		\end{align*}
		By $A_{\epsilon_k} u_k = f_k$ and Proposition \ref{prop_anisotropic_Stokes_lambda_equal_0}, we have
		\begin{align*} 
			1
			\leq \alpha \left(
				\Vert
					f_k
				\Vert_{L^q (\Omega)}	
				+ \Vert
					u_k
				\Vert_{L^q (\Omega)}	
			\right)
		\end{align*}
		for some constant $\alpha>0$, which is independent of $\epsilon_k$.
		Letting $k \rightarrow \infty$ implies
		\begin{align} \label{eq_lower_bound_1_u_k}
			\frac{1}{\alpha}
			\leq \liminf_{k \rightarrow \infty} \Vert
				u_k
			\Vert_{L^q (\Omega)}.
		\end{align}
		By the Poincar\'{e} inequality for $u_k$ our bound $\Vert \nabla u_k \Vert_{L^q (\Omega)}$ implies that $u_k$ and $\nabla u_k$ are bounded in $L^q (\Omega)$.
		By Rellich's compactness theorem, we observe that $u_k \rightarrow u$ for some $u \in L^q (\Omega)$ strongly in $L^q (\Omega)$ by taking a subsequence.
		The estimate (\ref{eq_lower_bound_1_u_k}) implies that
		\begin{align*}
			\Vert
				u
			\Vert_{L^q (\Omega)}
			\geq \frac{1}{\alpha}.
		\end{align*}
		We may assume $\epsilon_k \rightarrow \epsilon_\ast \in [0, 1]$ and $u_k \rightarrow u$ as $k \rightarrow \infty$ by taking a subsequence.
		The situation is divided into two cases, i.e. $\epsilon_\ast = 0$ or $\epsilon_\ast > 0$.
		By definition, 
		\begin{align*}
			\begin{array}{rcll}
				 - \Delta u_k + \nabla_{\epsilon_k} \pi_k 
				&= 
				& f_k
				& \quad  \mathrm{in} \quad \Omega , \\
				\mathrm{div}_\epsilon \, u_k
				&= 
				& 0
				& \quad  \mathrm{in} \quad \Omega, \\
				u_k
				&= 
				& 0
				&  \quad \mathrm{on} \quad \partial \Omega, 
			\end{array}
		\end{align*}
		with some function $\pi_k$ satisfying $\int_{\Omega} \pi_k dx = 0$.
		Since $\Vert \nabla^2 u_k \Vert_{L^q (\Omega)} \leq 1$, we see that
		\begin{align*}
			\Vert
				\nabla_{\epsilon_k} \pi_k
			\Vert_{L^q (\Omega)}
			\leq \Vert
				f
			\Vert_{L^q (\Omega)}
			+ 1.
		\end{align*}
		By the Poincar\'{e} inequality $\{ \pi_k \}$ is bounded in $L^q (\Omega)$.
		By Rellich's compactness theorem we may assume $\pi_k \rightarrow \pi$ in $L^q (\Omega)$ for some \FurukawaCommentSixth{$\pi \in L^q (\Omega)$} strongly by taking a subsequence. 
		If $\epsilon_\ast = 0$, this implies $\pi$ is independent of \FurukawaCommentSixth{$z$}.
		Since $\mathrm{div}_{\epsilon_k} \, u_k = 0$ and the vertical component $w_k = 0$ on $x_3 = \pm 1$, integration vertically on $(-1 , 1)$ yields that the horizontal limit $v$ satisfies
		\begin{align*}
			\mathrm{div}_H \overline{v} = 0,
		\end{align*}
		where $\mathrm{div}_H = \nabla_H \cdot$.
		Thus the horizontal component $v$ satisfies the hydrostatic Stokes equations
		\begin{align*}
			\begin{array}{rcll}
				 - \Delta u + \nabla_H \pi 
				&= 
				& 0
				& \quad  \mathrm{in} \quad \Omega , \\
				\mathrm{div}_H \, \overline{v}
				&= 
				& 0
				&  \quad \mathrm{in} \quad \Omega, \\
				v
				&= 
				& 0
				&  \quad \mathrm{on} \quad \partial \Omega.
			\end{array}
		\end{align*}
		Since we know the only possible $W^{2 , q}$-solution is zero, so we conclude that $v = 0$.
		Since $\Vert \nabla v_k \Vert_{L^q (\Omega)}$ is bounded, $\mathrm{div}_{\epsilon_k} $-free condition implies that the horizontal limit $w$ is independent of the vertical variable.
		By the boundary condition $w = 0$ at $x_3 = \pm 1$, this implies $w$ must be zero.
		We thus observe that $u_k \rightarrow 0$ strongly in $L^q (\Omega)$, this contradicts $\Vert u \Vert_{L^q (\Omega)} \geq 1 / \alpha > 0$.
		The case $\epsilon_\ast$ is easier since the limit satisfies the anisotropic Stokes equations
		\begin{align*}
			\begin{array}{rcll}
				 - \Delta u + \nabla_{\epsilon_\ast} \pi 
				&= 
				& 0
				&  \quad \mathrm{in} \quad \Omega , \\
				\mathrm{div}_{\epsilon_\ast} \, u
				&= 
				& 0
				&  \quad \mathrm{in} \quad \Omega, \\
				u
				&= 
				& 0
				& \quad  \mathrm{on} \quad \partial \Omega.
			\end{array}
		\end{align*}
		By the uniqueness $u \equiv 0$ in $\Omega$.
		This again contradicts $\Vert u \Vert_{L^q (\Omega)} \geq 1 / \alpha > 0$.
		The proof of Lemma \ref{lem_apriori_estimate_nabla2} is now complete.
	\end{proof}	}
	\FurukawaCommentSixth{As a direct application of Lemma \ref{lem:A1001} we obtain}
		\begin{corollary} \label{cor_mr_estimate}
		Let $p,q\in (1,\infty )$, $T>0$, $F = (f_H, f_z) \in \mathbb{E} _0(T)$, $U_0\in X_\gamma $ and \FurukawaCommentFourth{$0<\epsilon \leq 1$}. 
		Then there is a unique solution $( U_\epsilon,
		P_\epsilon) \in \mathbb{E}_1 (T) \times \mathbb{E}_0(T)$ to the \FurukawaCommentFourth{equations}
			\begin{equation} \label{eq_diff-eq_abstract_2}
				\left \{
					\begin{array}{rclll}
						\partial _t V-\Delta V + \nabla_{H} P
						& =
						& f_H 
						&\mathrm{ in }
						& \Omega \times  (0,T) ,\\
						\partial _t (\epsilon W) -\Delta (\epsilon W) + \frac{\partial_3 P}{\epsilon}
						& =
						& f_z 
						&\mathrm{in}  
						& \Omega \times (0, T) ,\\
						\mathrm{div}_H \, V + \frac{\partial_3}{\epsilon
						} ({\FurukawaCommentFourth{\epsilon}} W)
						&=
						& 0
						&\mathrm{ in }
						& \Omega \times (0, T),\\
						U 
						& =
						& 0
						&  \mathrm{on} 
						& \partial \Omega \times (0, T) \FurukawaCommentFourth{,} \\
						U (0)
						& =
						& U_0
						&\mathrm{ in }
						& \Omega ,
					\end{array}\right .
			\end{equation}
		where $P$ is unique up to a constant.
		Moreover, there exist constants $C>0$ and $C_T>0$, which is independent of $\epsilon$,  such that
		\begin{equation} \label{maximal_regularity_of_anisotropic_Stokes_2}
			\Vert \left(
				V, {\FurukawaCommentFourth{\epsilon}} W
			\right)	
			\Vert _{\mathbb{E}_1(T)}
			+ \Vert 
				\nabla_{{\FurukawaCommentFourth{\epsilon}}} P \Vert_{\mathbb{E}_0(T)}
			\le  C\Vert F \Vert _{\mathbb{E}_0(T)}
			+ C_T \Vert 
			\left(
				V_0, \epsilon W_0
			\right)	
			\Vert _{X_\gamma }.
		\end{equation}
	\end{corollary}

	\begin{proof}
		Lemma \ref{lem:A1001} implies there exists a solution $(\tilde{U}, \tilde{P})$ to (\ref{eq_LE}) with initial data $U_0$ such that
		\begin{equation*} 
			\Vert 
				\widetilde{U}
			\Vert _{\mathbb{E}_1(T)}
			+ \Vert 
				\nabla_{{\FurukawaCommentFourth{\epsilon}}} \widetilde{P} \Vert_{\mathbb{E}_0(T)}
			\le  C\Vert F \Vert _{\mathbb{E}_0(T)}
			+ C_T \Vert 
				U_0
			\Vert _{X_\gamma }.
		\end{equation*}
		Set
		\begin{align*}
			V = \tilde{V}, \quad
			W = \epsilon \tilde{W}, \quad
			P = \tilde{P}.
		\end{align*}
		Then $(U, P)$ is the desired solution satisfying (\ref{maximal_regularity_of_anisotropic_Stokes_2}).
		\FurukawaComment{Note that $\lim_{T \rightarrow \infty} C_T < \infty$.}
	\end{proof}
		
	
	\section{Non-linear Estimates and Regularity of w} \label{section_improved_regularity_w}
	
	In this section\FurukawaCommentFourth{,} we introduce some Propositions on non-linear estimates to estimate $F_H$, $F_z$ and $F$ and on the regularity of $w$, which is \FurukawaCommentFourth{the} vertical component of the solution to the primitive equations.
	Although the following Propositions have already proved in \cite{FurukawaGigaHieberHusseinKashiwabaraWrona2018}, we introduce them to explain our restriction for $p$ and $q$ and for \FurukawaCommentFourth{the reader's convenience}.
	
	
	\begin{proposition}[\cite{FurukawaGigaHieberHusseinKashiwabaraWrona2018}] \label{prop:0201}
		Let $T > 0$, $p, q \in \bracket{ 1, \infty } $ such that $2 / 3p + 1 / q \leq 1$. 
		Then there exist a constant $C = C (p, q) > 0$ such that
			\begin{align*}
				\LongNorm{ v_1 \partial_x v_2}{\mathbb{E}_0(T)} \leq C \LongNorm{v_1}{\mathbb{E}_1 (T)} \LongNorm{v_2}{\mathbb{E}_1 (T)} 
			\end{align*}
		\FurukawaCommentFourth{for all $v_1, v_2 \in \mathbb{E}_1(T)$.}
	\end{proposition}


	\begin{proposition}[\cite{FurukawaGigaHieberHusseinKashiwabaraWrona2018}] \label{prop:0202}
		Let $T > 0$ and $z \in \bracket{ - 1, 1}$. Let $p, q \in \bracket{ 1, \infty }$ such that $ 1 / p + 1 / q \leq 1 $. Then there exist a constant $C = C (p, q) > 0$ such that
			\begin{align*}
				\LongNorm{ w_1 \partial_3 v_2}{\mathbb{E}_0 (T)} \leq C \LongNorm{ v_1 }{\mathbb{E}_1 (T)} \LongNorm{ v_2 }{\mathbb{E}_1 (T)}
			\end{align*}
		\FurukawaCommentFourth{for all $v_1, v_2 \in \mathbb{E}_2 (T)$ and $ w_1 : = - \TimeInt{z}{-1}{ \mathrm{div}_H v_1 }{\zeta} $.}
 	\end{proposition}
		The \FurukawaCommentFourth{restriction} for $p$ and $q$ in our theorem is due to \FurukawaCommentFifth{Propositions \ref{prop:0201} and \ref{prop:0202}}. 
	Let us show $w \in \mathbb{E}_1 \bracket{T}$.
	In our previous paper \cite{FurukawaGigaHieberHusseinKashiwabaraWrona2018}, we first derive the equation which $w$ satisfies by applying $\TimeInt{-1}{x_3}{\mathrm{div}_H \, \cdot}{\zeta}$ to the equations $v$ satisfies.
	Then, estimating the corresponding non-linear terms and applying the maximal regularity principle, we obtain $w \in \mathbb{E}_1 \bracket{T}$. 
	\FurukawaCommentThird{
		Note that, in the present paper, we invoke additional regularity for $v$ to deal with the trace of the second derivative.
	}
	\FurukawaCommentSecond{
		Although, in \cite{GigaGriesHieberHusseinKashiwabara2017_analiticity}, the authors treat higher order regularity of the solution to the primitive equations, \FurukawaCommentThird{they do not explicitly write the maximal regularity in fractional Sobolev spaces.}
		\FurukawaCommentThird{However, it is easy to modify their proof to get the maximal regularity in the fractional Sobolev spaces.}
		In \cite{GigaGriesHieberHusseinKashiwabara2017_H_infty}, the argument to get $H^\infty$-calculus of hydrostatic Stokes operator is based on $H^\infty$-calculus for the Laplace operator and perturbations arguments.
		Since the Laplace operator admits $H^\infty$-calculus in fractional Sobolev spaces, \FurukawaCommentThird{it is not difficult to establish $H^{\infty}$-calculus of the hydrostatic Stokes operator in fractional Sobolev spaces.}
		We also find local well-posedness of the primitive equations in \FurukawaCommentThird{fractional maximal regularity space $W^{1, p} (0, T ;W^{s \FurukawaCommentFifth{, q}} (\Omega) ) \cap L^{p} (0, T ;W^{2 + s, q} (\Omega) )$ for $s > 1 / q$ \FurukawaCommentThird{in the same way} \cite{GigaGriesHieberHusseinKashiwabara2017_analiticity} to get local well-posedness, namely, use Lemma 6.1, Corollary 6.2 and Theorem 5.1 in \cite{GigaGriesHieberHusseinKashiwabara2017_analiticity}.}}

%
 	
 	
	\begin{remark}
		It \FurukawaCommentFifth{is} already known that $v \in \mathbb{E}_1 \FurukawaCommentFifth{(T)}$ with initial data $v_0 \in X_{\gamma}$ by Giga, et al. \cite{GigaGriesHieberHusseinKashiwabara2017_H_infty}.
	\end{remark}
 	
	\begin{proof}[Proof of Lemma \ref{lem:A1002}]
 		Integrating (PE) both sides over $\bracket{-1, 1}$, we find $\bracket{\overline{v}, \overline{\pi}}$ \FurukawaCommentFifth{satisfy}
			\begin{align} \label{eq:B0801}
				\begin{array}{rcll}
					\dt \overline{v} - \Delta \overline{v} + \nabla_H \overline{\pi}
					& = 
					& - \TimeInt{-1}{1}{ v \cdot \nabla_H v + w \partial_3 v
						}{\zeta}
					&  \\ 
					\,
					& \,
					&  + \bracket{\partial_3 v} |_{x_3 = -1}^{x_3 = 1} 
					& \In{\Omega \times \bracket{0, T}}\FurukawaCommentFourth{,} \\ 
					\mathrm{div}_H \overline{v} 
					& = 
					& 0 
					&\In{\Omega \times \bracket{0, T}}\FurukawaCommentFourth{,} \\
					\overline{v}
					& = 
					& 0 
					&\On{ \partial \Omega \times \bracket{0, T}}\FurukawaCommentFourth{,} \\
					\overline{v} (0)
					& = 
					& \overline{v_0}
					& \quad \mathrm{in} \quad \Omega.
				\end{array}				
			\end{align}			 		
 		 Put $\tilde{v} =  v - \overline{v}$.
 		 Then, $\tilde{u} = (\tilde{v}, w)$ satisfies
 		\begin{align} \label{eq;B1001}
 			\begin{array}{rcll}
	 			\dt \tilde{v} - \Delta \tilde{v} 
 				& = 
 				& - \tilde{v} \cdot \nabla_H\tilde{v} 
 				- w \partial_z \tilde{v} 
 				- \overline{v} \cdot \nabla_H \tilde{v} 
 				- \tilde{v} \cdot \nabla_H \overline{v}
 				& \, \\
 				\,
 				& \, 
 				& - \frac{1}{2} 
 				\TimeInt{-1}{1}{
 					\tilde{v} \cdot \nabla_H \tilde{v} 
 					- \bracket{ \mathrm{div}_H \, \tilde{v} } 
 				}{\zeta}
 				& \, \\
 				\,
 				& \, 
 				& +  \frac{1}{2} \bracket{\partial_3 v} |_{x_3 = -1}^{x_3 = 1} 
 				& \In{\Omega \times \bracket{0, T}}, \\
 				\mathrm{div}_H \, \tilde{v} + \partial_z w 
 				& = 
 				& 0 
 				& \In{\Omega \times \bracket{0, T}} \FurukawaCommentFourth{,} \\
 				\tilde{v}
				& = 
				& 0 
				&\On{ \partial \Omega \times \bracket{0, T}} \FurukawaCommentFourth{,} \\
 				\tilde{v} (0)
 				& = 
 				& v (0) - \overline{v_0}
 				& \In{\Omega} \FurukawaCommentFourth{.}
 			\end{array}
 		\end{align}
 		Note that the pressure term no longer appears in the above equations and
 		\begin{align*}
 			\mathrm{div}_H \, \tilde{v} + \partial_3 w 
 			= 0.	
		\end{align*} 			
 		Applying $ - \mathrm{div}_H $ to (\ref{eq;B1001}) and integrating over $(-1, x_3)$ with respect to vertical variable, we find
 		\begin{align*}
 			& \dt w - \Delta w \notag \\
 			& = \FurukawaComment{
 				\partial_z \mathrm{div}_H \tilde{v} |_{x_3 = - 1} 
 				- \int_{-1}^{x_3}
 					\frac{1}{2} \mathrm{div}_H  \left[ 
 						\bracket{\partial_3 v} |_{x_3 = -1}^{x_3 = 1} 
 					\right]
 					d \zeta \notag
 				}\\
 			& + \int_{-1}^{x_3}
 				\FurukawaCommentFourth{\mathrm{div}_H} 
 					\bracket{ 
 						- \tilde{v} \cdot \nabla_H\tilde{v} 
 						- w \partial_{\zeta} \tilde{v} 
 						- \overline{v} \cdot \nabla_H \tilde{v} 
 						- \tilde{v} \cdot \nabla_H \overline{v} }
 			d \zeta \notag \\
 			& - \frac{1}{2} 
			\TimeInt{-1}{x_3}{	
	 			\mathrm{div}_H
 				\TimeInt{-1}{1}{
 					\tilde{v} \cdot \nabla_H \tilde{v} - \bracket{ \mathrm{div}_H \, \tilde{v} } \tilde{v} 
 				}{\zeta}
 			}{\eta}  \notag \\
 			& = : I_1 + I_2 + I_3,
		\end{align*}
		with initial data $w_0$. 
		\FurukawaCommentSecond{
			Since $v_0 \in X_\gamma \cap B^{s}_{q,p}(\Omega)$ for $s > 2 - 2/ p + 1/q$, we have $v \in \mathbb{E}_1 (T) \cap L^p (0, T ; W^{\FurukawaCommentFourth{2 + 1/q +\delta}, q}(\Omega))$ \FurukawaCommentFourth{for some $\delta>0$ by \cite{GigaGriesHieberHusseinKashiwabara2017_analiticity} and \cite{GigaGriesHieberHusseinKashiwabara2017_H_infty} }, and thus $\nrm{I_1}{\mathbb{E}_0 \bracket{T}} \leq C$ for some $C>0$.}
		We use integration by parts to get
		\begin{align*}
			I_2 
			& = \tilde{v} \cdot \nabla_H w 
			- w \mathrm{div}_H v
			+ \overline{v} \cdot \nabla_H w \\
			& + \int_{-1}^{x_3}
					\partial_j \tilde{v} \cdot \nabla_H \tilde{v}_{\FurukawaCommentFourth{j}}
					-
					\bracket{ 
						\partial_{\zeta} \tilde{v}
					\cdot \nabla_H w
				} 
				+ \nabla_H w \cdot \partial_{\zeta} \tilde{v} 
				- 
						\partial_{\zeta} w \,
						\mathrm{div}_H \tilde{v} 
			\, d \zeta \\
			& + \TimeInt{-1}{x_3}{ 
				\partial_j \overline{v} \cdot \nabla_H \tilde{v} _{\FurukawaCommentFourth{j}}
				+ \partial_j \tilde{v} \cdot \nabla_H \overline{v}_j 
			}{\zeta}.
		\end{align*}
	\FurukawaCommentFourth{for $j = 1, 2$.}
	We can apply \FurukawaCommentFourth{Propositions} \ref{prop:0201} and \ref{prop:0202} to $I_2$ to get
		\begin{align} \label{eq_estimate_I_2_improved_reg}
			\LongNorm{
				I_2
			}{\mathbb{E}_0 \bracket{T}}
			\leq C
			\LongBracket{
				\nrm{w}{\mathbb{E}_1 \bracket{T}} \LongNorm{\tilde{v}}{\mathbb{E}_1 \bracket{T}}
				+ \LongNorm{\tilde{v}}{\mathbb{E}_1 \bracket{T}}^2
			}.
		\end{align}
	Similarly, we have
		\begin{align} \label{eq_estimate_I_3_improved_reg}
			\LongNorm{
				I_3
			}{\mathbb{E}_0 \bracket{T}}
			\leq C
			\LongBracket{
				\nrm{w}{\mathbb{E}_1 \bracket{T}} \LongNorm{\tilde{v}}{\mathbb{E}_1 \bracket{T}}
				+ \LongNorm{\tilde{v}}{\mathbb{E}_1 \bracket{T}}^2
			}.
		\end{align} 
	Note that constants in (\ref{eq_estimate_I_2_improved_reg}) and (\ref{eq_estimate_I_3_improved_reg}) are independent of $T$ since constants in \FurukawaCommentFourth{Propositions} \ref{prop:0201} and \ref{prop:0202} are independent of \FurukawaComment{$T$}.
	Thus we find from the maximal regularity of the heat equation, implicit function theorem and Neumann series argument, which is the same way as in Proposition 4.8 in \cite{FurukawaGigaHieberHusseinKashiwabaraWrona2018}, that
		\begin{align*}
			\Vert
				w
			\Vert_{\mathbb{E}_1 (T)}
			\leq C
		\end{align*}
	for some \FurukawaComment{$C > 0$}.
 	\end{proof}
 	
 	\section{Justification of the Hydrostatic approximation and Global-well-posedness of the anisotropic Navier-Stokes Equations} \label{section_proof_of_main_thoerem}

		
	Let us prove our main theorem.
	
	\begin{proof}[Proof of Theorem \ref{thm:A0105}]
		Let $\mathcal{T} > 0$.
		Let $C_1$ be the maximum of constants $C$ in \FurukawaCommentFourth{Propositions} \ref{prop:0201} and \ref{prop:0202},  (\ref{maximal_regularity_of_anisotropic_Stokes_2}) and the constant in the trace theorem.
		Let us construct a solution $(V_{\FurukawaCommentFourth{\epsilon}}, {\FurukawaCommentFourth{\epsilon}} W_{\FurukawaCommentFourth{\epsilon}})$ to (\ref{eq:A0502})
		with zero initial data on $[0, \mathcal{T}]$.
		Set $(u_{\FurukawaCommentFourth{\epsilon}}, p_{\FurukawaCommentFourth{\epsilon}}) := (v + V_{\FurukawaCommentFourth{\epsilon}}, w + W_{\FurukawaCommentFourth{\epsilon}}, p + P_{\FurukawaCommentFourth{\epsilon}})$, then this is the desired solution to (SNS).
		We denote by $\Vert \cdot \Vert_{\mathbb{E}_1 (mT, (m + 1)T)}$ and $\Vert \cdot \Vert_{\mathbb{E}_0 (mT, (m + 1)T)}$ the $\mathbb{E}_1$-norm and $\mathbb{E}_0$-norm on the time interval $[m T, (m + 1)T]$, respectively.
		Let us take $0 < T \leq 1$ \FurukawaCommentFifth{satisfying $\mathcal{T} = N T$ for sufficiently large integer $N$ and}
 		\begin{align} \label{Choice_of_T}
 			\Vert
 				u
 			\Vert_{\mathbb{E}_1 (m T, (m + 1) T)}
 			\leq \frac{1}{\FurukawaCommentFifth{10 C_1}},
 		\end{align}
 		for all integer \FurukawaCommentFifth{$m \in (1, N)$}.
 		This choice of $T$ is clearly independent of ${\FurukawaCommentFourth{\epsilon}}$.
 		We divide the time interval $[0, \mathcal{T}]$ into $\cup_{m = 0}^N [m T, (m + 1) T ]$.
		Put $F = F ( V_{\FurukawaCommentFourth{\epsilon}}, W_{\FurukawaCommentFourth{\epsilon}}, u) := \left( F_H ( V_{\FurukawaCommentFourth{\epsilon}}, W_{\FurukawaCommentFourth{\epsilon}}, u)), F_z ( V_{\FurukawaCommentFourth{\epsilon}}, W_{\FurukawaCommentFourth{\epsilon}}, u) \right)$ be the left hand side of (\ref{eq:A0502}).
		We denote by $\mathcal{R} (F, U_0) = (\mathcal{R}^u (F, U_0), \mathcal{R}^p (F, U_0 ) \FurukawaComment{)}= ( U, P)$ the solution to (\ref{eq_diff-eq_abstract_2}) with initial data $U_0$.
		Set inductively
		\begin{align*}
			& U_{{\FurukawaCommentFourth{\epsilon}}, 1}
			= \mathcal{R}^u (F (0, u), 0) , \quad
			P_{{\FurukawaCommentFourth{\epsilon}}, 1} 
			=  \mathcal{R}^p (F (0, u), 0),\\
			& U_{{\FurukawaCommentFourth{\epsilon}}, j + 1}
			= \mathcal{R}^u (F ( U_j, u), 0), \quad
			P_{{\FurukawaCommentFourth{\epsilon}}, j + 1}
			= \mathcal{R}^p (F (U_j, u), 0) \FurukawaCommentFourth{.}
		\end{align*} 
		\FurukawaCommentFourth{Propositions} \ref{prop:0201}, \ref{prop:0202} and \FurukawaCommentSeventh{Corollary} \ref{cor_mr_estimate} lead \FurukawaCommentFifth{to}
		\begin{align} 
			&\Vert
				\left(
				V_{{\FurukawaCommentFourth{\epsilon}}, j + 1}
				, {\FurukawaCommentFourth{\epsilon}} W_{ {\FurukawaCommentFourth{\epsilon}}, j + 1}
				\right)
			\Vert_{\mathbb{E}_1 (T)}
			+ \Vert
				\nabla_{\FurukawaCommentFourth{\epsilon}} P_{{\FurukawaCommentFourth{\epsilon}}, j + 1}
			\Vert_{\mathbb{E}_0 (T)} \notag \\ 
			& \leq C_1 T^\eta \left( 
				\Vert
					u
				\Vert_{\mathbb{E}_1 (T)}
				\Vert
					\left(
						V_{{\FurukawaCommentFourth{\epsilon}}, j}
						, {\FurukawaCommentFourth{\epsilon}} W_{ {\FurukawaCommentFourth{\epsilon}}, j }
					\right)
				\Vert_{\mathbb{E}_1 (T)}
				+ \Vert
					\left(
						V_{{\FurukawaCommentFourth{\epsilon}}, j}
						, {\FurukawaCommentFourth{\epsilon}} W_{ {\FurukawaCommentFourth{\epsilon}}, j }
					\right)
				\Vert_{\mathbb{E}_1 (T)}^2
			\right) \notag \\ \label{ineq_quadratic_estimate}
 			& + {\FurukawaCommentFourth{\epsilon}} C_1 T^\eta \left(
				\Vert
					u
				\Vert_{\mathbb{E}_1 (T)}
				+ \Vert
					u
				\Vert_{\mathbb{E}_1 (T)}^2 			
 			\right) \FurukawaCommentFourth{.}
 		\end{align}
 		This quadratic inequality and (\ref{Choice_of_T}) \FurukawaCommentFifth{imply}
 		\begin{align} \label{estimate_max_norm_on_0_T}
			\Vert
				\left(
						V_{{\FurukawaCommentFourth{\epsilon}}, j}
						, {\FurukawaCommentFourth{\epsilon}} W_{ {\FurukawaCommentFourth{\epsilon}}, j }
					\right)
			\Vert_{\mathbb{E}_1 (T)}
			+ \Vert
				\nabla_{\FurukawaCommentFourth{\epsilon}} P_{{\FurukawaCommentFourth{\epsilon}}, j}
			\Vert_{\mathbb{E}_0 (T)}
			\leq 2 {\FurukawaCommentFourth{\epsilon}} C^\ast
 		\end{align}
 		for $C^\ast = \left( 1/4 C_1 + 1/ 16 C_1^2 \right)$ and small ${\FurukawaCommentFourth{\epsilon}} > 0$.
 		Put 
 		\begin{align*}
 			& \widetilde{U}_{{\FurukawaCommentFourth{\epsilon}}, j} 
 			= U_{{\FurukawaCommentFourth{\epsilon}}, j+1} - U_{{\FurukawaCommentFourth{\epsilon}}, j} \, ( j \geq 1), \quad 
 			\widetilde{U}_{{\FurukawaCommentFourth{\epsilon}}, 0} 
 			= U_{{\FurukawaCommentFourth{\epsilon}}, 1}\FurukawaCommentFourth{,} \\
 			& \widetilde{P}_{{\FurukawaCommentFourth{\epsilon}}, j} 
 			= P_{{\FurukawaCommentFourth{\epsilon}}, j + 1} - P_{{\FurukawaCommentFourth{\epsilon}}, j } \, (j \geq 1 ), \quad 
 			\widetilde{P}_{{\FurukawaCommentFourth{\epsilon}}, 0} 
 			= P_{{\FurukawaCommentFourth{\epsilon}}, 1}.
 		\end{align*}
 		Then seeking the equation which $(\widetilde{U}_{{\FurukawaCommentFourth{\epsilon}}, j}, \widetilde{P}_{{\FurukawaCommentFourth{\epsilon}}, j})$ satisfies and applying Propositions \ref{prop:0201}, \ref{prop:0202} and \FurukawaCommentFifth{Corollary} \ref{cor_mr_estimate}, we have
 		\begin{align*} \label{ineq_quadratic_estimate_for_diff_eq}
 			& \Vert
 				(
						\widetilde{V}_{{\FurukawaCommentFourth{\epsilon}}, j + 1}
						, {\FurukawaCommentFourth{\epsilon}} \widetilde{W}_{ {\FurukawaCommentFourth{\epsilon}}, j + 1 }
				)
 			\Vert_{\mathbb{E}_1 (T)}
 			+ \Vert
 				\nabla_{\FurukawaCommentFourth{\epsilon}} \tilde{P}_{j + 1}
 			\Vert_{\mathbb{E}_0 (T)} \\
 			& \leq C_1 T^\eta \left(
 				\Vert
 					\left(
						V_{{\FurukawaCommentFourth{\epsilon}}, j}
						, {\FurukawaCommentFourth{\epsilon}} W_{ {\FurukawaCommentFourth{\epsilon}}, j }
					\right)
 				\Vert_{\mathbb{E}_1 (T)}
 				+ \Vert
 					\left(
						V_{{\FurukawaCommentFourth{\epsilon}}, j+ 1}
						, {\FurukawaCommentFourth{\epsilon}} W_{ {\FurukawaCommentFourth{\epsilon}}, j + 1 }
					\right)
 				\Vert_{\mathbb{E}_1 (T)}
 			\right. \\
 			& \left.
 				\quad \quad \quad \quad
 				+ 2 \Vert
 					u
 				\Vert_{\mathbb{E}_1 (T)}
 			\right) \Vert
 				(
						\widetilde{V}_{{\FurukawaCommentFourth{\epsilon}}, j}
						, {\FurukawaCommentFourth{\epsilon}} \widetilde{W}_{ {\FurukawaCommentFourth{\epsilon}}, j }
				)
 			\Vert_{\mathbb{E}_1 (T)} \\
 			& \leq \frac{3}{4}
 			\left(
 				\Vert
 				(
						\widetilde{V}_{{\FurukawaCommentFourth{\epsilon}}, j}
						, {\FurukawaCommentFourth{\epsilon}} \widetilde{W}_{ {\FurukawaCommentFourth{\epsilon}}, j }
				)
 			\Vert_{\mathbb{E}_1 (T)}
 			+ \Vert
 				\nabla_{\FurukawaCommentFourth{\epsilon}} \tilde{P}_{j}
 			\Vert_{\mathbb{E}_0 (T)}
 			\right).
 		\end{align*}
 		Thus $(U_{\FurukawaCommentFourth{\epsilon}}, P_{\FurukawaCommentFourth{\epsilon}}) : = ( \lim_{j \rightarrow \infty} U_j, \lim_{j \rightarrow \infty} P_j ) = (\sum_{j = 0} \tilde{U}_{{\FurukawaCommentFourth{\epsilon}}, j}, \sum_{j = 0} \tilde{P}_{{\FurukawaCommentFourth{\epsilon}}, j})$ exists on $[0, T]$ and satisfies
 		\begin{align}
 			\Vert
 				\left(
						V_{{\FurukawaCommentFourth{\epsilon}}}
						, {\FurukawaCommentFourth{\epsilon}} W_{ {\FurukawaCommentFourth{\epsilon}} }
					\right)
 			\Vert_{\mathbb{E}_1 (T)}
 			+ \Vert
 				\nabla_{\FurukawaCommentFourth{\epsilon}} P_{\FurukawaCommentFourth{\epsilon}}
 			\Vert_{\mathbb{E}_0 (T)}
 			\leq 2 {\FurukawaCommentFourth{\epsilon}} C^\ast.
 		\end{align}
 		By construction $(U_{\FurukawaCommentFourth{\epsilon}}, P_{\FurukawaCommentFourth{\epsilon}})$ satisfies (\ref{eq:A0502}) on $[0, T]$.
 		Moreover, by trace theorem there exists a constant $C_{tr} > 0$ such that
 		\begin{align} \label{ineq_trace_at_T}
 			\Vert
 				\left(
						V_{{\FurukawaCommentFourth{\epsilon}}} (T)
						, {\FurukawaCommentFourth{\epsilon}} W_{ {\FurukawaCommentFourth{\epsilon}} } (T)
					\right)
 			\Vert_{X_\gamma}
 			\leq C_{tr} \Vert
 				\left(
						V_{{\FurukawaCommentFourth{\epsilon}}}
						, {\FurukawaCommentFourth{\epsilon}} W_{ {\FurukawaCommentFourth{\epsilon}} }
					\right)
 			\Vert_{E_1 (0, T)}
 			\leq 2 {\FurukawaCommentFourth{\epsilon}} C^\ast C_{tr}.
 		\end{align}
 		Next let us construct the solution to (\ref{eq:A0502}) on $[ T, 2T]$ with initial data $U_{\FurukawaCommentFourth{\epsilon}} (T)$.
 		By \FurukawaCommentFourth{(\ref{ineq_trace_at_T}), we have $\Vert U_{\FurukawaCommentFourth{\epsilon}} (T) \Vert_{X_\gamma} \leq 2 {\FurukawaCommentFourth{\epsilon}}  C^\ast C_{tr} $}.
		Put $a_{{\FurukawaCommentFourth{\epsilon}}, 1} = (b_{{\FurukawaCommentFourth{\epsilon}}, 1}, c_{{\FurukawaCommentFourth{\epsilon}}, 1}) = \mathcal{R}^u (0, U_{\FurukawaCommentFourth{\epsilon}} (T))$ and $\pi_{{\FurukawaCommentFourth{\epsilon}}, 1} = \mathcal{R}^p (0, U_{{\FurukawaCommentFourth{\epsilon}}} (T))$.
		Corollary \ref{cor_mr_estimate} implies 
		\begin{align} \label{ineq_a_1_pi_1_MR}
			\Vert
				\left(
					b_{{\FurukawaCommentFourth{\epsilon}}, 1},
					{\FurukawaCommentFourth{\epsilon}} c_{{\FurukawaCommentFourth{\epsilon}}, 1}
				\right)
			\Vert_{\mathbb{E}_1 (T, 2 T)}	
			+ \Vert
				\nabla_{\FurukawaCommentFourth{\epsilon}} \pi_{{\FurukawaCommentFourth{\epsilon}}, 1}
			\Vert_{\mathbb{E}_0 (T, 2 T)}
			\leq 2 {\FurukawaCommentFourth{\epsilon}} C^\ast C_{tr} C_T. 
		\end{align}
 		Let the vector field $a_{\FurukawaCommentFourth{\epsilon}} = (b_{\FurukawaCommentFourth{\epsilon}}, c_{\FurukawaCommentFourth{\epsilon}})$ be the solution to
		\begin{equation} \label{eq_DE_modified_2}
			\left \{
			\begin{array}{rl}
				\partial_t b_{\FurukawaCommentFourth{\epsilon}} - \Delta b_{\FurukawaCommentFourth{\epsilon}} + \nabla_H \pi_{\FurukawaCommentFourth{\epsilon}}	
				& = F_H ( b_{1, {\FurukawaCommentFourth{\epsilon}}} + b_{\FurukawaCommentFourth{\epsilon}}, c_{{\FurukawaCommentFourth{\epsilon}}, 1} +c_{\FurukawaCommentFourth{\epsilon}}, u) \FurukawaCommentFourth{,} \\
				\partial_t ( {\FurukawaCommentFourth{\epsilon}} c_{\FurukawaCommentFourth{\epsilon}} ) - \Delta ( {\FurukawaCommentFourth{\epsilon}} c_{\FurukawaCommentFourth{\epsilon}} ) + \frac{\partial_3}{{\FurukawaCommentFourth{\epsilon}}} \pi_{\FurukawaCommentFourth{\epsilon}}	
				& = {\FurukawaCommentFourth{\epsilon}} F_z ( b_{1, {\FurukawaCommentFourth{\epsilon}}} + b_{\FurukawaCommentFourth{\epsilon}}, c_{{\FurukawaCommentFourth{\epsilon}}, 1} +c_{\FurukawaCommentFourth{\epsilon}}, u) \FurukawaCommentFourth{,} \\
				\mathrm{div} \, a_{\FurukawaCommentFourth{\epsilon}}
				& = 0 \FurukawaCommentFourth{,} \\
				a_{\FurukawaCommentFourth{\epsilon}} (T)
				& =  0 \FurukawaCommentFourth{.}
 			\end{array} 
			\right.
		\end{equation}
		Then $U_{{\FurukawaCommentFourth{\epsilon}}} = a_{{\FurukawaCommentFourth{\epsilon}}, 1}  + a_{{\FurukawaCommentFourth{\epsilon}}}$ and $P_{\FurukawaCommentFourth{\epsilon}} = \pi_{{\FurukawaCommentFourth{\epsilon}}, 1} + \pi_{{\FurukawaCommentFourth{\epsilon}}}$ is a solution to ((\ref{eq:A0502}) with initial data $U_{\FurukawaCommentFourth{\epsilon}} (T)$.
		Let us construct the solution to (\ref{eq_diff-eq_abstract_2}). 
		Let $F ( b_{1, {\FurukawaCommentFourth{\epsilon}}} + b_{\FurukawaCommentFourth{\epsilon}}, c_{{\FurukawaCommentFourth{\epsilon}}, 1} +c_{\FurukawaCommentFourth{\epsilon}}, u) = \left( F_H ( b_{1, {\FurukawaCommentFourth{\epsilon}}} + b_{\FurukawaCommentFourth{\epsilon}}, c_{{\FurukawaCommentFourth{\epsilon}}, 1} +c_{\FurukawaCommentFourth{\epsilon}}, u), {\FurukawaCommentFourth{\epsilon}} F_z ( b_{1, {\FurukawaCommentFourth{\epsilon}}} + b_{\FurukawaCommentFourth{\epsilon}}, c_{{\FurukawaCommentFourth{\epsilon}}, 1} +c_{\FurukawaCommentFourth{\epsilon}}, u) \right)$
		Set inductively
 		\begin{align*}
 			& a_{{\FurukawaCommentFourth{\epsilon}}, j + 1}
 			= a_{{\FurukawaCommentFourth{\epsilon}}, 1}
 			+ \mathcal{R}^u (F ( b_{1, {\FurukawaCommentFourth{\epsilon}}} + b_{{\FurukawaCommentFourth{\epsilon}} + j}, c_{{\FurukawaCommentFourth{\epsilon}}, 1} +c_{{\FurukawaCommentFourth{\epsilon}} + j}, u), 0), \\
 			& \pi_{{\FurukawaCommentFourth{\epsilon}}, j + 1}
 			= \mathcal{R}^p (F ( b_{1, {\FurukawaCommentFourth{\epsilon}}} + b_{{\FurukawaCommentFourth{\epsilon}} + j}, c_{{\FurukawaCommentFourth{\epsilon}}, 1} +c_{{\FurukawaCommentFourth{\epsilon}} + j}, u), 0) \FurukawaCommentFourth{,}
		\end{align*}
		for $j \geq 1$.
		Applying \FurukawaCommentFourth{Propositions} \ref{prop:0201}, \ref{prop:0202} and \FurukawaCommentSeventh{Corollary} \ref{cor_mr_estimate} to (\ref{eq_DE_modified_2}), we find
		\begin{align*}
			& \Vert
				\left(
					b_{{\FurukawaCommentFourth{\epsilon}} ,j + 1}, 
					{\FurukawaCommentFourth{\epsilon}} c_{{\FurukawaCommentFourth{\epsilon}} ,j + 1}
				\right)
			\Vert_{\mathbb{E}_1 (T, 2T)}
			+ \Vert
				\nabla_{\FurukawaCommentFourth{\epsilon}} \pi_{{\FurukawaCommentFourth{\epsilon}}, j + 1}
			\Vert_{\mathbb{E}_0 (T, 2T)} \notag \\
			& \leq C_1 T^\eta
				\Vert
					u
				\Vert_{\mathbb{E}_1 (T, 2T)}
				\Vert
					\left(
						b_{{\FurukawaCommentFourth{\epsilon}}, 1} +  b_{{\FurukawaCommentFourth{\epsilon}} ,j}, 
						{\FurukawaCommentFourth{\epsilon}} ( c_{{\FurukawaCommentFourth{\epsilon}}, 1} + c_{{\FurukawaCommentFourth{\epsilon}} ,j} )
					\right)
				\Vert_{\mathbb{E}_1 (T, 2T)} \\
			&	+ C_1 T^\eta \Vert
					\left(
						b_{{\FurukawaCommentFourth{\epsilon}}, 1} +  b_{{\FurukawaCommentFourth{\epsilon}} ,j}, 
						{\FurukawaCommentFourth{\epsilon}} ( c_{{\FurukawaCommentFourth{\epsilon}}, 1} + c_{{\FurukawaCommentFourth{\epsilon}} ,j} )
					\right)
				\Vert_{\mathbb{E}_1 (T, 2T)}^2 \\
			& + {\FurukawaCommentFourth{\epsilon}} C_1 T^\eta \left[
				\Vert
					u
				\Vert_{\mathbb{E}_1 (T, 2T)}
				+ \Vert
					u
				\Vert_{\mathbb{E}_1 (T, 2T)}^2
			\right] \notag \\
			& \leq C_1 T^\eta 
				\Vert
					\left(
						b_{{\FurukawaCommentFourth{\epsilon}} ,j}, 
						{\FurukawaCommentFourth{\epsilon}} c_{{\FurukawaCommentFourth{\epsilon}} ,j}
					\right)
				\Vert_{\mathbb{E}_1 (T, 2T)}^2 \\
				& + C_1 T^\eta \left( 
					\Vert
						u
					\Vert_{\mathbb{E}_1 (T, 2T)}
					+ 2	\Vert
						\left(
							b_{{\FurukawaCommentFourth{\epsilon}} ,1}, 
							{\FurukawaCommentFourth{\epsilon}} c_{{\FurukawaCommentFourth{\epsilon}} ,1}
						\right)
					\Vert_{\mathbb{E}_1 (T, 2T)}
				\right) \Vert
					\left(
						b_{{\FurukawaCommentFourth{\epsilon}} ,j}, 
						{\FurukawaCommentFourth{\epsilon}} c_{ {\FurukawaCommentFourth{\epsilon}} ,j}
					\right)
				\Vert_{\mathbb{E}_1 (T, 2T)} \\
				& + {\FurukawaCommentFourth{\epsilon}} C_1 T^\eta \left[ \Vert
					u
				\Vert_{\mathbb{E}_1 (T, 2T)} 
				\Vert
					\left(
						b_{{\FurukawaCommentFourth{\epsilon}} ,1}, 
						{\FurukawaCommentFourth{\epsilon}} c_{{\FurukawaCommentFourth{\epsilon}} ,1}
					\right)
				\Vert_{\mathbb{E}_1 (T, 2T)}
				+ \Vert
					\left(
						b_{{\FurukawaCommentFourth{\epsilon}} ,1}, 
						{\FurukawaCommentFourth{\epsilon}} c_{{\FurukawaCommentFourth{\epsilon}} ,1}
					\right)
				\Vert_{\mathbb{E}_1 (T, 2T)}^2
			\right] \\
			& + {\FurukawaCommentFourth{\epsilon}} C_1 T^\eta \left[
				\Vert
					u
				\Vert_{\mathbb{E}_1 (T, 2T)}
				+ \Vert
					u
				\Vert_{\mathbb{E}_1 (T, 2T)}^2
			\right] \notag \FurukawaCommentFourth{.}
		\end{align*}
		If we take ${\FurukawaCommentFourth{\epsilon}}$ so small that
		\begin{align*}
			\Vert
				a_{{\FurukawaCommentFourth{\epsilon}}, 1} 
			\Vert_{\mathbb{E}_1 (T, 2T)}
			\leq2 {\FurukawaCommentFourth{\epsilon}} C^\ast C_{tr} C_T
			\leq \frac{1}{8 C_1} \FurukawaCommentFourth{,}
		\end{align*}
		we have
		\begin{align*}
			& \Vert
				\left(
					b_{{\FurukawaCommentFourth{\epsilon}} ,j + 1}, 
					{\FurukawaCommentFourth{\epsilon}} c_{{\FurukawaCommentFourth{\epsilon}} ,j + 1}
				\right)
			\Vert_{\mathbb{E}_1 (T, 2T)}
			+ \Vert
				\nabla_{\FurukawaCommentFourth{\epsilon}} \pi_{{\FurukawaCommentFourth{\epsilon}}, j + 1}
			\Vert_{\mathbb{E}_0 (T, 2T)}  \\ 
			& \leq C_1 \Vert
				\left(
						b_{{\FurukawaCommentFourth{\epsilon}} ,j}, 
						{\FurukawaCommentFourth{\epsilon}} c_{{\FurukawaCommentFourth{\epsilon}} ,j}
					\right)
			\Vert_{\mathbb{E}_1 (T, 2T)}^2
			+ \frac{1}{2} \Vert
				\left(
						b_{{\FurukawaCommentFourth{\epsilon}} ,j}, 
						{\FurukawaCommentFourth{\epsilon}} c_{{\FurukawaCommentFourth{\epsilon}} ,j}
					\right)
			\Vert_{\mathbb{E}_1 (T, 2T)}
			+ {\FurukawaCommentFourth{\epsilon}} C^\ast C_T C_{tr}
			+ {\FurukawaCommentFourth{\epsilon}} C^\ast \notag \\
			& \leq C_1 \Vert
				\left(
						b_{{\FurukawaCommentFourth{\epsilon}} ,j}, 
						{\FurukawaCommentFourth{\epsilon}} c_{{\FurukawaCommentFourth{\epsilon}} ,j}
					\right)
			\Vert_{\mathbb{E}_1 (T, 2T)}^2
			+ \frac{1}{2} \Vert
				\left(
						b_{{\FurukawaCommentFourth{\epsilon}} ,j}, 
						{\FurukawaCommentFourth{\epsilon}} c_{{\FurukawaCommentFourth{\epsilon}} ,j}
					\right)
			\Vert_{\mathbb{E}_1 (T, 2T)}
			+ {\FurukawaCommentFourth{\epsilon}} C^\ast (1 + C_T C_{tr}). \notag
		\end{align*}
		Thus, we have by induction
		\begin{align*}
			\Vert
				\left(
						b_{{\FurukawaCommentFourth{\epsilon}} ,j}, 
						{\FurukawaCommentFourth{\epsilon}} c_{{\FurukawaCommentFourth{\epsilon}} ,j}
					\right)
			\Vert_{\mathbb{E}_1 (T, 2T)}
			+ \Vert
				\nabla_{\FurukawaCommentFourth{\epsilon}} \pi_{{\FurukawaCommentFourth{\epsilon}}, j}
			\Vert_{\mathbb{E}_0 (T, 2T)}
			\leq 2 {\FurukawaCommentFourth{\epsilon}} C^\ast (1 + C_T C_{tr})
		\end{align*}
		for all $j \geq 1$.
		Set
		\begin{align*}
			& \tilde{a}_{{\FurukawaCommentFourth{\epsilon}}, j} 
			= a_{{\FurukawaCommentFourth{\epsilon}}, j + 1} 
			- a_{{\FurukawaCommentFourth{\epsilon}}, j} \quad (j \geq 1), \quad 
			\tilde{a}_{{\FurukawaCommentFourth{\epsilon}}, 0} 
			= a_{{\FurukawaCommentFourth{\epsilon}}, 0}  \FurukawaCommentFourth{,} \\
			& \tilde{\pi}_{{\FurukawaCommentFourth{\epsilon}}, j} 
			= \pi_{{\FurukawaCommentFourth{\epsilon}}, j + 1} 
			- \pi_{{\FurukawaCommentFourth{\epsilon}}, j} \quad (j \geq 1), \quad
			\tilde{\pi}_{{\FurukawaCommentFourth{\epsilon}}, 0} 
			= \pi_{{\FurukawaCommentFourth{\epsilon}}, 0} \FurukawaCommentFourth{.}
		\end{align*}
		Applying \FurukawaCommentFourth{Propositions} \ref{prop:0201}, \ref{prop:0202} and \FurukawaCommentFifth{Corollary} \ref{cor_mr_estimate} to the equations that 
		\begin{align*}
			(\tilde{a}_{{\FurukawaCommentFourth{\epsilon}}, j + 1}, \tilde{\pi}_{{\FurukawaCommentFourth{\epsilon}}, j + 1})
		\end{align*} 
		satisfies, we find
		\begin{align*}
			&\Vert
				(
					\tilde{b}_{{\FurukawaCommentFourth{\epsilon}}, j + 1}, 
					{\FurukawaCommentFourth{\epsilon}} \tilde{c}_{{\FurukawaCommentFourth{\epsilon}}, j + 1}
				)
			\Vert_{\mathbb{E}_1 (T, 2 T)}
			+ \Vert
 				\nabla_{\FurukawaCommentFourth{\epsilon}} \tilde{\pi}_{{\FurukawaCommentFourth{\epsilon}}, j + 1}
			\Vert_{\mathbb{E}_0 (T, 2 T)} \\
			& \leq C_1 T^\eta \left(
				\Vert
					\left(
						b_{{\FurukawaCommentFourth{\epsilon}} ,j}, 
						{\FurukawaCommentFourth{\epsilon}} c_{{\FurukawaCommentFourth{\epsilon}} ,j}
					\right)
				\Vert_{\mathbb{E}_1 (T, 2 T)}
				+ \Vert
					\left(
						b_{{\FurukawaCommentFourth{\epsilon}} ,j + 1}, 
						{\FurukawaCommentFourth{\epsilon}} c_{{\FurukawaCommentFourth{\epsilon}} ,j + 1}
					\right)
				\Vert_{\mathbb{E}_1 (T, 2 T)}
				+ 2 \Vert
					u
				\Vert_{\mathbb{E}_1 (T, 2 T)}
			\right) \\
			& \quad \quad \times \Vert
				(
					\tilde{b}_{{\FurukawaCommentFourth{\epsilon}}, j }, 
					{\FurukawaCommentFourth{\epsilon}} \tilde{c}_{{\FurukawaCommentFourth{\epsilon}}, j },
				)
			\Vert_{\mathbb{E}_1 (T, 2 T)} \\
			& \leq \left[
				C_1 \left(
					\Vert
						\left(
						b_{{\FurukawaCommentFourth{\epsilon}} ,j}, 
						{\FurukawaCommentFourth{\epsilon}} c_{{\FurukawaCommentFourth{\epsilon}} ,j}
					\right)
					\Vert_{\mathbb{E}_1 (T, 2 T)}
					+ \Vert
						\left(
						b_{{\FurukawaCommentFourth{\epsilon}} ,j + 1}, 
						{\FurukawaCommentFourth{\epsilon}} c_{{\FurukawaCommentFourth{\epsilon}} ,j + 1}
					\right)
					\Vert_{\mathbb{E}_1 (T, 2 T)}
				\right)
				+ \frac{1}{2}
			\right] \\
			& \quad \quad  \times \Vert
				(
					\tilde{b}_{{\FurukawaCommentFourth{\epsilon}}, j}, 
					{\FurukawaCommentFourth{\epsilon}} \tilde{c}_{{\FurukawaCommentFourth{\epsilon}}, j},
				)
			\Vert_{\mathbb{E}_1 (T, 2 T)} \\
			& \leq \frac{3}{4} \Vert
				(
					\tilde{b}_{{\FurukawaCommentFourth{\epsilon}}, j}, 
					{\FurukawaCommentFourth{\epsilon}} \tilde{c}_{{\FurukawaCommentFourth{\epsilon}}, j },
				)
			\Vert_{\mathbb{E}_1 (T, 2 T)}.
		\end{align*}
		The last inequality holds if ${\FurukawaCommentFourth{\epsilon}}$ is sufficiently small.
		Thus, 
		\begin{align*}
			( a_{{\FurukawaCommentFourth{\epsilon}}}, \pi_{\FurukawaCommentFourth{\epsilon}}) 
			: = (
				\lim_{ j \rightarrow \infty} a_{{\FurukawaCommentFourth{\epsilon}}, j} , 
				\lim_{ j \rightarrow \infty} \pi_{{\FurukawaCommentFourth{\epsilon}}, j}
			) 
			= (
				\sum_{j = 0}\tilde{a}_{{\FurukawaCommentFourth{\epsilon}}, j}, 
				\sum_{j = 0} \tilde{\pi}_{{\FurukawaCommentFourth{\epsilon}}, j}
			)
		\end{align*} 
		exists and satisfies (\ref{eq_DE_modified_2}) such that
		\begin{align*}
			\Vert
				\left(
						b_{{\FurukawaCommentFourth{\epsilon}}}, 
						{\FurukawaCommentFourth{\epsilon}} c_{{\FurukawaCommentFourth{\epsilon}}}
					\right)
			\Vert_{\mathbb{E}_1 (T, 2T)}
			+ \Vert
				\nabla_{\FurukawaCommentFourth{\epsilon}} \pi_{{\FurukawaCommentFourth{\epsilon}}}
			\Vert_{\mathbb{E}_0 (T, 2T)}
			\leq 2 {\FurukawaCommentFourth{\epsilon}} C^\ast (1 + C_T C_{tr}).
		\end{align*}
		\FurukawaCommentFourth{The functions} $(U_{\FurukawaCommentFourth{\epsilon}}, P_{\FurukawaCommentFourth{\epsilon}})$ solves (\ref{eq:A0502}) on the time interval $[T, 2 T]$ with initial data $U_{\FurukawaCommentFourth{\epsilon}} (T)$ such that
		\begin{align*}
			& \Vert
 				\left(
						V_{{\FurukawaCommentFourth{\epsilon}}}
						, {\FurukawaCommentFourth{\epsilon}} W_{ {\FurukawaCommentFourth{\epsilon}} }
					\right)
 			\Vert_{\mathbb{E}_1 (T)}
 			+ \Vert
 				\nabla_{\FurukawaCommentFourth{\epsilon}} P_{\FurukawaCommentFourth{\epsilon}}
 			\Vert_{\mathbb{E}_0 (T)} \\
			& \leq \Vert
				\left(
					b_{{\FurukawaCommentFourth{\epsilon}}, 1}
					, {\FurukawaCommentFourth{\epsilon}} c_{{\FurukawaCommentFourth{\epsilon}}, 1}
				\right)
			\Vert_{\mathbb{E}_1 (T, 2 T)}
			+ \Vert
				\left(
					b_{{\FurukawaCommentFourth{\epsilon}}}
					, {\FurukawaCommentFourth{\epsilon}} c_{{\FurukawaCommentFourth{\epsilon}}}
				\right)
			\Vert_{\mathbb{E}_1 (T, 2 T)} \\
			& + \Vert
				\nabla_{\FurukawaCommentFourth{\epsilon}} \pi_{{\FurukawaCommentFourth{\epsilon}}, 1}
			\Vert_{\mathbb{E}_0 (T, 2 T)} 
			+ \Vert
				\nabla_{\FurukawaCommentFourth{\epsilon}} \pi_{\FurukawaCommentFourth{\epsilon}}
			\Vert_{\mathbb{E}_0 (T, 2 T)} \\
			& \leq C_T \Vert
				U_{\FurukawaCommentFourth{\epsilon}} (T)
			\Vert_{X_\gamma}
			+ 2 {\FurukawaCommentFourth{\epsilon}} C^\ast (1 +  C_{tr} C_T)
			\leq 2 {\FurukawaCommentFourth{\epsilon}} C^\ast ( 1 + 2 C_{tr} CT)  \FurukawaCommentFourth{.}
		\end{align*}
		By induction, the solution $(U_{\FurukawaCommentFourth{\epsilon}}, P_{\FurukawaCommentFourth{\epsilon}})$ constructed by the same way on the time interval $[m T , (m + 1) T]$ satisfies
		\begin{align*}
			& \Vert
 				\left(
						V_{{\FurukawaCommentFourth{\epsilon}}}
						, {\FurukawaCommentFourth{\epsilon}} W_{ {\FurukawaCommentFourth{\epsilon}} }
					\right)
 			\Vert_{\mathbb{E}_1 (m T, (m + 1) T)}
 			+ \Vert
 				\nabla_{\FurukawaCommentFourth{\epsilon}} P_{\FurukawaCommentFourth{\epsilon}}
 			\Vert_{\mathbb{E}_0 (m T, (m + 1) T)} \\
			& \leq 2 {\FurukawaCommentFourth{\epsilon}} C^\ast \left[
				1 + 3 C_T C_{tr} \left(
					1 + 3 C_T C_{tr} ( \cdots )
				\right)
			\right]
			=: 2 {\FurukawaCommentFourth{\epsilon}} \alpha_j \FurukawaCommentFourth{.}
		\end{align*}
		Since $\mathcal{T}$ is finite, this induction ends in finite steps.
		Thus we conclude
		\begin{align*}
		 \Vert 
		 	\left(
						V_{{\FurukawaCommentFourth{\epsilon}}}
						, {\FurukawaCommentFourth{\epsilon}} W_{ {\FurukawaCommentFourth{\epsilon}} }
					\right)
		 \Vert_{\mathbb{E}_1 (\mathcal{T})} 
		 + \Vert 
		 	\nabla_{\FurukawaCommentFourth{\epsilon}} P_{\FurukawaCommentFourth{\epsilon}}
		 \Vert_{\mathbb{E}_1  (\mathcal{T})}
		 \leq 2 {\FurukawaCommentFourth{\epsilon}} \sum_{1 \leq j \leq \FurukawaCommentFifth{N}} \alpha_j.
		\end{align*}
	\end{proof}
\FurukawaCommentThird{
	\section*{Acknowledgements}
	The authors are grateful to Professor Matthias Hieber and Professor Amru Hussein for helpful discussions and comments.
}


\begin{thebibliography}{99}
%
%


\bibitem{Abels2002}
      H. Abels, Boundedness of imaginary powers of the {S}tokes operator in an
  infinite layer, J. Evol. Equ., 2, 4, (2002), 439-- 457.


\bibitem{AzeradGuillen2001}
      P. Az\'{e}rad
      \FurukawaCommentSixth{and} F. Guill\'{e}n, Mathematical justification of the hydrostatic approximation in
  the primitive equations of geophysical fluid dynamics, SIAM J. Math. Anal., 33, 4, (2001), 847--859.



\bibitem{CaoTiti2007}
     C. Cao
      \FurukawaCommentSixth{and} E. S. Titi, Global well-posedness of the three-dimensional viscous primitive
  equations of large scale ocean and atmosphere dynamics, Ann. of Math. (2), 166, 1,  (2007), 245--267.



\bibitem{CheminDesjardinsGallagherGrenier2006}
      J. Y. Chemin, 
      B. Desjardins, 
      I. Gallagher, 
      \FurukawaCommentSixth{and} E. Grenier, Mathematical geophysics, Oxford Lecture Series in Mathematics and its Applications, The Clarendon Press, Oxford University Press, Oxford, 32, (2006).




\bibitem{DenkHieberPruss2003}
      R. Denk,
      M. Hieber, 
      \FurukawaCommentSixth{and} J. Pr\"{u}ss, {$\mathcal{R}$}-boundedness, {F}ourier multipliers and problems of
  elliptic and parabolic type, Mem. Amer. Math. Soc., 166, 788, (2003), viii+114.




\bibitem{DoreVenni1990}
      G. Dore
     \FurukawaCommentSixth{and} A. Venni,  Some results about complex powers of closed operators, J. Math. Anal. Appl., 149, 1, (1990), 124--136.
      

\bibitem{FarwigKozonoSohr2005}
      R. Farwig, 
      H. Kozono, 
      \FurukawaCommentSixth{and} H. Sohr,  An {$L^q$}-approach to {S}tokes and {N}avier-{S}tokes equations
  in general domains, Acta Math., 195, (2005), 21--53.



\bibitem{FarwigKozonoSohr2018}
      R. Farwig, 
      H. Kozono, \FurukawaCommentSixth{and} H. Sohr,  Stokes semigroups, strong, weak, and very weak solutions for
  general domains, Handbook of mathematical analysis in mechanics of viscous
  fluids, Springer, Cham, (2018), 419--459.



\bibitem{FujitaKato1964}
      H. Fujita
     \FurukawaCommentSixth{and}  T. Kato, {On the {N}avier-{S}tokes initial value problem. {I}}, Arch. Rational Mech. Anal., 16, (1964), 269--315.



\bibitem{FurukawaGigaHieberHusseinKashiwabaraWrona2018}
      K. Furukawa, 
      Y. Giga,
      M. Hieber, 
      A. Hussein, 
      T. Kashiwabara,
      \FurukawaCommentSixth{and} M. Wrona, Rigorous justification of the hydrostatic approximation for the
  primitive equations by scaled navier-stokes equations, arxiv, preprint, (2018).




\bibitem{Gallagher2018}
      I. Gallagher, Critical function spaces for the well-posedness of the
  {N}avier-{S}tokes initial value problem, Handbook of mathematical analysis in mechanics of viscous
  fluids, Springer, Cham, (2018), 647--685.



\bibitem{Giga1985}
    Y. Giga, Domains of fractional powers of the {S}tokes operator in {$L_r$}
  spaces,  Arch. Rational Mech. Anal., 89, 3, (1985), 251- 265.


\bibitem{GigaGriesHieberHusseinKashiwabara2017_analiticity}
      Y. Giga,
     M. Gries,
      M. Hieber,
      A. Hussein,
     \FurukawaCommentSixth{and} T. Kashiwabara,
Analyticity of solutions to the primitive equations,
     \FurukawaCommentFifth{Math. Nachr. 293(2), (2020), 284-304}.


\bibitem{GigaGriesHieberHusseinKashiwabara2017_H_infty}
      Y. Giga,
      M. Gries,
      M. Hieber,
     A. Hussein,
      \FurukawaCommentSixth{and} T. Kashiwabara,
       Bounded {$H^\infty$}-calculus for the hydrostatic {S}tokes
  operator on {$L^p$}-spaces and applications, Proc. Amer. Math. Soc., 145, 9, (2017), 3865--3876.



\bibitem{GigaMiyakawa1989}
      Y. Giga
      \FurukawaCommentSixth{and} T. Miyakawa,
      Navier-{S}tokes flow in {$\bold R^3$} with measures as initial
  vorticity and {M}orrey spaces, Comm. Partial Differential Equations, 14, 5, (1989), 577--618.



\bibitem{GigaSohr1991}
      Y. Giga,
      \FurukawaCommentSixth{and} H. Sohr,
       Abstract {$L^p$} estimates for the {C}auchy problem with
  applications to the {N}avier-{S}tokes equations in exterior domains, J. Funct. Anal., 102, 1, (1991), 72--94.


\bibitem{Grafakos2008}
    L. Grafakos,
 	Classical {F}ourier analysis, Second edition, Graduate Texts in Mathematics, Springer, New York, (2008).



\bibitem{GuillenMasmoudiRodriguez2001}
     F. Guill\'{e}n-Gonz\'{a}lez,
      N. Masmoudi, 
      \FurukawaCommentSixth{and} M. A. Rodr\'{i}guez-Bellido, 
       Anisotropic estimates and strong solutions of the primitive
  equations,
     Differential Integral Equations,
    14,  1, (2001),
       1381--1408.




\bibitem{HeckKimKozono2009}
      H. Heck,
      H. Kim,
      \FurukawaCommentSixth{and} H. Kozono, H,
    Stability of plane {C}ouette flows with respect to small periodic  perturbations,
    Nonlinear Anal.,
      71,
      9, (2009), 3739--3758.


\bibitem{HieberKashiwabara2016}
   M. Hieber \FurukawaCommentSixth{and T. Kashiwabara}, Global strong well-posedness of the three dimensional primitive
  equations in $l^p$-spaces, Archive Rational Mech. Anal., (2016).



\bibitem{Kato1984}
     T. Kato,
Strong ${{L}^p}$ solutions of the {N}avier-{S}tokes equations in
  $\mathbb{{R}}^m$ with applications to weak solutions,
    Math. {Z}, 187, (1984), 471--480,


\bibitem{Ladyzhenskaya1969}
     O. A. Ladyzhenskaya, 
      The mathematical theory of viscous incompressible flow,
      Second English edition, Mathematics and its
  Applications, Vol. 2,
   Gordon and Breach, Science Publishers, New York-London-Paris,
      1969.



\bibitem{Lemarie-Rieusset2016}
      \FurukawaCommentSixth{P. G. Lemari\'{e}-Rieusset},
       The {N}avier-{S}tokes problem in the 21st century,
	CRC Press, Boca Raton, FL, 2016.




\bibitem{Leray1934}
      J. Leray,
       Sur le mouvement d'un liquide visqueux emplissant l'espace,
Acta Math.,
      63,
      1,
      (1934),
       193--248.



\bibitem{LiTiti2017}
      J. Li \FurukawaCommentSixth{and} E. S. Titi,
       The primitive equations as the small aspect ratio limit of the  Navier-Stokes equations: rigorous justification of the hydrostatic
  approximation, arXiv, preprint 2017.




\bibitem{LionsTemamWangShou1992}
      J. L. Lions, R. Temam, \FurukawaCommentSixth{and} S. H. Wang, 
       New formulations of the primitive equations of atmosphere and
  applications,
Nonlinearity, 5, 2, (1992), 237--288.

\bibitem{Saito2015}
	H. Saito, On the $\mathcal{R}$-boundedness of solution operator families of the generalized Stokes resolvent problem in an infinite layer, Math. Meth. Appl. Sci., 38, (2015), 1888-1925. 

\bibitem{Solonnikov1964}
     V. A. Solonnikov, Estimates for solutions of a non-stationary linearized system of
  {N}avier-{S}tokes equations, Trudy Mat. Inst. Steklov.,
      70, (1964), 213--317.


\bibitem{Weis2001}
	L. Weis,
  Operator-valued {F}ourier multiplier theorems and maximal
              {$L_p$}-regularity,
 Math. Ann.,
  Mathematische Annalen,
    319, 4, 
     (2001), 735--758.





\end{thebibliography}



\end{document}